\documentclass[11pt]{amsart}
\pdfoutput=1

\usepackage[letterpaper,margin=1in]{geometry}

\usepackage[dvipsnames]{xcolor}
\usepackage{amssymb, mathrsfs}
\usepackage{mathtools}
\usepackage{scalerel}
\usepackage{stackrel}
\usepackage{bbm}
\usepackage[square]{natbib}
\usepackage[citecolor=PineGreen,colorlinks=true,linkcolor=RoyalBlue,urlcolor=BrickRed,pdfstartview=FitV,pagebackref=false]{hyperref}
\usepackage[export]{adjustbox}
\usepackage{caption}
\usepackage{subcaption}
\usepackage{esint}

\newtheorem{proposition}{Proposition}
\newtheorem{theorem}[proposition]{Theorem}
\newtheorem{lemma}[proposition]{Lemma}
\newtheorem{corollary}[proposition]{Corollary}

\theoremstyle{remark}

\theoremstyle{definition}

\newtheorem{assumption}[proposition]{Assumption}

\numberwithin{equation}{section}
\numberwithin{proposition}{section}
\numberwithin{table}{section}

\newcommand{\E}{\mathbb{E}}
\renewcommand{\a}{\mathbf{a}}
\renewcommand{\b}{\mathbf{b}}
\newcommand{\R}{\mathbb{R}}
\newcommand{\N}{\mathbb{N}}
\newcommand{\Q}{\mathbb{Q}}
\newcommand{\ZZ}{\mathbb{Z}}

\renewcommand{\tilde}{\widetilde}
\renewcommand{\epsilon}{\varepsilon}
\newcommand{\ahom}{\bar{\mathbf{a}}}
\newcommand{\cu}{\square}
\newcommand{\X}{\mathcal{X}}

\makeatletter
\newif\ifv@
\newif\ifh@
\newcommand{\negphantom}{\v@true\h@true\negph@nt}
\newcommand{\neghphantom}{\v@false\h@true\negph@nt}
\newcommand{\negph@nt}{\ifmmode\expandafter\mathpalette
	\expandafter\mathnegph@nt\else\expandafter\makenegph@nt\fi}
\newcommand{\makenegph@nt}[1]{%
	\setbox\z@\hbox{\color@begingroup#1\color@endgroup}\finnegph@nt}
\newcommand{\finnegph@nt}{%
	\setbox\tw@\null
	\ifv@ \ht\tw@\ht\z@\dp\tw@\dp\z@\fi
	\ifh@\wd\tw@-\wd\z@\fi\box\tw@}
\newcommand{\mathnegph@nt}[2]{%
	\setbox\z@\hbox{$\m@th #1{#2}$}\finnegph@nt}
\makeatother

\newcommand{\Hminusuls}[1]{\hat{\phantom{H}}\negphantom{H}\underline{H}^{#1}}
\newcommand{\Wminusul}[2]{\hat{\phantom{W}}\negphantom{W}\underline{W}^{#1,#2}}

\def\Xint#1{\mathchoice
{\XXint\displaystyle\textstyle{#1}}%
{\XXint\textstyle\scriptstyle{#1}}%
{\XXint\scriptstyle\scriptscriptstyle{#1}}%
{\XXint\scriptscriptstyle\scriptscriptstyle{#1}}%
\!\int}
\def\XXint#1#2#3{{\setbox0=\hbox{$#1{#2#3}{\int}$}
\vcenter{\hbox{$#2#3$}}\kern-.5\wd0}}

\def\fint{\Xint-}

\title[Quantitative Einstein relation]{Quantitative Einstein relation for reversible diffusions in a random environment}

\author{Ahmed Bou-Rabee}
\address{Department of Mathematics, University of Pennsylvania, Philadelphia, PA 19104}
\email{ahmedb@sas.upenn.edu}

\author{Ruizhe Xu}
\address{School of Mathematical Sciences, Zhejiang University, Hangzhou, China}
\email{12135028@zju.edu.cn}

\begin{document}

\begin{abstract}
The Einstein relation describes the response of a diffusing particle to a small constant external force. It states that, as the force tends to zero, the ratio of the limiting velocity to the force magnitude converges to the diffusivity matrix of the unforced particle, evaluated in the force direction.  \citet*{Einstein2012gantert} proved this identity for reversible diffusions in random environments. We prove a quantitative version, with an explicit quenched algebraic rate.
\end{abstract}

\maketitle

\section{Introduction}\label{sec.intro}

A particle diffusing in a disordered medium, subjected to a small constant force, eventually drifts at an asymptotic velocity. To leading order, this velocity is a linear function of the force; the proportionality matrix is called the \emph{mobility}. The \emph{Einstein relation} predicts that the mobility equals the diffusivity matrix of the unforced particle. The relation admits several mathematical formulations, depending on the order in which the small-force and long-time limits are taken (see Section~\ref{ss.intro-context} for further discussion). We work with the formulation of \citet*{Einstein2012gantert} for reversible diffusions in a random environment, and we prove a quantitative relation, showing that the mobility error is bounded by an explicit power of the force magnitude (Theorem~\ref{t.main}).

\begin{figure}[hb]
	\centering
	\includegraphics[width=0.9\textwidth]{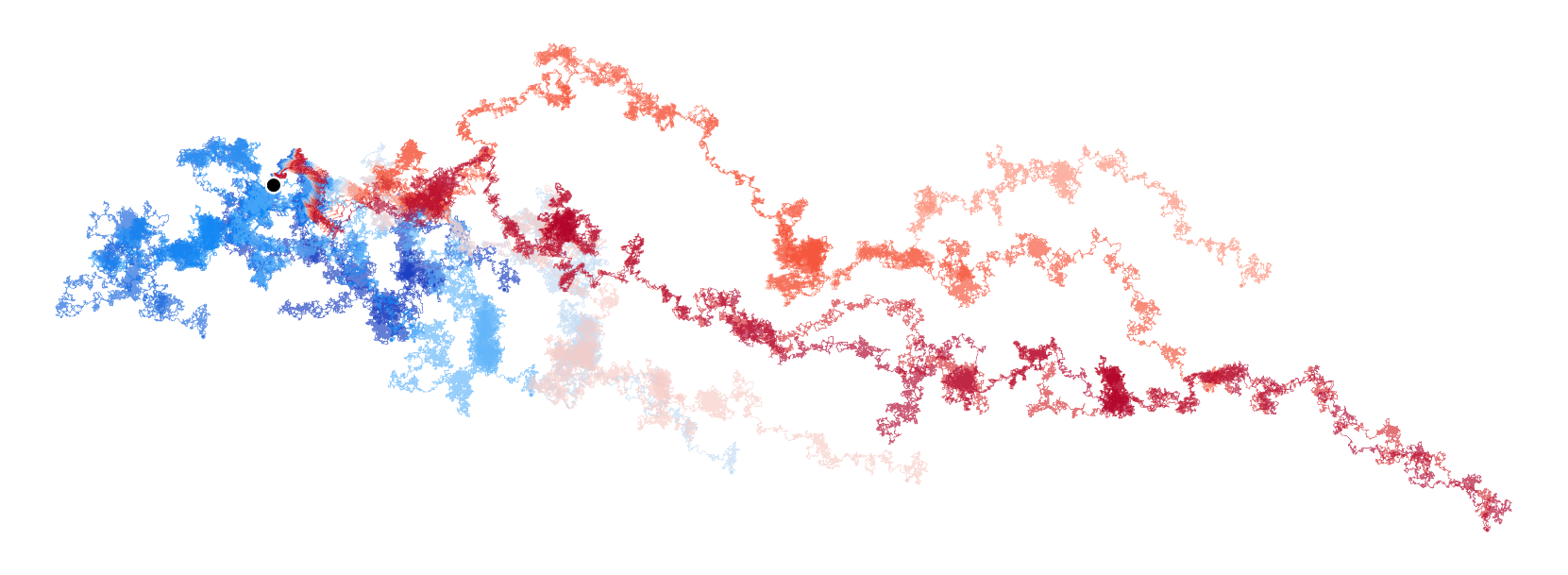}
	\caption{Sample paths of Brownian motions with drift~$\lambda e_1$, started from a common origin (black dot), with $\lambda$ varying continuously from $\lambda=0$ (blue) to a small positive value (red).}
	\label{fig.einstein-illustration}
\end{figure}

\subsection{Setting}\label{ss.intro-setting}

Let $\a$ be a uniformly elliptic, Lipschitz random matrix field on $\R^d$ and let $V$ be a bounded, Lipschitz random potential. Both are stationary with finite range of dependence; precise hypotheses are stated in Assumption~\ref{a.environment} below. Let $X$ be the reversible Markov diffusion on $\R^d$ with generator
\[
	\mathcal Lf=\tfrac12 e^{2V}\nabla\cdot\bigl(e^{-2V}\a\nabla f\bigr)\, .
\]
Under Assumption~\ref{a.environment}, $X$ satisfies the invariance principle of \citet*{DeMasiFerrariGoldsteinWick1989} (see also \citealt{KipnisVaradhan1986,varadhan1982,kozlov1985}): for $\Q$-a.e.\ environment, 
\[
	\epsilon X(\epsilon^{-2}t) \xrightarrow[\epsilon\downarrow 0]{} W_t\quad\text{in law},\qquad \E[W_t\otimes W_t]=\Sigma_X t\, ,
\]
for a deterministic effective diffusivity $\Sigma_X$.

Applying a constant force of magnitude $\lambda\in(0,1]$ in the direction $e_1=(1,0,\ldots,0)$ adds a drift term to the generator:
\begin{equation}\label{eq.intro-generator-tilted}
	\mathcal L^\lambda f\coloneqq \mathcal Lf+\lambda \a e_1\cdot\nabla f\, .
\end{equation}
The resulting biased diffusion $X^\lambda$, started at the origin, satisfies a quenched law of large numbers \citep*{ballistic2003shen,Einstein2012gantert}: for almost every environment,
\[
	\frac{X^\lambda(t)}{t}\xrightarrow[t\to\infty]{}\ell_X(\lambda)\quad\text{almost surely}\, ,
\]
for a deterministic effective velocity $\ell_X(\lambda)$. The choice of direction~$e_1$ is without loss of generality. The Einstein relation predicts
\begin{equation}\label{eq.intro-einstein}
	\lim_{\lambda\downarrow 0}\frac{\ell_X(\lambda)}{\lambda}=\Sigma_X e_1\, .
\end{equation}
\citet*{Einstein2012gantert} established \eqref{eq.intro-einstein} as a qualitative limit.

\subsection{Main result}\label{ss.intro-main}
Throughout, we assume:

\begin{assumption}\label{a.environment}
	Let $(\Omega,\mathcal F,\Q)$ be a probability space, let $\Lambda\ge 1$, and let $r_0\ge 1$ be an integer. Let $(\a,V)$ be a measurable random field on $\Omega\times\R^d$, with $\a\colon\Omega\times\R^d\to\R^{d\times d}_{\mathrm{sym},+}$ and $V\colon\Omega\times\R^d\to\R$. Suppose that $(\a,V)$ satisfies, for every $\omega\in\Omega$:
	\begin{enumerate}
		\item[(S1)] (\emph{Uniform ellipticity.}) $\Lambda^{-1}|\xi|^2\le\xi\cdot\a(x)\xi\le\Lambda|\xi|^2$ for every $x,\xi\in\R^d$.
		\item[(S2)] (\emph{Bounded potential.}) $|V(x)|\le\Lambda$ for every $x\in\R^d$.
		\item[(S3)] (\emph{Lipschitz regularity.}) Both $\a$ and $V$ are $\Lambda$-Lipschitz on $\R^d$, with $\a$ measured in the operator norm.
	\end{enumerate}
	In addition, the law of $(\a,V)$ under $\Q$ satisfies:
	\begin{enumerate}
		\item[(S4)] (\emph{Stationarity.}) For every $z\in\R^d$, $(\a(\cdot+z),V(\cdot+z))$ and $(\a,V)$ have the same joint law.
		\item[(S5)] (\emph{Finite range of dependence.}) For every pair of Borel sets $U_1,U_2\subseteq\R^d$ with $\mathrm{dist}(U_1,U_2)>r_0$, the $\sigma$-algebras of $(\a,V)|_{U_1}$ and $(\a,V)|_{U_2}$ are $\Q$-independent.
	\end{enumerate}
\end{assumption}

All deterministic constants will depend only on $(d,\Lambda,r_0)$ unless explicitly specified.

\begin{theorem}\label{t.main}
	Under Assumption~\ref{a.environment} with $d\ge 2$, $\Q$-almost surely,
	\begin{equation}\label{eq.intro-rate}
		\Bigl|\frac{\ell_X(\lambda)}{\lambda}-\Sigma_X e_1\Bigr|
		\le \lambda^{\beta^*}
		\qquad\text{for every }0<\lambda\le\lambda_0(\omega)\, ,
	\end{equation}
	where $\beta^*=\beta^*(d)\in(0,1/4)$ is a deterministic exponent and $\lambda_0\colon\Omega\to(0,1]$ is a random scale satisfying the tail bound
	\begin{equation}\label{eq.intro-rate-tail}
		\Q\bigl(\lambda_0\le\lambda\bigr)\le C\exp\bigl(-c (\log(1/\lambda))^\sigma\bigr)\qquad\text{for every }\lambda\in(0,1]\, ,
	\end{equation}
	with $C=C(d,\Lambda,r_0)<\infty$, $c=c(d,\Lambda,r_0)>0$, and an exponent $\sigma=\sigma(d,\Lambda,r_0)\in(0,1)$.
\end{theorem}

The exponent $\beta^*$ is not expected to be optimal: we conjecture the sharp rate is $\lambda$ in $d\ge 3$ and $\lambda\sqrt{\log(1/\lambda)}$ in $d=2$; see Subsection~\ref{ss.intro-sharpness} for further discussion. The one-dimensional setting has a separate literature with sharper, often explicit, results; see \citet{LamDepauw2016} and \citet{FaggionatoSalvi2019} for reversible one-dimensional random walks, and \citet*{FaggionatoGantertSalvi2019} and \citet*{GantertMeinersMueller2019} for related one-dimensional models.

\subsection{Proof strategy}\label{ss.intro-strategy}

The starting point is a balance of scales. Under a force of magnitude $\lambda$, the ballistic displacement $\lambda t$ and the diffusive fluctuation $\sqrt t$ are comparable precisely when $t\asymp\lambda^{-2}$, that is, on spatial scales of order $\lambda^{-1}$. This scale is called the \emph{critical length scale} and is the scale at which the proof operates. 

Our strategy is as follows. First, a time change absorbs the potential into the divergence-form coefficient, which makes the tilted operator amenable to homogenization. Second, on a cube of side slightly larger than $\lambda^{-1}$, we solve a resolvent equation whose solution, via the Feynman--Kac formula, encodes the mobility error as an exponentially weighted average of the path. Third, we expand this average into four contributions --- short-time, long-time, exit, and homogenization --- and estimate each separately. The short-time and exit terms are controlled directly. The long-time term is controlled by a law of large numbers uniform in $\lambda$ (Section~\ref{sec.renewal}), and the homogenization term by a quantitative estimate for the tilted resolvent (Section~\ref{sec.tilted-resolvent}). Optimizing the free parameters in the resulting bound yields the algebraic rate of Theorem~\ref{t.main}.

\subsubsection{Time-change reduction to divergence form}

	Multiplying $\mathcal L^\lambda$ by $e^{-2V}$ moves the potential into the divergence-form coefficient, leaving a stationary coefficient field $e^{-2V}\a$ and an explicit first-order tilted drift:
\[
	e^{-2V}\mathcal L^\lambda f  =  \tfrac12\nabla\cdot(e^{-2V}\a\nabla f)  +  \lambda e^{-2V}\a e_1\cdot\nabla f \, .
\]
The diffusion $Y^\lambda$ generated by $e^{-2V}\mathcal L^\lambda$ is a time change of $X^\lambda$. Set
\[
	A^\lambda(s) \coloneqq \int_0^s e^{-2V(Y^\lambda(u))} du \, .
\]
The quantity $A^\lambda(s)$ is the elapsed $X^\lambda$-clock at $Y^\lambda$-time $s$. The two processes are related pathwise by
\begin{equation}\label{eq.intro-time-change}
	X^\lambda(t) = Y^\lambda\circ(A^\lambda)^{-1}(t) \, . 
\end{equation}

\smallskip

Let $\ell(\lambda)\coloneqq\lim_{t\to\infty}Y^\lambda(t)/t$ and $\eta(\lambda)\coloneqq\lim_{t\to\infty}A^\lambda(t)/t$ be the asymptotic velocity and clock rate of $Y^\lambda$. Let $\ahom$ be the homogenized matrix of $-\nabla\cdot(e^{-2V}\a\nabla)$, the analogue of $\Sigma_X$ for $Y^0$. Combining \eqref{eq.intro-time-change} with these limits gives
\begin{equation}\label{eq.intro-time-change-identities}
	\ell_X(\lambda)  =  \eta(\lambda)^{-1} \ell(\lambda), \qquad \Sigma_X  =  \E[e^{-2V(0)}]^{-1} \ahom\, ,
\end{equation}
so Theorem~\ref{t.main} reduces to two intermediate estimates:
\begin{equation}\label{eq.intro-mobility-rate}
	\Bigl|\frac{\ell(\lambda)}{\lambda}-\ahom e_1\Bigr|  \le  C\lambda^{\beta^*}
\end{equation}
and
\[
	|\eta(\lambda)-\E[e^{-2V(0)}]|  \le   C\lambda^{\beta^*}\, .
\]

\subsubsection{Feynman--Kac decomposition at the critical scale}

We describe the mobility case~\eqref{eq.intro-mobility-rate}; the estimate of the other error term is similar. Choose an integer $m$ at the critical length scale, so that $3^m=\rho/\lambda$, with $\rho\ge 1$ a free parameter.

\smallskip

Fix a buffer~$h \geq 0$. On the triadic cube $\cu_{m+h}$ of side $3^{m+h}$, let $u_m\colon\cu_{m+h}\to\R^d$ solve the velocity resolvent equation
\begin{equation}\label{eq.intro-resolvent}
	(\lambda/3^m) u_m - e^{-2V}\mathcal L^\lambda u_m = \mathbf b \quad\text{in }\cu_{m+h}, \qquad u_m = 3^m \ahom e_1 \quad\text{on }\partial\cu_{m+h}\, ,
\end{equation}
	where $\mathbf b\coloneqq\tfrac12\nabla\cdot(e^{-2V}\a)+\lambda e^{-2V}\a e_1$ is the It\^o drift of $Y^\lambda$. Replacing $e^{-2V}\a$ by $\ahom$ in \eqref{eq.intro-resolvent} gives the homogenized tilted equation. Since $\ahom$ is constant, the drift term $\lambda\ahom e_1\cdot\nabla\bar u$ and the divergence $\tfrac12\nabla\cdot(\ahom\nabla\bar u)$ both vanish for constant $\bar u$, so $\bar u\equiv 3^m\ahom e_1$ solves
	\[
		(\lambda/3^m)\bar u=\lambda\ahom e_1
		\qquad\text{in }\cu_{m+h},\qquad \bar u=3^m\ahom e_1\quad\text{on }\partial\cu_{m+h}\, .
	\]
	The deviation $3^{-m}u_m(0)-\ahom e_1$ measures the homogenization error at scale $3^m$.

\smallskip

	Let $\tau_{m,h}\coloneqq\inf\{t\ge 0:Y^\lambda(t)\notin\cu_{m+h}\}$ be the exit time. The Feynman--Kac formula for $Y^\lambda$ started at the origin gives
\begin{equation}\label{eq.intro-FK}
	u_m(0)  =  \frac{\lambda}{3^m} E_0^{\lambda,\omega} \biggl[\int_0^{\tau_{m,h}}   e^{-(\lambda/3^m)t} Y^\lambda(t) dt\biggr]  +  E_0^{\lambda,\omega} \Bigl[e^{-(\lambda/3^m)\tau_{m,h}}\bigl(Y^\lambda(\tau_{m,h})+3^m \ahom e_1\bigr)\Bigr]\, .
\end{equation}
At $3^m\sim\lambda^{-1}$, the integral concentrates on the diffusive timescale $t\sim\lambda^{-2}$, and substituting the law-of-large-numbers approximation $Y^\lambda(t)\approx\ell(\lambda)t$ gives $u_m(0)/3^m\approx\ell(\lambda)/\lambda$. The mobility error then decomposes as
\begin{equation}\label{eq.intro-velocity-decomp}
	\frac{\ell(\lambda)}{\lambda}-\ahom e_1
	 = \underbrace{\Bigl(\frac{\ell(\lambda)}{\lambda}-\frac{u_m(0)}{3^m}\Bigr)}_{\text{time-average error}}
	 + \underbrace{\Bigl(\frac{u_m(0)}{3^m}-\ahom e_1\Bigr)}_{\text{homogenization error}}\, .
\end{equation}

\subsubsection{The four error contributions}

The time-average error in~\eqref{eq.intro-velocity-decomp} splits into three terms: a short-time term where the law-of-large-numbers approximation $Y^\lambda(t)\approx\ell(\lambda)t$ fails; a long-time term controlled by the uniform-in-$\lambda$ quenched LLN of Corollary~\ref{cor.uniform-renewal}; and a boundary term of order $e^{-c 3^h}$ controlled by an exit-time bound (Proposition~\ref{pro.exit}). The homogenization error in~\eqref{eq.intro-velocity-decomp} is the fourth, controlled by a resolvent homogenization estimate (Proposition~\ref{pro.velocity-resolvent-infty}). With $\rho\coloneqq\lambda 3^m\ge 1$, these combine to give
\[
	\Bigl|\frac{\ell(\lambda)}{\lambda}-\ahom e_1\Bigr|
	 \le \underbrace{C\rho^{-2}\lambda^{-2\epsilon}}_{\text{short-time}}
	+\underbrace{C\rho^{-1/2+\theta}}_{\text{LLN rate}}
	+\underbrace{Ce^{-c 3^h}}_{\text{boundary}}
	+\underbrace{C\rho^{1/2} 3^{(d/2+1-\beta_h)h-\delta m}}_{\text{homogenization}}\, ,
\]
where $\beta_h>0$ is a homogenization exponent supplied by \citet{AK_book} (Proposition~\ref{pro.ak-inputs}), $\delta>0$ is the exponent of our homogenization result, Proposition~\ref{pro.velocity-resolvent-infty} below, and $\theta,\epsilon>0$ are small parameters. Optimizing the choice of parameters then gives an algebraic rate. 

\smallskip

The uniform-in-$\lambda$ LLN required by this argument (Corollary~\ref{cor.uniform-renewal}) is established by a conditional renewal identity at regeneration times (Lemma~\ref{lem.conditional-renewal-tau-k}), using the regeneration structure of \citet*{ballistic2003shen} and the $\lambda$-dependent adaptation in \citet[Section~5]{Einstein2012gantert}.

\subsection{Heuristic for the conjectural rate}\label{ss.intro-sharpness}

We conjecture the sharp dimension-dependent upper bound
\begin{equation}\label{eq.sharp}
	\Bigl|\frac{\ell_X(\lambda)}{\lambda}-\Sigma_X e_1\Bigr|
	 \lesssim
	\begin{cases}
		\lambda, & d\ge 3\, ,\\[1mm]
		\lambda\sqrt{\log(1/\lambda)}, & d=2\, .
	\end{cases}
\end{equation}
By~\eqref{eq.intro-time-change-identities}, it suffices to bound the velocity response $\ell(\lambda)/\lambda - \ahom e_1$ and the clock error $\eta(\lambda) - \E[e^{-2V(0)}]$ at the same rates. We sketch the heuristic for the velocity. The decomposition~\eqref{eq.intro-velocity-decomp} splits the error into a time-average contribution and a homogenization contribution; we conjecture an optimal rate for each.

\smallskip

Refining the Feynman--Kac decomposition~\eqref{eq.intro-FK}, we decompose the expected displacement into a ballistic part $\ell(\lambda) t$ and a remainder $R^\omega_\lambda(t)$:
\[
	E_0^{\lambda,\omega}[Y^\lambda(t)] = \ell(\lambda) t + R^\omega_\lambda(t)\, .
\]
Substituting this into~\eqref{eq.intro-FK}, the ballistic part contributes $3^m\ell(\lambda)/\lambda$ to $u_m(0)$, while the boundary term gives an exit-tail $O(e^{-c 3^h})$. Dividing by $3^m$ yields
\[
	\frac{u_m(0)}{3^m} \approx \frac{\ell(\lambda)}{\lambda} + \frac{\bar R^\omega_\lambda}{3^m} + O(e^{-c 3^h})\, ,
\]
where
\[
	\bar R^\omega_\lambda \coloneqq \frac{\lambda}{3^m}\int_0^\infty e^{-(\lambda/3^m)t} R^\omega_\lambda(t) dt
\]
is the time-average of $R^\omega_\lambda(t)$ against an exponential kernel of mean $3^m/\lambda$. At the critical length scale $3^m\sim\lambda^{-1}$, this mean is the diffusive timescale $\lambda^{-2}$.

\smallskip

At $\lambda=0$, the homogenization corrector $\chi$ \citep{kozlov1985} --- the sublinear vector field with $\chi(0)=0$ and $\nabla\cdot(e^{-2V}\a\nabla(x_i+\chi_i))=0$ for each $i$ --- makes $Y^0+\chi(Y^0)$ a martingale, so
\[
	E_0^{0,\omega}[Y^0(t)] = -E_0^{0,\omega}[\chi(Y^0(t))]\, .
\]
	Heuristically, at time $t\sim 3^{2m}$ the heat kernel for $Y^0$ averages mostly over $\cu_m$, suggesting that $|R^\omega_0(t)|$ should be controlled by the corrector's amplitude on $\cu_m$:
\[
	A_m \coloneqq \max_{1\le i\le d} \|\chi_i\|_{\underline L^2(\cu_m)},\qquad A_m \lesssim \begin{cases} 1, & d\ge 3\, ,\\[1mm] \sqrt{m}, & d=2\, ,\end{cases}
\]
	where the inequality on the right follows from~\citet[Theorem~4.1 and Proposition~4.27]{AKM2019book}. This gives the heuristic estimate $|R^\omega_0(t)| \lesssim A_m$.

\smallskip

We conjecture that this corrector control persists for the biased process at the critical length scale $3^m\asymp\lambda^{-1}$:
\begin{equation}\label{eq.remainder-conj}
	\bar R^\omega_\lambda = O(A_m)\quad\text{uniformly as }\lambda\to 0\, .
\end{equation}
The existing argument, which relies on regeneration alone, gives $\bar R^\omega_\lambda = O(\lambda^{-1})$: blocks have duration $\asymp\lambda^{-2}$ and displacement $\asymp\lambda^{-1}$, and the exponential kernel of mean $\lambda^{-2}$ averages over $O(1)$ blocks. The conjecture \eqref{eq.remainder-conj} requires sharper cancellation, analogous to the corrector mechanism of the symmetric case above. Granting \eqref{eq.remainder-conj}, the remainder term $\bar R^\omega_\lambda/3^m$ in the schematic expansion has size $\lesssim A_m/3^m$. Combined with the exit-tail, this gives the conjectural time-average rate
\[
	\Bigl|\frac{\ell(\lambda)}{\lambda} - \frac{u_m(0)}{3^m}\Bigr| \lesssim \frac{A_m}{3^m} + e^{-c 3^h} \asymp \lambda A_m\, ,
\]
equal to $\lambda$ in $d\ge 3$ and $\lambda\sqrt{\log(1/\lambda)}$ in $d=2$.

\smallskip

We also expect the homogenization error to be controlled by the interior slope of the corrector across $\cu_m$: an oscillation of size $A_m$ across a cube of side $3^m$ produces a slope $A_m/3^m$, so
\begin{equation}\label{eq.flux-discrepancy}
	\Bigl|\frac{u_m(0)}{3^m} - \ahom e_1\Bigr| \lesssim \frac{A_m}{3^m}\, .
\end{equation}
Adding the two contributions gives \eqref{eq.sharp}.

\smallskip

Making this heuristic rigorous requires resolving two problems, one for each contribution.

For the time-average contribution, the bound $|R^\omega_0(t)|\lesssim A_m$ comes from the martingale $Y^0+\chi$, which exists because $\chi$ corrects the self-adjoint generator $\mathcal L^0$; the biased generator $\mathcal L^\lambda$ has no analogous corrector with quantitative control of its amplitude. The steady-state approach gives the leading small-$\lambda$ behavior of the weak steady state $\nu_\lambda$: for centered observables $f$,
\[
	\nu_\lambda(f)=\lambda\bar\Gamma(f)+o(\lambda),
	\qquad \bar\Gamma(f)=-2\int\chi f d\Q\, .
\]
This is proved in \citet[Theorem~1.5, Corollary~5.2, and eq.~(1.7)]{Steady2018mathieu} when a stationary corrector exists, for instance under finite range in $d\ge3$; \citet*{GantertGuoNagel2017} prove the analogous expansion on $\ZZ^d$ in $d\ge3$. The conjecture~\eqref{eq.remainder-conj} is different: it is a finite-time statement at the critical timescale $t\sim\lambda^{-2}$. Even a quantitative second-order bound such as $|\nu_\lambda(b)-\lambda\bar\Gamma(b)|\le C\lambda^2 A_m$ for the velocity observable $b$ would only control the infinite-time velocity component. One would also need a critical-scale relaxation estimate from the initial environment to $\nu_\lambda$. In $d\ge3$, this second-order input is a $C^{1,1}$-type regularity statement for $\nu_\lambda$ at $\lambda=0$; in $d=2$, the same formal argument predicts the additional $\sqrt{\log(1/\lambda)}$ loss.

	For the homogenization contribution, our proof uses the sharp homogenization estimates of \citet{AK_book} but does not recover the conjectural rate $A_m/3^m$ of~\eqref{eq.flux-discrepancy}. Two passages introduce slack. First, the resolvent equation carries a first-order drift and a spectral parameter that the corrector equation does not, and these inflate the energy bound. Second, we need the value at a single point, while the energy estimate gives an $L^2$ average over the full cube. To extract a pointwise bound, we apply a deterministic interior $L^\infty$ estimate, which requires the elliptic operator to be uniformly elliptic. The tilted operator $-\nabla\cdot(e^{2\lambda x_1}\a\nabla)$ has ellipticity constants that depend on $e^{2\lambda x_1}$, which ranges over $e^{\pm\lambda 3^{m+h}}$ on the full cube $\cu_{m+h}$ and is unbounded as $m\to\infty$. We therefore restrict to a concentric subcube $\cu_{m_1}$ small enough that $e^{2\lambda x_1}$ is bounded by an absolute constant. The corresponding volume ratio is $3^{d(m-m_1+h)/2}$: the loss from $m-m_1$ is absorbed by the homogenization exponent, while the $h$-dependent loss is controlled in the final parameter choice.

\subsection{Background}\label{ss.intro-context}

The Einstein relation is a prototype of \emph{fluctuation--dissipation}: the principle that the equilibrium fluctuations of a system determine its response to a small external perturbation. Here the fluctuation is the diffusive spreading of the unforced particle, and the response is the drift induced by the force. In its classical form the relation reads
\begin{equation*}
  D = \mu\, k_B T \, ,
\end{equation*}
where $D$ is the diffusion constant, $\mu$ the mobility, $T$ the temperature, and $k_B$ the Boltzmann constant.

For a test particle in a stationary random environment, the Einstein relation has been approached in three ways, differing in how the small-force and long-time limits are ordered. \citet{LebowitzRost1994} couple the two limits, scaling the force together with a diffusive rescaling of space and time; this sidesteps the need to interpret the mobility as a genuine derivative, at the price of a weaker conclusion. The remaining two approaches both send time to infinity at a fixed force, but extract different objects from that limit. \citet{Steady2018mathieu} and \citet*{GantertGuoNagel2017} encode it in an invariant measure --- a \emph{steady state} --- for the environment viewed from the biased particle, and recover the Einstein relation from the small-force expansion of that measure. We follow the third route, due to \citet*{Einstein2012gantert}: for each fixed $\lambda$ the long-time limit yields the velocity $\ell_X(\lambda)$, and the small-force limit of $\ell_X(\lambda)/\lambda$ is taken afterward.

A parallel theory for tagged particles in interacting particle systems was initiated by \citet{LebowitzSpohn1982}, with a general framework given in \citet*{FerrariGoldsteinLebowitz1985}.

The general principle of fluctuation--dissipation, expressing transport coefficients as equilibrium time correlations, was developed in \citet{Nyquist1928,Onsager1931,Onsager1931II,CallenWelton1951}; a rigorous Markov-process formulation appears in \citet{DemboDeuschel2010}. The invariant-measure approach described above embeds the Einstein relation in this framework through the weak steady state $\nu_\lambda$ of the environment seen from the biased particle, with early constructions in related settings due to \citet{KomorowskiKrupa2004}. In our continuous reversible setting, \citet{Steady2018mathieu} construct $\nu_\lambda$, prove the Lipschitz bound $|\nu_\lambda(f)|\le C\lambda\|f\|_{H^{-1}_\infty}$, and identify $\partial_\lambda\nu_\lambda|_{\lambda=0}$ in terms of the corrector. On $\ZZ^d$ in $d\ge 3$, \citet*{GantertGuoNagel2017} obtain the analogous first-order expansion of $\nu_\lambda$, building on the variance-decay estimates of \citet{vardecay} and \citet{deBuyerMourrat2015}. A different perturbative expansion is obtained by \citet*{CamposRamirez2017} in the low-disorder regime on $\ZZ^d$, where the small parameter is the disorder rather than the bias $\lambda$.

The Einstein relation has been studied in many related discrete models. On $\ZZ^d$ in $d\ge 3$, qualitative versions are known for nearest-neighbor walks with two-valued random jump rates \citep{KomorowskiOllaRWRE2005}, for balanced i.i.d.\ walks under a Kalikow-type ballisticity condition and for Sznitman-ballistic i.i.d.\ walks \citep{Guo2016}, and for tagged tracers in time-mixing environments \citep{KomorowskiOlla2005}. One-dimensional models admit separate techniques: reversible walks in ergodic conductances \citep{LamDepauw2016}, random walks in random environment \citep{FaggionatoSalvi2019}, biased walks on a one-dimensional percolation model \citep*{GantertMeinersMueller2019}, Mott variable-range hopping \citep*{FaggionatoGantertSalvi2019}, and the complex (frequency-dependent) mobility matrix for the random conductance model \citep{FaggionatoSalvi2025}.

On trees, the relation has been established for biased walks on Galton--Watson trees with deep traps \citep*{BenArousHuOllaZeitouni2013}. For interacting particle systems, \citet{Loulakis2002} proved the tagged-particle Einstein relation under a constant external force for symmetric simple exclusion in $d\ge 3$, building on \citet{KipnisVaradhan1986}; in one dimension under weak asymmetry, the driven tracer was treated earlier by \citet*{LandimOllaVolchan1998}.

Reversibility of the process is an essential hypothesis: for mean-zero asymmetric simple exclusion in $d\ge 3$, \citet{Loulakis2005} showed that the mobility need not coincide with the self-diffusion matrix.

Theorem~\ref{t.main} gives, to our knowledge, the first quantitative rate in $\lambda$ for the Einstein relation under a constant external force, in dimensions $d\ge 2$. The rate is quenched: it holds for $\Q$-a.e.\ environment, on a random window $\lambda\in(0,\lambda_0(\omega)]$ whose threshold has the tail~\eqref{eq.intro-rate-tail}. The physics conjecture~\eqref{eq.sharp} is quenched in the same sense.

Unlike previous work, our proof relies on the sharp quantitative homogenization estimates of \citet{AK_book, AK_HC, AK_cg_theory}. The renewal estimates of Section~\ref{sec.renewal}, by contrast, are fairly straightforward consequences of the regeneration structure established by \citet*{ballistic2003shen} and \citet*{Einstein2012gantert}. These quantitative homogenization results build on a series of works in quantitative homogenization, initiated by \citet{armstrong2016mesoscopic,armstrong2017additive,ArmstrongSmart2016} and separately \citet{GloriaOtto2011,gloria2020regularity}. The Armstrong--Kuusi--Mourrat--Smart quantitative homogenization theory has been applied to many other problems, including the supercritical percolation cluster~\citep*{armstrong2018elliptic,dario2021corrector,dario2021quantitative,gu2022efficient,bou2025rigidity,gu2025coupling}; the critical long-range random conductance model~\citep*{bourabee2026critical}; the $\nabla\phi$ interface model~\citep*{armstrong2022surface}; the closely related Villain rotator~\citep*{dario2024villain}; gradient interface dynamics~\citep*{armstrong2024hydrodynamic}; interacting particle systems~\citep*{GGMN,giunti2022quantitative,funaki2026hydrodynamic}; and random drifts~\citep*{ABKL_clt,algsuperdiff}.
  
\subsection{Outline of the paper}\label{ss.intro-outline}
Section~\ref{sec.notation} fixes notation and applies a time change that reduces the tilted operator to divergence form. Section~\ref{sec.coarse-graining-inputs} then states the quantitative homogenization estimates we use, essentially from \citet*{AK_book}. The main analytic contribution is a localized homogenization estimate for the tilted finite-volume resolvent at the critical length scale (Section~\ref{sec.tilted-resolvent}). Section~\ref{sec.renewal} upgrades the regeneration estimates of \citet*{ballistic2003shen} and \citet*{Einstein2012gantert} to a quenched law of large numbers uniform in $\lambda$. Section~\ref{sec.einstein-rate} combines these via a Feynman--Kac decomposition to prove Theorem~\ref{t.main}. Section~\ref{sec.shen-renewal-k} extends the Shen--GMP renewal structure to the form needed in Section~\ref{sec.renewal}.   

\subsection*{Acknowledgments}

We thank Scott Armstrong for suggesting this problem and for helpful discussions.

\section{Notation and preliminaries}\label{sec.notation}

\subsection{Notation}\label{ss.notation}

	We write $\N\coloneqq\{0,1,2,\ldots\}$ for the nonnegative integers.

	\subsubsection{Triadic cubes and volume-normalized integrals}

	For every $m\in\ZZ$, the triadic cube of side length $3^m$ centered at the origin is
	\[
		\cu_m\coloneqq \Bigl(-\tfrac{3^m}{2},\tfrac{3^m}{2}\Bigr)^d\subset\R^d\, .
	\]
	For a bounded domain $U\subseteq\R^d$, we use volume-normalized integrals and $L^p$ norms,
	\[
	(f)_U \coloneqq  \fint_U f(x) dx = \frac{1}{|U|}\int_U f(x) dx, \qquad \|f\|_{\underline{L}^p(U)} \coloneqq  \Bigl(\fint_U |f(x)|^p dx\Bigr)^{1/p}\, .
	\]

	\subsubsection{Volume-normalized Sobolev norms}

	For $s\in(0,1)$ and $p\in[1,\infty)$, the volume-normalized Sobolev seminorm is
	\begin{equation}\label{def.semiwsp}
		[f]_{\underline{W}^{s,p}(U)} \coloneqq  \Bigl(\fint_U \int_U\frac{|f(x)-f(y)|^p}{|x-y|^{d+sp}} dy dx\Bigr)^{1/p}\, .
	\end{equation}
	Here only the outer integral is volume-normalized.
	The full norm is
	\[
		\|f\|_{\underline{W}^{s,p}(U)}\coloneqq  \Bigl(|U|^{-sp/d}\|f\|_{\underline{L}^p(U)}^p+ [f]_{\underline{W}^{s,p}(U)}^p\Bigr)^{1/p}\, .
	\]
	For $s\in(0,1/2]$ and $p\in(1,\infty)$ with H\"older conjugate $p'=p/(p-1)$, define the volume-normalized negative dual norm by
	\begin{equation}\label{def.dual-Wsp}
		\|f\|_{\Wminusul{-s}{p'}(U)}\coloneqq\sup\Bigl\{\langle f,g\rangle_U:g\in\underline{W}^{s,p}(U),\ \|g\|_{\underline{W}^{s,p}(U)}\le 1\Bigr\}\, ,
	\end{equation}
	where $\langle f,g\rangle_U$ denotes the volume-normalized dual pairing, equal to $\fint_U fg$ when $f$ is a function. 
	Set $\underline{H}^s\coloneqq\underline{W}^{s,2}$ and $\Hminusuls{-s}\coloneqq\Wminusul{-s}{2}$. For $\R^d$-valued functions, norms are taken componentwise: if $f\colon U\to\R^d$, then
	\[
		\|f\|^2_{\underline{H}^s(U)}\coloneqq\sum_{i=1}^d\|f_i\|^2_{\underline{H}^s(U)},\qquad
		\|f\|^2_{\Hminusuls{-s}(U)}\coloneqq\sum_{i=1}^d\|f_i\|^2_{\Hminusuls{-s}(U)}\, .
	\]
	For matrix fields $F\colon U\to\R^{d\times d}$, we use the Frobenius component convention:
	\[
		\|F\|^2_{\underline{H}^s(U)}\coloneqq\sum_{i,j=1}^d\|F_{ij}\|^2_{\underline{H}^s(U)}\, ,
	\]
	and analogously for $\|F\|_{\Hminusuls{-s}(U)}$.

	For $f$ defined on $\cu_M$, set
	\[
		\tilde f(\tilde x)\coloneqq f(3^M\tilde x)\qquad(\tilde x\in \cu_0)\, .
	\]
	The same definition applies componentwise to $\R^d$- and $\R^{d\times d}$-valued $f$. Since $|\cu_0|=1$, the volume-normalized norms on $\cu_0$ coincide with the ordinary unnormalized ones, and we drop the underline in the identities that follow. We write $H^s(\cu_0)\coloneqq W^{s,2}(\cu_0)$ and $\hat H^{-s}(\cu_0)\coloneqq (H^s(\cu_0))^*$. A direct computation from \eqref{def.semiwsp}--\eqref{def.dual-Wsp} gives, for $s\in(0,1)$,
	\begin{equation}\label{eq.pullback-identities}
	\begin{aligned}
		\|f\|_{\underline L^2(\cu_M)}&=\|\tilde f\|_{L^2(\cu_0)}\, ,\\
		[f]_{\underline H^s(\cu_M)}&=3^{-sM} [\tilde f]_{H^s(\cu_0)}\, ,\\
		\|f\|_{\underline H^s(\cu_M)}&=3^{-sM} \|\tilde f\|_{H^s(\cu_0)}\, .
	\end{aligned}
	\end{equation}
	Additionally, for $s\in(0,1/2]$,
	\begin{equation}\label{eq.pullback-identities-dual}
		\|f\|_{\Hminusuls{-s}(\cu_M)}=3^{sM}\|\tilde f\|_{\hat H^{-s}(\cu_0)}\, .
	\end{equation}
	The gradient scales as $\|\nabla\tilde f\|_{L^2(\cu_0)}=3^M\|\nabla f\|_{\underline L^2(\cu_M)}$.

	\subsection{The reversible diffusion and its tilt}\label{ss.setting}

	This subsection gives the SDE form of the tilted diffusion and the probabilistic notation. After the time change in Subsection~\ref{ss.reduction}, the notation is simplified.

	The tilted generator \eqref{eq.intro-generator-tilted} admits the It\^o decomposition
	\[
		\mathcal L^\lambda f=\tfrac12\sum_{i,j=1}^d\a_{ij}\partial_{ij}f+\b^\lambda\cdot\nabla f\, ,
	\]
	with diffusion coefficient $\a$ and It\^o drift
	\[
		\b^\lambda\coloneqq\tfrac12\nabla\cdot\a-\a\nabla V+\lambda \a e_1,\qquad (\nabla\cdot\a)_j\coloneqq \sum_{i=1}^d\partial_i\a_{ij}\, ,
	\]
	where $\nabla\cdot\a$ and $\nabla V$ are taken in the a.e.\ sense. The corresponding SDE is
	\[
		dX^\lambda(t)=\b^\lambda(X^\lambda(t)) dt+\a^{1/2}(X^\lambda(t)) dW(t),\qquad X^\lambda(0)=x\, ,
	\]
	where $\a^{1/2}$ is the symmetric square root of $\a$. The Lipschitz assumptions give a unique-in-law solution. We write $P^{\lambda,\omega}_x,E^{\lambda,\omega}_x$ for the quenched law and expectation and $\mathbb E^\lambda_x\coloneqq \int E^{\lambda,\omega}_x d\Q(\omega)$ for the annealed expectation.

	\subsection{The time change}\label{ss.reduction}

	The potential $V$ enters \eqref{eq.intro-generator-tilted} both as a coefficient and through the equilibrium measure $e^{-2V} dx$ of the unperturbed operator. A time change absorbs the potential into the divergence-form coefficient, leaving it visible only through the clock. Write the generator as
\[
	\mathcal L^\lambda f=\tfrac12 e^{2V}\nabla\cdot(e^{-2V}\a\nabla f)+\lambda \a e_1\cdot\nabla f\, .
\]
	Multiplying by $e^{-2V}$ removes the prefactor in front of the divergence:
	\begin{equation}\label{eq.reduced-generator}
	\begin{aligned}
		e^{-2V} \mathcal L^\lambda f
		&=\tfrac12\nabla\cdot(e^{-2V}\a\nabla f)+\lambda e^{-2V}\a e_1\cdot\nabla f\\
		&=\tfrac12 e^{-2\lambda x_1}\nabla\cdot\bigl(e^{2\lambda x_1}e^{-2V}\a\nabla f\bigr)\, ,
	\end{aligned}
	\end{equation}
Thus $e^{-2V}\mathcal L^\lambda$ is in divergence form. Its coefficient matrix is $e^{-2V}\a$. The operator is symmetric in $L^2(e^{2\lambda x_1} dx)$.

The original generator and the divergence-form generator \eqref{eq.reduced-generator} correspond to the same diffusion with different time parametrizations. We record this time change in both directions. Let $t(s)$ be the inverse of the strictly increasing map
\[
	r\mapsto \int_0^r e^{2V(X^\lambda(u))} du\, .
\]
Then $dt/ds=e^{-2V(X^\lambda(t(s)))}$, and the time-changed process
\[
	Y^\lambda(s)\coloneqq X^\lambda(t(s))
\]
has generator $e^{-2V}\mathcal L^\lambda$.

In the other direction, the additive clock of $Y^\lambda$ is
\[
	A^\lambda(s)\coloneqq \int_0^s e^{-2V(Y^\lambda(u))} du\, .
\]
This clock equals $t(s)$, and inverting gives
\begin{equation}\label{eq.X-Y-relation}
	X^\lambda(t)=Y^\lambda\bigl((A^\lambda)^{-1}(t)\bigr)\qquad\text{(pathwise)}\, .
\end{equation}

For $p\in\R^d$, let $\chi_p$ be the corrector for the field $e^{-2V}\a$ in direction $p$, characterized by $\E_\Q[\nabla\chi_p]=0$ and
\[
	-\nabla\cdot\bigl(e^{-2V}\a(p+\nabla\chi_p)\bigr)=0\qquad\text{in }\R^d
\]
in the weak sense. The homogenized matrix of $-\nabla\cdot(e^{-2V}\a\nabla)$ is the matrix $\ahom$ defined by
\begin{equation}\label{eq.ahom-def}
	p\cdot\ahom q
	\coloneqq \E_\Q\Bigl[(p+\nabla\chi_p(0))\cdot e^{-2V(0)}\a(0)(q+\nabla\chi_q(0))\Bigr]\qquad(p,q\in\R^d)\, .
\end{equation}
The existence of these correctors and the deterministic matrix $\ahom$ follows from \citet{kozlov1985}. Set
\begin{equation}\label{eq.def-ell-eta}
	\ell(\lambda)\coloneqq\lim_{s\to\infty}s^{-1}Y^\lambda(s),\qquad
	\eta(\lambda)\coloneqq \lim_{s\to\infty}s^{-1}A^\lambda(s)\, ,
\end{equation}
the asymptotic velocity of the time-changed diffusion and its asymptotic clock rate; the existence of both limits is proved in Section~\ref{sec.renewal}. The time-change identities are
\begin{equation}\label{eq.time-change-identities}
	\ell_X(\lambda)=\eta(\lambda)^{-1} \ell(\lambda),\qquad \Sigma_X=\E[e^{-2V(0)}]^{-1} \ahom\, .
\end{equation}
The first identity follows from \eqref{eq.X-Y-relation}; the second is the identity $\Sigma_Y=\gamma\Sigma$ near the end of the proof of \citet[Proposition~3.1]{Einstein2012gantert}, applied with $\gamma=\E[e^{-2V(0)}]$ and $\Sigma_Y=\ahom$.

\subsection{Standing notation after the time change}\label{ss.standing}

All estimates in the remainder of the paper are carried out for the
divergence-form operator $e^{-2V}\mathcal L^\lambda$ obtained in
Subsection~\ref{ss.reduction}, whose coefficient field is $e^{-2V}\a$ and whose
associated diffusion is $Y^\lambda$; the operator $\mathcal L^\lambda$ and the
diffusion $X^\lambda$ in their original form are no longer needed. From this point onward, we write $\a$ for $e^{-2V}\a$ and $X^\lambda$ for $Y^\lambda$. The factor $e^{-2V}$ is absorbed into $\a$ and no longer appears as a coefficient; it appears separately only in the clock integrand.

After enlarging $\Lambda$ to absorb the constants from Assumption~\ref{a.environment}, the renamed $\a$ is symmetric, $\Lambda$-uniformly elliptic, and $\Lambda$-Lipschitz, and $e^{-2V}\in[\Lambda^{-1},\Lambda]$ is $\Lambda$-Lipschitz. The pair $(\a,e^{-2V})$ is stationary under $\Q$ with range of dependence $r_0$, and the matrix $\ahom$ from \eqref{eq.ahom-def} is the homogenized matrix of the renamed field $\a$, satisfying $\Lambda^{-1}I\le\ahom\le\Lambda I$.

	\section{Inputs from quantitative homogenization}\label{sec.coarse-graining-inputs}

The resolvent estimates of Section~\ref{sec.tilted-resolvent} reduce, after subtracting the finite-volume correctors, to bounds on the difference $\a-\ahom$ in a weak norm. This section records the two quantitative homogenization estimates that supply these bounds: a coarse-graining estimate for the flux error $(\a-\ahom)\nabla u$ of a solution $u$ (inequality~\eqref{eq.cg-RHS}), and a finite-volume corrector estimate (inequality~\eqref{eq.hc}). Both are from \citet{AK_book} and \citet{AK_HC}. We also record a stochastic integrability lemma and a negative-Sobolev bound on the centered clock source $e^{-2V}-\E[e^{-2V(0)}]$.

	\subsection{Finite-volume correctors and quantitative estimates}

	For every $M\in\N$ and $p\in\R^d$, let $w_{p,M}=w(\cdot,\cu_M,p)\in p\cdot x+H^1_0(\cu_M)$ be the Dirichlet corrector on $\cu_M$ in the sense of \citet{AK_HC}, that is, the unique weak solution of
	\begin{equation}\label{eq.ak-corrector-pde}
		-\nabla\cdot(\a\nabla w_{p,M})=0\quad\text{in }\cu_M\, ,
		\qquad
		w_{p,M}=p\cdot x\quad\text{on }\partial \cu_M\, .
	\end{equation}
	When the slope~$p$ is a unit basis vector~$e_j$, we abbreviate $w_{j,M}\coloneqq w_{e_j,M}$.

	The next proposition collects the two homogenization estimates used in Section~\ref{sec.tilted-resolvent}: a coarse-graining bound for solutions with a divergence-form right-hand side, and the finite-volume Dirichlet corrector estimate. Both hold beyond a random scale $\X_{\beta_h}$ with a stretched-exponential tail.

	\begin{proposition}[Quantitative homogenization estimates]\label{pro.ak-inputs}
		Fix any $\beta_h\in(0,1/8]$. There exist deterministic constants $C=C(d,\Lambda,r_0,\beta_h)<\infty$ and $c=c(d,\Lambda,r_0,\beta_h)>0$, a deterministic exponent $\sigma=\sigma(d,\Lambda,r_0,\beta_h)\in(0,1)$, and a random scale $\X_{\beta_h}\colon\Omega\to[0,\infty)$ with the following properties. First, the random scale has a stretched-exponential tail:
		\begin{equation}\label{eq.m-beta0-tail}
			\Q\bigl(\X_{\beta_h}\ge M\bigr)\le C\exp\bigl(-c M^\sigma\bigr)\qquad\text{for every }M\ge 0\, .
		\end{equation}
		Second, on a $\Q$-full-probability event, the following estimates hold for every integer $M\ge \X_{\beta_h}(\omega)$.
		\begin{itemize}
			\item Let $k\in\{1,d\}$, let $u\in H^1(\cu_M;\R^k)$, and let $f\in H^{1/6}(\cu_M;\R^k\otimes\R^d)$ satisfy $-\nabla\cdot(\a\nabla u)=\nabla\cdot f$ componentwise in $\cu_M$. Then
			\begin{equation}\label{eq.cg-RHS}
				3^{-M/6}\|(\a-\ahom)\nabla u\|_{\Hminusuls{-1/6}(\cu_M)}\le C 3^{-\beta_h M}\bigl(\|\nabla u\|_{\underline L^2(\cu_M)}+3^{M/6}[f]_{\underline H^{1/6}(\cu_M)}\bigr)\, .
			\end{equation}
			\item For every $j$ with $1\le j\le d$,
			\begin{equation}\label{eq.hc}
			\begin{aligned}
				&3^{-M}\|w_{j,M}-x_j\|_{\underline L^2(\cu_M)}
				+3^{-M/6}\|\nabla w_{j,M}-e_j\|_{\Hminusuls{-1/6}(\cu_M)}\\
				&\qquad
				+3^{-M/6}\|\a\nabla w_{j,M}-\ahom e_j\|_{\Hminusuls{-1/6}(\cu_M)}
				\le C 3^{-\beta_h M}\, .
			\end{aligned}
			\end{equation}
		\end{itemize}
	\end{proposition}

	\begin{proof}
		The finite range of dependence in Assumption~\ref{a.environment} implies the concentration-for-sums condition of \citet[Definition~3.2 and Section~3.2.1]{AK_book}. The coarse-graining estimate with a right-hand side, namely~\eqref{eq.cg-RHS}, is supplied by \citet[Lemma~2.12]{algsuperdiff}. For \eqref{eq.hc}, we use the finite-volume Dirichlet corrector estimate of \citet{AK_HC}: their display~(5.54) is the definition of $w(\cdot,\square_m,e)$ used in \eqref{eq.ak-corrector-pde}, and displays~(5.61)--(5.62), with $s=1/6$ and $a_1=\ahom$, reduce the three bounds in \eqref{eq.hc} to their large-scale finite-volume error estimate. In the uniformly elliptic finite-range setting, \citet[Theorem~B and displays~(5.102)--(5.103)]{AK_HC} imply that this error is at most $C3^{-\beta_h M}$ above a random scale with a stretched-exponential tail, after taking $\beta_h\le1/8$ no larger than the deterministic exponent supplied there. Finally, we choose $\X_{\beta_h}$ as the maximum of the scales required for \eqref{eq.cg-RHS} and \eqref{eq.hc}; a union bound preserves the tail bound \eqref{eq.m-beta0-tail}.
	\end{proof}

	Throughout the rest of the paper we will enlarge $\X_{\beta_h}$ by deterministic constants depending only on $(d,\Lambda,r_0,\beta_h)$ to absorb deterministic threshold conditions arising in later proofs. 

	The remaining bookkeeping lemma converts a stretched-exponential scale tail, together with a polynomial tail for a small-$\lambda$ threshold, into the small-$\lambda$ tail used in the proof of Theorem~\ref{t.main}.
	
	\begin{lemma}\label{lem.tail-conversion}
		Let $\alpha\in(0,1)$ and $\sigma\in(0,1]$. Let $T\colon\Omega\to[0,\infty)$ and $\lambda_0\colon\Omega\to(0,1]$ be random variables. Suppose there exist deterministic $C_T<\infty$ and $c_T>0$ such that
		\begin{equation}\label{eq.tail-conversion-T}
			\Q(T\ge M)\le C_T\exp(-c_T M^\sigma)\qquad\text{for every }M\ge 0\, ,
		\end{equation}
		and that for every $p>0$ there exists a deterministic $C_p<\infty$ such that
		\begin{equation}\label{eq.tail-conversion-lambda0}
			\Q(\lambda_0\le 2^{-k})\le C_p 2^{-pk}\qquad\text{for every }k\in\N\, .
		\end{equation}
		Let $\lambda_*(\omega)=3^{-(1-\alpha)T(\omega)}\wedge\lambda_0(\omega)\wedge 1$. Then there exist deterministic $C_*<\infty$ and $c_*>0$, depending only on $\alpha$, $\sigma$, $C_T$, $c_T$, and the family $(C_p)$, such that
		\begin{equation}\label{eq.tail-conversion-conclusion}
			\Q(\lambda_*\le\lambda)\le C_*\exp\bigl(-c_* (\log(1/\lambda))^\sigma\bigr)\qquad\text{for every }\lambda\in(0,1]\, .
		\end{equation}
	\end{lemma}

	\begin{proof}
		The event $\{\lambda_*\le\lambda\}$ is contained in $\{T\ge\log_3(1/\lambda)/(1-\alpha)\}\cup\{\lambda_0\le\lambda\}$. By \eqref{eq.tail-conversion-T}, the first event has $\Q$-probability at most $C_T\exp(-c_T((1-\alpha)\log 3)^{-\sigma}(\log(1/\lambda))^\sigma)$. For the second, dyadic discretization $2^{-k-1}<\lambda\le 2^{-k}$ and \eqref{eq.tail-conversion-lambda0} give $\Q(\lambda_0\le\lambda)\le 2^p C_p\lambda^p$ for every $p>0$; since $\sigma\le 1$ and $\lambda^p\le\exp(-p(\log(1/\lambda))^\sigma)$ for $\log(1/\lambda)\ge 1$, choosing $p\ge c_T((1-\alpha)\log 3)^{-\sigma}$ and applying the union bound gives \eqref{eq.tail-conversion-conclusion} on $\lambda\le e^{-1}$. For $\lambda\in(e^{-1},1]$, the trivial bound $\Q(\lambda_*\le\lambda)\le 1$ is dominated by $C_*\exp(-c_*(\log(1/\lambda))^\sigma)$ after enlarging $C_*$ by the factor $\exp(c_*)$.
	\end{proof}

\section{Tilted resolvent estimates}\label{sec.tilted-resolvent}

This section proves the analytic estimates used in the proof of
Theorem~\ref{t.main}. Recall from Section~\ref{sec.notation} that
$\ell(\lambda)$ and $\eta(\lambda)$ denote the asymptotic velocity and clock
rate of the time-changed diffusion, and that the time-change identities reduce
the mobility error $\ell_X(\lambda)/\lambda-\Sigma_X e_1$ to estimates on
$\ell(\lambda)/\lambda-\ahom e_1$ and
$\eta(\lambda)-\E[e^{-2V(0)}]$.

As explained in the introduction, the velocity error decomposes into a
time-average contribution and a homogenization contribution. The time-average
contribution is controlled by the renewal estimates of Section~\ref{sec.renewal}.
The homogenization contribution is the deviation
$3^{-m}u_m(0)-\ahom e_1$, where $u_m$ solves a resolvent equation on a cube
$\cu_{m+h}$ at the critical length scale $3^m\sim \lambda^{-1}$, with $h$
controlling the distance from the origin to the boundary. In a homogeneous
medium with constant coefficient $\ahom$, the resolvent has the constant
solution $3^m\ahom e_1$, so $3^{-m}u_m(0)-\ahom e_1$ measures how far the
heterogeneous resolvent departs from its homogenized value. 

The additional difficulty, compared to the homogenization estimates of
Section~\ref{sec.coarse-graining-inputs}, is the first-order drift
$\lambda(\a e_1)\cdot\nabla u$ in the resolvent equation. The estimates of
Section~\ref{sec.coarse-graining-inputs} apply to solutions of the symmetric
operator $-\nabla\cdot(\a\nabla u)=0$. The tilted resolvent instead solves an
equation of the form
$-\nabla\cdot(\a\nabla u)+\lambda(\a e_1)\cdot\nabla u=\cdots$, and the drift
term is not perturbative at the critical scale. To see why, test the equation
against its solution: in the symmetric case, this produces an energy identity
that controls $|\nabla u|_{L^2}$. The drift adds an integral
$\lambda\int u\,(\a e_1)\cdot\nabla u$, which at $\lambda\sim 3^{-m}$ carries
no small prefactor. We handle it by subtracting the finite-volume Dirichlet
correctors of Section~\ref{sec.coarse-graining-inputs}: the correctors absorb
the leading-order part of this integral, and the residual is small enough to
close the energy estimate (Lemma~\ref{lem.drift-bound}). This gives a global
$\underline L^2$ bound on the homogenization error. A pointwise bound at the
origin then follows by restricting to a concentric subcube on which the
exponential weight $e^{2\lambda x_1}$ is bounded, so that a deterministic
interior $L^\infty$ estimate applies. The velocity bound is
Proposition~\ref{pro.velocity-resolvent-infty}. The clock resolvent, which captures the homogenization contribution to the clock error $\eta(\lambda)-\E[e^{-2V(0)}]$, is treated by the same two-step argument, with the corrector bound replaced by cancellations in the negative-Sobolev norm of $e^{-2V}-\E[e^{-2V(0)}]$; the result is Proposition~\ref{pro.clock-resolvent-infty}. 

The exit-time bound (Proposition~\ref{pro.exit}) does not use homogenization; it transfers a crude
Aronson heat-kernel estimate from the symmetric diffusion to the tilted one via
Girsanov's theorem. A more refined exit-time bound is accessible via homogenization theory, but we do not pursue this.

Subsection~\ref{ss.resolvent-setup} sets up the equations and the
Feynman--Kac formula. Subsection~\ref{ss.apriori} records the analytic tools.
Subsections~\ref{ss.velocity-resolvent} and~\ref{ss.clock-resolvent} prove the
velocity and clock estimates, and Subsection~\ref{ss.exit-time} proves the
exit-time bound.

\subsection{Setup}\label{ss.resolvent-setup}

Following the convention of Subsection~\ref{ss.standing}, $\a$ denotes the renamed field $e^{-2V}\a$ and $X^\lambda$ denotes the time-changed diffusion $Y^\lambda$. The tilted generator is
\[
	\mathcal{L}^\lambda f(x)\coloneqq \tfrac12 e^{-2\lambda x_1}\nabla\cdot(e^{2\lambda x_1} \a(x)\nabla f)\, .
\]
The corresponding SDE is
\begin{equation}\label{eq.SDE-degenerate}
	dX^\lambda(t)=\b(X^\lambda(t)) dt+\a^{1/2}(X^\lambda(t)) dW(t)\, ,
\end{equation}
driven by a $d$-dimensional Brownian motion $W$, with $\a^{1/2}$ the symmetric Lipschitz square root and It\^o drift
\begin{equation}\label{eq.diffusivity_and_drift-degenerate}
	\b(x)=\tfrac12\nabla\cdot\a(x)+\lambda\a(x)e_1\, .
\end{equation}
Let $m,h\in\N$, let $\rho\ge 1$, and set $\lambda=\rho3^{-m}$. Let $u_m\in 3^m\ahom e_1+H_0^1(\cu_{m+h};\R^d)$ solve
\begin{equation}\label{eq.resolvent-equation}
	\begin{cases}
		\rho 3^{-2m}u_m-\mathcal{L}^\lambda u_m=\mathbf{b}&\text{in }\cu_{m+h}\, ,\\
		u_m=3^m\ahom e_1&\text{on }\partial\cu_{m+h}\, .
	\end{cases}
\end{equation}

Let $q_m\in H^1_0(\cu_{m+h})$ solve
\begin{equation}\label{eq.clock-resolvent}
	\rho3^{-2m}q_m-\mathcal L^\lambda q_m=\rho3^{-2m}\bigl(e^{-2V}-\E[e^{-2V(0)}]\bigr)
	\qquad\text{in }\cu_{m+h},\qquad q_m=0\quad\text{on }\partial\cu_{m+h}\, .
\end{equation}

The exit time from the cube is
\begin{equation}\label{eq.tau-mh}
	\tau_{m,h}\coloneqq\inf\{t\ge0:X^\lambda(t)\notin\cu_{m+h}\}\, .
\end{equation}
Set $v_m(x)\coloneqq u_m(x)+x$. Since $\mathcal L^\lambda x=\b$, $v_m$ solves
\[
	\rho3^{-2m}v_m-\mathcal L^\lambda v_m=\rho3^{-2m}x
	\quad\text{in }\cu_{m+h},\qquad
	v_m=x+3^m\ahom e_1\quad\text{on }\partial\cu_{m+h}.
\]
The Feynman--Kac formula gives
\begin{equation}\label{eq.FK-resolvent}
	\begin{aligned}
		u_m(x)
		&=\rho 3^{-2m} E^{\lambda,\omega}_x \biggl[\int_{0}^{\tau_{m,h}}e^{-\rho 3^{-2m} t} X^\lambda(t) dt\biggr]\\
		&\quad+E^{\lambda,\omega}_x \Bigl[e^{-\rho 3^{-2m} \tau_{m,h}}\bigl(X^\lambda(\tau_{m,h})+3^m\ahom e_1\bigr)\Bigr]-x\, .
	\end{aligned}
\end{equation}

Substituting the rescaled homogenization error
\begin{equation}\label{def.Um}
	U_m(x)\coloneqq 3^{-m}u_m(x)-\ahom e_1\in H^1_0(\cu_{m+h};\R^d)
\end{equation}
into \eqref{eq.resolvent-equation} produces
\begin{equation}\label{eq.resolvent-equation-Um}
	\rho 3^{-2m}U_m-\tfrac12 e^{-2\lambda x_1}\nabla\cdot(e^{2\lambda x_1}\a\nabla U_m)=\tfrac{3^{-m}}2 e^{-2\lambda x_1}\nabla\cdot(e^{2\lambda x_1}(\a-\ahom))\quad\text{in }\cu_{m+h}\, ,
\end{equation}
or componentwise, for $1\le i\le d$,
\begin{equation}\label{eq.resolvent-equation-Um-expand}
	\rho 3^{-2m}U_{m,i}-\tfrac12\nabla\cdot(\a\nabla U_{m,i})-\lambda(\a e_1)\cdot\nabla U_{m,i}
	=\tfrac{3^{-m}}{2}\nabla\cdot((\a-\ahom)e_i)+\lambda 3^{-m}((\a-\ahom)e_1)_i\, .
\end{equation}

\subsection{Divergence inverses, drift absorption, and interior boundedness}\label{ss.apriori}

	This subsection records basic deterministic and probabilistic estimates used in the resolvent bounds.

	\begin{lemma}\label{lem.divinverses}
		Let $M\in\N$.
		\begin{enumerate}
			\item[\rm(i)] For every $s\in(0,1)$ there exists $C=C(d,s)<\infty$ such that, for every $g\in L^2(\cu_M)$, there exists $G\in H^1(\cu_M;\R^d)$ with $\nabla\cdot G=g$ in $\cu_M$ and
			\begin{equation}\label{eq.bog}
				3^{sM}\|G\|_{\underline H^s(\cu_M)}\le C 3^M\|g\|_{\underline L^2(\cu_M)}\, .
			\end{equation}
			\item[\rm(ii)] For every $s\in(0,1/2]$ there exists $C=C(d,s)<\infty$ such that, for every $g\in\Hminusuls{-s}(\cu_M)$ with $\langle g,\mathbf 1\rangle_{\cu_M}=0$, there exists $G\in H^s(\cu_M;\R^d)$ satisfying 
			\begin{equation}\label{eq.weak-neumann-G}
				\fint_{\cu_M}\nabla\varphi\cdot G=-\langle g,\varphi\rangle_{\cu_M}\qquad\text{for every }\varphi\in C^\infty(\overline{\cu_M})\, ,
			\end{equation}
			and
			\begin{equation}\label{eq.neg-div-inv}
				\|G\|_{\underline L^2(\cu_M)}+3^{sM}[G]_{\underline H^{s}(\cu_M)}\le C 3^{(1-s)M}\|g\|_{\Hminusuls{-s}(\cu_M)}\, .
			\end{equation}
		\end{enumerate}
	\end{lemma}

	\begin{proof}
		\emph{Part (i).} Let $A(x)\coloneqq(\fint_{\cu_M}g)x_1e_1$ and $g_0\coloneqq g-(g)_{\cu_M}$. The Bogovski\u\i\ operator on cubes \citep[Section~III.3]{galdi2011introduction}, rescaled from $\cu_0$ to $\cu_M$, gives $B\in H^1_0(\cu_M;\R^d)$ with $\nabla\cdot B=g_0$ and
		\[
			\|B\|_{\underline L^2(\cu_M)}+3^M\|\nabla B\|_{\underline L^2(\cu_M)}\le C3^M\|g_0\|_{\underline L^2(\cu_M)}\, .
		\]
		The affine field satisfies the same bound with $B$ and $g_0$ replaced by $A$ and $(g)_{\cu_M}$. For $G=A+B$, interpolation and the pullback identities \eqref{eq.pullback-identities} give \eqref{eq.bog}.

		\emph{Part (ii).} Let $\tilde g\in \hat H^{-s}(\cu_0)$ be defined by $\langle\tilde g,\tilde\varphi\rangle_{\cu_0}=\langle g,\varphi\rangle_{\cu_M}$ whenever $\tilde\varphi(\tilde x)=\varphi(3^M\tilde x)$; by \eqref{eq.pullback-identities-dual}, $\|\tilde g\|_{\hat H^{-s}(\cu_0)}=3^{-sM}\|g\|_{\Hminusuls{-s}(\cu_M)}$. On the unit cube, Lax--Milgram gives the zero-average solution $v$ of the Neumann problem with datum $\tilde g$. The regularity estimate obtained from $H^2$ Neumann regularity on cubes \citep[Lemma~B.19]{AKM2019book} and interpolation give
		\[
			\|v\|_{H^{2-s}(\cu_0)}\le C\|\tilde g\|_{\hat H^{-s}(\cu_0)},\qquad
			\int_{\cu_0}\nabla v\cdot\nabla\varphi=-\langle \tilde g,\varphi\rangle_{\cu_0}\qquad(\varphi\in H^1(\cu_0))\, .
		\]
		Set $G_0=\nabla v$, use $H^{1-s}(\cu_0)\subset H^s(\cu_0)$ for $s\le1/2$, and scale back with $G(x)=3^M G_0(3^{-M}x)$. The weak identity \eqref{eq.weak-neumann-G} and the bound \eqref{eq.neg-div-inv} follow from \eqref{eq.pullback-identities}--\eqref{eq.pullback-identities-dual}.
	\end{proof}

	The centered clock source $e^{-2V}-\E[e^{-2V(0)}]$ is small in the negative Sobolev norm at large scales.

	\begin{lemma}[Negative Sobolev bound on the centered clock source]\label{lem.clock-negnorm}
		Let $s\in(0,1/2]\cap(0,d/2)$ and $\varepsilon>0$. There exists $C_\varepsilon=C_\varepsilon(d,\Lambda,s,\varepsilon,r_0)<\infty$ such that, for every integer $N\ge0$,
		\begin{equation}\label{eq.clock-neg}
			\Q\Bigl(\text{there is an integer }M\ge N\text{ with }\|e^{-2V}-\E[e^{-2V(0)}]\|_{\Hminusuls{-s}(\cu_M)}>3^{\varepsilon M}\Bigr)\le C_\varepsilon 3^{-2\varepsilon N}\, .
		\end{equation}
		Moreover, for $\Q$-a.e.\ $\omega$, the bound $\|e^{-2V}-\E[e^{-2V(0)}]\|_{\Hminusuls{-s}(\cu_M)}\le 3^{\varepsilon M}$ holds for all sufficiently large integers $M$.
	\end{lemma}

	\begin{proof}
		Fix $M\ge0$ and let $F=e^{-2V}-\E[e^{-2V(0)}]$. For $0\le n\le M$, partitioning $\cu_M$ into triadic subcubes $y+\cu_n$ with centers ranging over $3^n\ZZ^d\cap\cu_M$, Proposition~1.10, display~(1.84), of \citet{AK_book}, applied to $H^s$ test functions and combined with the fractional Poincar\'e inequality on cubes \citep[Proposition~B.11]{AKM2019book}, gives
		\[
			\|F\|_{\Hminusuls{-s}(\cu_M)}
			\le C\left(\|F\|_{\underline L^2(\cu_M)}+\sum_{n=0}^M3^{sn}\left(3^{-(M-n)d}\sum_{y\in 3^n\ZZ^d\cap\cu_M}|(F)_{y+\cu_n}|^2\right)^{1/2}\right)\, .
		\]
		Since $F$ is bounded, the Minkowski inequality gives
		\[
		\begin{aligned}
		\left(\E_\Q\bigl[\|F\|_{\Hminusuls{-s}(\cu_M)}^2\bigr]\right)^{1/2}
		&\le C\left(1+\sum_{n=0}^M3^{sn}\left(3^{-(M-n)d}\sum_{y\in 3^n\ZZ^d\cap\cu_M}\E_\Q\bigl[(F)_{y+\cu_n}^2\bigr]\right)^{1/2}\right)\\
		&\le C\left(1+\sum_{n=0}^M3^{(s-d/2)n}\right)\le C\, ,
		\end{aligned}
		\]
		where the second inequality follows from $\E_\Q[(F)_{y+\cu_n}^2]\le C3^{-dn}$. Indeed,
		\[
			\E_\Q[(F)_{y+\cu_n}^2]
			=|{\cu_n}|^{-2}\int_{y+\cu_n}\int_{y+\cu_n}\E_\Q[F(x)F(z)] dx dz
			\le C|{\cu_n}|^{-1}
			=C3^{-dn}\, ,
		\]
		because the covariance vanishes when $|x-z|>r_0$ and $F$ is bounded. Markov's inequality and a union bound over $M\ge N$ give \eqref{eq.clock-neg}; the almost-sure eventual bound follows by letting $N\to\infty$.
	\end{proof}

\begin{lemma}[Tail of $\X_\mu$]\label{lem.m-mu-tail}
	Let $\X_\mu(\omega)$ be the smallest integer $n\ge \X_{\beta_h}(\omega)$ satisfying
	\begin{equation}\label{eq.m-mu}
		\|e^{-2V}-\E[e^{-2V(0)}]\|_{\Hminusuls{-1/6}(\cu_M)}\le 3^{M/24}\qquad\text{for every integer }M\ge n\, .
	\end{equation}
	There exist deterministic $C<\infty$ and $c>0$ such that
	\begin{equation}\label{eq.m-mu-tail}
		\Q(\X_\mu\ge M)\le C\exp(-cM^{\sigma})\qquad\text{for every }M\ge 0\, ,
	\end{equation}
	with $\sigma\in(0,1)$ as in \eqref{eq.m-beta0-tail}.
\end{lemma}

\begin{proof}
	For every integer $N\ge1$,
	\[
	\begin{aligned}
		\{\X_\mu\ge N\}\subseteq{}& \{\X_{\beta_h}\ge N-1\}\\
			&\cup
			\left\{
			\begin{aligned}
			&\text{there is an integer }K\ge N-1\text{ such that}\\
			&\|e^{-2V}-\E[e^{-2V(0)}]\|_{\Hminusuls{-1/6}(\cu_K)}>3^{K/24}
			\end{aligned}
			\right\}\, .
	\end{aligned}
	\]
	Combined with \eqref{eq.clock-neg} and \eqref{eq.m-beta0-tail}, this gives
	\[
		\Q(\X_\mu\ge N)\le C\exp(-c(N-1)^\sigma)+C 3^{-(N-1)/12}\le C\exp(-cN^{\sigma})
	\]
	after decreasing $c$ and absorbing constants. 
\end{proof}

For the rest of this section, $\beta_h$ is the exponent fixed in Proposition~\ref{pro.ak-inputs}, and
\begin{equation}\label{eq.eta-delta-def}
	\delta\coloneqq 2\beta_h/(d+4)\, .
\end{equation}
Let $\zeta_H=\zeta_H(d,\beta_h)\in(0,1)$ be a constant small enough that
\begin{align*}
	\tfrac{7\delta}{24}+\max(0,\tfrac16-\beta_h)\zeta_H&<\beta_h\, ,\\
	\tfrac{19\delta}{24}-\beta_h+\bigl(\tfrac76+\max(0,\tfrac16-\beta_h)-\tfrac16\bigr)\zeta_H&<0\, ,\\
	\tfrac{7\delta}{12}-2\beta_h+(2+\max(0,\tfrac13-2\beta_h))\zeta_H&<0
\end{align*}
are satisfied; this is possible since $\delta=2\beta_h/(d+4)$ gives $7\delta/24<\beta_h$, $19\delta/24<\beta_h$, and $7\delta/12<2\beta_h$. We consider only parameters $m,h\in\N$ and $\rho\ge 1$ in the admissible regime
\begin{equation}\label{eq.admissible}
	m\ge\X_{\beta_h}(\omega),\qquad h\le\zeta_H m,\qquad 1\le\rho\le 3^{\delta m/2}\, ,
\end{equation}
and we write $\lambda=\rho 3^{-m}$ and $M=m+h$. We allow deterministic constants to depend on $(d,\Lambda,r_0,\beta_h)$.

	\begin{lemma}\label{lem.drift-bound}
		There exists $C<\infty$ such that, for $\Q$-a.e.\ $\omega$ and every $(m,h,\rho)$ in the admissible regime~\eqref{eq.admissible}, the following holds. Suppose that $z=(z_1,\ldots,z_d)\in H^1_0(\cu_M;\R^d)$ solves, componentwise in $\cu_M$,
		\begin{equation}\label{eq.drift-bound-pde}
			-\nabla\cdot(\a\nabla z_i)
			=2\lambda 3^{-m} e_1 \cdot(\a\nabla w_{i,M}-\ahom e_i)-2\rho 3^{-3m}(w_{i,M}-x_i)+2\lambda(\a e_1) \cdot \nabla z_i-2\rho 3^{-2m}z_i\, .
		\end{equation}
		Then the drift integral satisfies
		\begin{equation}\label{eq.drift-bound-conclusion}
			\begin{aligned}
			\left|\lambda\sum_{i=1}^d\fint_{\cu_M} z_i\bigl((\a-\ahom)e_1\bigr)\cdot\nabla z_i\right|
			&\le \frac{1}{16}\left(\rho 3^{-2m}\|z\|_{\underline L^2(\cu_M)}^2
			+\frac12\sum_{i=1}^d\fint_{\cu_M}\nabla z_i\cdot\a\nabla z_i\right)\\
			&\quad+C \rho^2 3^{2h-2m} 3^{-2\beta_h M}\, .
			\end{aligned}
		\end{equation}
	\end{lemma}

	\begin{proof}
		Let
		\[
			N_z=\rho^{1/2}3^{-m}\|z\|_{\underline L^2(\cu_M)},\qquad
			E_z^2=\tfrac12\sum_{i=1}^d\fint_{\cu_M}\nabla z_i\cdot\a\nabla z_i\, ,
		\]
		\[
			A_z\coloneqq\lambda\sum_{i=1}^d\fint_{\cu_M} z_i\bigl((\a-\ahom)e_1\bigr)\cdot\nabla z_i\, .
		\]
			Write $\mathsf E_z\coloneqq N_z^2+E_z^2$. In this proof, constants denoted by $C$ depend only on $(d,\Lambda,r_0,\beta_h)$. The three inequalities defining $\zeta_H$ are used in the absorption argument of Step~4.

		\emph{Step 1.} Denote the cube average
		\[
			\bar c_i\coloneqq e_1\cdot\fint_{\cu_M}(\a\nabla w_{i,M}-\ahom e_i)\, .
		\]
		By duality with the constant test function, and since
		$\|1\|_{\underline H^{1/6}(\cu_M)}=3^{-M/6}$, the estimate \eqref{eq.hc}
		gives, for $M\ge \X_{\beta_h}(\omega)$,
		\begin{equation}\label{eq.drift-cbar-bound}
			|\bar c_i|\le 3^{-M/6}\|\a\nabla w_{i,M}-\ahom e_i\|_{\Hminusuls{-1/6}(\cu_M)}\le C3^{-\beta_h M}\, .
		\end{equation}

		\emph{Step 2.} We rewrite the right-hand side of \eqref{eq.drift-bound-pde} in divergence form. Split it as $g^{(L^2)}_i+g^{(\mathrm{flux})}_i$ with
		\[
		\begin{aligned}
			g^{(L^2)}_i&\coloneqq-2\rho 3^{-3m}(w_{i,M}-x_i)+2\lambda(\a e_1)\cdot\nabla z_i-2\rho 3^{-2m}z_i+2\lambda 3^{-m}\bar c_i\, ,\\
			g^{(\mathrm{flux})}_i&\coloneqq 2\lambda 3^{-m}\bigl(e_1\cdot(\a\nabla w_{i,M}-\ahom e_i)-\bar c_i\bigr)\, .
		\end{aligned}
		\]
		Lemma~\ref{lem.divinverses}(i) applied to $g^{(L^2)}_i\in L^2(\cu_M)$ produces $G^{(L^2)}_i\in H^1(\cu_M;\R^d)$ with $\nabla\cdot G^{(L^2)}_i=g^{(L^2)}_i$ and
		\begin{equation}\label{eq.drift-bog-bound}
			3^{M/6}[G^{(L^2)}_i]_{\underline H^{1/6}(\cu_M)}\le C 3^M \|g^{(L^2)}_i\|_{\underline L^2(\cu_M)}\, .
		\end{equation}
		The construction makes $g^{(\mathrm{flux})}_i$ mean-zero, so Lemma~\ref{lem.divinverses}(ii) (with $s=1/6$) yields $G^{(\mathrm{flux})}_i\in H^{1/6}(\cu_M;\R^d)$ with $\nabla\cdot G^{(\mathrm{flux})}_i=g^{(\mathrm{flux})}_i$ and
		\begin{equation}\label{eq.drift-negdiv-bound}
			\|G^{(\mathrm{flux})}_i\|_{\underline L^2(\cu_M)}+3^{M/6}[G^{(\mathrm{flux})}_i]_{\underline H^{1/6}(\cu_M)}\le C 3^{5M/6} \|g^{(\mathrm{flux})}_i\|_{\Hminusuls{-1/6}(\cu_M)}\, .
		\end{equation}
		Setting $f_i\coloneqq G^{(L^2)}_i+G^{(\mathrm{flux})}_i\in H^{1/6}(\cu_M;\R^d)$, we have $-\nabla\cdot(\a\nabla z_i)=\nabla\cdot f_i$ componentwise.

		\emph{Step 3.} We bound the source $3^{M/6}[f_i]_{\underline H^{1/6}}$. Each term of $g^{(L^2)}_i$ is controlled directly:
		\begin{align*}
			2\rho 3^{-3m}\|w_{i,M}-x_i\|_{\underline L^2}&\le C\rho 3^{-2m+h}3^{-\beta_h M}&&\text{by \eqref{eq.hc},}\\
			2\lambda\|\a e_1\cdot\nabla z_i\|_{\underline L^2}&\le C\rho 3^{-m}E_z&&\text{by ellipticity and }\lambda=\rho 3^{-m}\, ,\\
			2\rho 3^{-2m}\|z_i\|_{\underline L^2}&\le 2\rho^{1/2}3^{-m}N_z&&\text{by the definition of }N_z\, ,\\
			2\lambda 3^{-m}|\bar c_i|&\le C\rho 3^{-2m}3^{-\beta_h M}&&\text{by \eqref{eq.drift-cbar-bound}.}
		\end{align*}
		Summing and inserting into \eqref{eq.drift-bog-bound} (with $3^M=3^{m+h}$) yields
		\begin{equation}\label{eq.drift-G-L2-Hs0}
			3^{M/6}[G^{(L^2)}_i]_{\underline H^{1/6}}\le C\bigl(\rho^{1/2}3^h N_z+\rho 3^h E_z+\rho 3^{2h-m}3^{-\beta_h M}\bigr)\, .
		\end{equation}

		\smallskip

			For every $w\in H^{1/6}(\cu_M)$, $|\fint_{\cu_M}w|\le\|w\|_{\underline L^2(\cu_M)}\le 3^{M/6}\|w\|_{\underline H^{1/6}(\cu_M)}$. Hence $\|1\|_{\Hminusuls{-1/6}(\cu_M)}\le 3^{M/6}$ and
		\[
			\|g^{(\mathrm{flux})}_i\|_{\Hminusuls{-1/6}}\le 2\lambda 3^{-m}\bigl(\|\a\nabla w_{i,M}-\ahom e_i\|_{\Hminusuls{-1/6}}+|\bar c_i| 3^{M/6}\bigr)\le C\rho 3^{-2m+M/6}3^{-\beta_h M}\, .
		\]
		Inserting into \eqref{eq.drift-negdiv-bound} and using $5M/6+M/6=M=m+h$,
		\begin{equation}\label{eq.drift-G-F-Hs0}
			3^{M/6}[G^{(\mathrm{flux})}_i]_{\underline H^{1/6}}\le C\rho 3^{h-m}3^{-\beta_h M}\, .
		\end{equation}
		Adding \eqref{eq.drift-G-L2-Hs0} and \eqref{eq.drift-G-F-Hs0} (and absorbing $\rho 3^{h-m}3^{-\beta_h M}\le\rho 3^{2h-m}3^{-\beta_h M}$),
		\[
			3^{M/6}[f_i]_{\underline H^{1/6}}\le C\bigl(\rho^{1/2}3^h N_z+\rho 3^h E_z+\rho 3^{2h-m}3^{-\beta_h M}\bigr)\, .
		\]

		\emph{Step 4.} We bound $|A_z|$ by combining the coarse-graining estimate with Sobolev interpolation. The estimate \eqref{eq.cg-RHS} applied to $z$ on $\cu_M$ gives, for $M\ge \X_{\beta_h}(\omega)$,
		\begin{equation}\label{eq.drift-cg-bound}
			\|(\a-\ahom)\nabla z\|_{\Hminusuls{-1/6}(\cu_M)}\le C 3^{(1/6-\beta_h)M}\Bigl(\sqrt{2\Lambda} E_z+\rho^{1/2}3^h N_z+\rho 3^h E_z+\rho 3^{2h-m}3^{-\beta_h M}\Bigr)\, .
		\end{equation}
		By the symmetry of $\a$ and $\ahom$ and duality,
		\begin{equation}\label{eq.drift-Az-duality}
			|A_z|\le \lambda \|z\|_{\underline H^{1/6}(\cu_M)} \|(\a-\ahom)\nabla z\|_{\Hminusuls{-1/6}(\cu_M)}\, .
		\end{equation}
		Interpolation gives
		\begin{equation}\label{eq.drift-interp}
			\|z\|_{\underline H^{1/6}(\cu_M)}\le C \|z\|_{\underline L^2(\cu_M)}^{5/6} \|\nabla z\|_{\underline L^2(\cu_M)}^{1/6}\qquad\text{for }z\in H^1_0(\cu_M;\R^d)\, .
		\end{equation}
		Substituting $\|z\|_{\underline L^2(\cu_M)}=\rho^{-1/2}3^m N_z$ and $\|\nabla z\|_{\underline L^2(\cu_M)}\le\sqrt{2\Lambda}E_z$,
		\begin{equation}\label{eq.drift-z-Hs0}
			\lambda \|z\|_{\underline H^{1/6}(\cu_M)}\le C\rho^{7/12} 3^{-m/6} N_z^{5/6} E_z^{1/6}\, .
		\end{equation}

		\emph{Step 5.} We absorb the four resulting terms into $\mathsf E_z$ via Young's inequality. Combining \eqref{eq.drift-Az-duality} with the bounds \eqref{eq.drift-cg-bound} and \eqref{eq.drift-z-Hs0} gives
		\[
			|A_z|\le\mathrm{(I)}+\mathrm{(II.a)}+\mathrm{(II.b)}+\mathrm{(II.c)}\, ,
		\]
		where the four terms are
		\begin{align*}
			\mathrm{(I)}&\coloneqq C \rho^{7/12} 3^{-m/6+(1/6-\beta_h)M} N_z^{5/6}E_z^{7/6}\, ,\\
			\mathrm{(II.a)}&\coloneqq C\rho^{13/12} 3^{-m/6+(1/6-\beta_h)M+h} N_z^{11/6}E_z^{1/6}\, ,\\
			\mathrm{(II.b)}&\coloneqq C\rho^{19/12} 3^{-m/6+(1/6-\beta_h)M+h} N_z^{5/6}E_z^{7/6}\, ,\\
			\mathrm{(II.c)}&\coloneqq C\rho^{19/12} 3^{-m/6+(1/6-\beta_h)M+2h-m} 3^{-\beta_h M} N_z^{5/6}E_z^{1/6}\, .
		\end{align*}

		For (I) and (II.b), Young's inequality with exponents $(12/5,12/7)$ gives $N_z^{5/6}E_z^{7/6}\le \mathsf E_z$. After substituting $\rho\le 3^{\delta m/2}$ and $h\le\zeta_H m$, the exponents in $m$ of the (I) and (II.b) prefactors are bounded above by $\tfrac{7\delta}{24}+\max(0,\tfrac16-\beta_h)\zeta_H-\beta_h$ and $\tfrac{19\delta}{24}-\beta_h+(\tfrac76+\max(0,\tfrac16-\beta_h)-\tfrac16)\zeta_H$, both strictly negative by the choice of $\zeta_H$. For $m\ge\X_{\beta_h}(\omega)$ enlarged by a deterministic constant depending only on $(d,\Lambda,r_0,\beta_h)$, the corresponding $3^{[\cdots]m}$ factors are at most $\tfrac{1}{64}$, so
		\[
			\mathrm{(I)}+\mathrm{(II.b)}\le \tfrac{1}{32}\mathsf E_z\, .
		\]

		For (II.a), Young's inequality with $(12/11,12)$ gives $N_z^{11/6}E_z^{1/6}\le\mathsf E_z$. The prefactor exponent of $m$ is at most that of (II.b) (since $13/12<19/12$), so this exponent is also strictly negative, and the same enlargement of $\X_{\beta_h}$ gives $\mathrm{(II.a)}\le \tfrac{1}{64}\mathsf E_z$.

		For (II.c), the degree of $(N_z,E_z)$ is $5/6+1/6=1$. The elementary inequality $N_z^{5/6}E_z^{1/6}\le N_z+E_z\le \sqrt{2\mathsf E_z}$ and Young's inequality,
		\[
			a \sqrt{2\mathsf E_z}\le \tfrac{1}{64}\mathsf E_z+32a^2,
		\]
		give
		\[
			\mathrm{(II.c)}\le \tfrac{1}{64}\mathsf E_z+C^2\rho^{19/6} 3^{-m/3+(1/3-2\beta_h)M+4h-2m} 3^{-2\beta_h M}\, .
		\]
		Substituting $\rho\le 3^{\delta m/2}$ and $h\le\zeta_H m$, the residual factors as
		\[
			C\rho^2 3^{2h-2m} 3^{-2\beta_h M}\cdot 3^{[7\delta/12-2\beta_h+(2+\max(0,1/3-2\beta_h))\zeta_H]m}\, ,
		\]
		with bracketed exponent strictly negative by the choice of $\zeta_H$. The same enlargement of $\X_{\beta_h}$ makes the $3^{[\cdots]m}$ factor $\le 1$, so the residual is bounded by $C\rho^2 3^{2h-2m}3^{-2\beta_h M}$.

		Adding the three contributions to the $\mathsf E_z$ coefficient ($\tfrac{1}{32}+\tfrac{1}{64}+\tfrac{1}{64}=\tfrac{1}{16}$),
		\[
				|A_z|\le \tfrac{1}{16}\mathsf E_z+C \rho^2 3^{2h-2m}3^{-2\beta_h M}\, ,
		\]
		which is \eqref{eq.drift-bound-conclusion}.
	\end{proof}

    We next record the interior $L^2$--$L^{\infty}$ estimate.

	\begin{lemma}[Interior $L^2$--$L^{\infty}$ estimate, Theorem 4.1 in \cite{han2000elliptic}]\label{lem.linfty}
		Let $M\in\N$, let $\a:\cu_M\to\R^{d\times d}$ be uniformly elliptic with constants $(\Lambda^{-1},\Lambda)$, let $c\in L^\infty(\cu_M)$ with $c\ge0$ and $\|c\|_{L^\infty}\le\Lambda$, and let $f\in L^\infty(\cu_M)$. If $u\in H^1(\cu_M)$ is a weak solution of $cu-\nabla\cdot(\a\nabla u)=f$ in $\cu_M$, then there exists $C=C(d,\Lambda)<\infty$ such that 
		\[
			\sup_{x\in\cu_{M-1}}|u(x)|\le C\bigl(\|u\|_{\underline L^2(\cu_M)}+3^{2M}\|f\|_{L^\infty(\cu_M)}\bigr)\, .
		\]
	\end{lemma}

	\subsection{Velocity resolvent}\label{ss.velocity-resolvent}
    We first prove an energy estimate for the velocity resolvent error and then obtain the pointwise estimate by passing to a subcube on which the exponential tilt is uniformly bounded.

	\begin{proposition}[Global $\underline L^2$ bound for the velocity resolvent]\label{pro.velocity-resolvent-L2}
		There exists $C<\infty$ such that, for $\Q$-a.e.\ $\omega$ and every $(m,h,\rho)$ in the admissible regime~\eqref{eq.admissible}, the homogenization error~$U_m$ defined by~\eqref{def.Um} satisfies
		\begin{equation}\label{eq.energy-l2-main}
			\|U_m\|_{\underline L^2(\cu_{m+h})}\le C \rho^{1/2}  3^{h} 3^{-\beta_h(m+h)}\, .
		\end{equation}
	\end{proposition}

	\begin{proof}
			Throughout, $C$ denotes a deterministic constant depending only on $(d,\Lambda,r_0,\beta_h)$, which may change from line to line. The full-probability event is the intersection of the events from Proposition~\ref{pro.ak-inputs} and Lemma~\ref{lem.drift-bound}. Set $M\coloneqq m+h$.

		\emph{Step 1.} For $1\le i\le d$, set $z_i\coloneqq U_{m,i}-3^{-m}(w_{i,M}-x_i)$ and $z\coloneqq(z_1,\ldots,z_d)$. Since $U_m\in H^1_0(\cu_M;\R^d)$ from \eqref{def.Um} and $w_{i,M}-x_i\in H^1_0(\cu_M)$, $z\in H^1_0(\cu_M;\R^d)$. Since $w_{i,M}$ is $\a$-harmonic, $\nabla\cdot(\a\nabla(w_{i,M}-x_i))=-\nabla\cdot(\a e_i)$ in $\cu_M$. Substituting into \eqref{eq.resolvent-equation-Um-expand} and using $\nabla\cdot(\ahom e_i)=0$ yields
	\begin{equation}\label{eq.z-pde}
		\rho 3^{-2m}z_i-\tfrac12\nabla\cdot(\a\nabla z_i)-\lambda(\a e_1)\cdot\nabla z_i=T_i,\qquad T_i\coloneqq \lambda 3^{-m}e_1\cdot(\a\nabla w_{i,M}-\ahom e_i)-\rho 3^{-3m}(w_{i,M}-x_i)\, ,
	\end{equation}
	componentwise in $\cu_M$. Multiplying \eqref{eq.z-pde} by $2$ puts it in the form \eqref{eq.drift-bound-pde} of Lemma~\ref{lem.drift-bound}.

		Set
		\[
			N_z=\rho^{1/2}3^{-m}\|z\|_{\underline L^2(\cu_M)},\qquad
			E_z^2=\tfrac12\sum_{i=1}^d\fint_{\cu_M}\nabla z_i\cdot\a\nabla z_i,\qquad
			\mathsf E_z=N_z^2+E_z^2\, ,
		\]
		and
		\[
			A_z\coloneqq\lambda\sum_{i=1}^d\fint_{\cu_M} z_i\bigl((\a-\ahom)e_1\bigr)\cdot\nabla z_i\, .
		\]

		\emph{Step 2.} We derive the energy identity \eqref{eq.energy-id}. Test \eqref{eq.z-pde} against $z_i$, integrate $\fint_{\cu_M}$, sum over $i$, and use $z_i\in H^1_0(\cu_M)$ to integrate by parts: $-\fint_{\cu_M} z_i\nabla\cdot(\a\nabla z_i)=\fint_{\cu_M}\nabla z_i\cdot\a\nabla z_i$. With the notation above,
		\[
			N_z^2+E_z^2-\lambda\sum_{i=1}^d\fint_{\cu_M} z_i(\a e_1)\cdot\nabla z_i=\sum_{i=1}^d\fint_{\cu_M} z_i T_i\, .
		\]
		Split $\a=(\a-\ahom)+\ahom$ in the drift integral. The $\ahom$-piece vanishes by integration by parts:
		\[
			\lambda\sum_{i=1}^d\fint z_i(\ahom e_1)\cdot\nabla z_i
			=\frac{\lambda}{2}\fint(\ahom e_1)\cdot\nabla|z|^2=0.
		\]
		Here $|z|^2\in W^{1,1}_0(\cu_M)$ since $z\in H^1_0(\cu_M;\R^d)$ and $\nabla|z|^2=2\sum_{i=1}^d z_i\nabla z_i\in L^1$, while $\nabla\cdot(\ahom e_1)=0$. Therefore
		\begin{equation}\label{eq.energy-id}
			N_z^2+E_z^2=A_z+R,\qquad R\coloneqq \sum_{i=1}^d\fint_{\cu_M} z_iT_i\, .
		\end{equation}

	\emph{Step 3.} We bound $|R|$. Decompose $R=R_{\mathrm{flux}}+R_w$ where
	\begin{align*}
		R_{\mathrm{flux}}&\coloneqq \lambda 3^{-m}\sum_{i=1}^d\fint_{\cu_M} z_i\bigl(e_1\cdot(\a\nabla w_{i,M}-\ahom e_i)\bigr),&
		R_w&\coloneqq -\rho 3^{-3m}\sum_{i=1}^d\fint_{\cu_M} z_i(w_{i,M}-x_i)\, .
	\end{align*}
	For $R_{\mathrm{flux}}$, the duality $\underline H^{1/6}$--$\Hminusuls{-1/6}$ together with $\sum_{i=1}^d\|\a\nabla w_{i,M}-\ahom e_i\|^2_{\Hminusuls{-1/6}}\le C3^{M/3}3^{-2\beta_h M}$ from \eqref{eq.hc} and the interpolation \eqref{eq.drift-interp} give, with $\|z\|_{\underline L^2}=\rho^{-1/2}3^m N_z$ and $\|\nabla z\|_{\underline L^2}\le\sqrt{2\Lambda}E_z$,
	\begin{align*}
		|R_{\mathrm{flux}}|&\le \lambda 3^{-m} \sqrt d 3^{M/6}3^{-\beta_h M}\cdot C (\rho^{-1/2}3^m N_z)^{5/6}(\sqrt{2\Lambda}E_z)^{1/6}\\
		&\le C\rho^{7/12} 3^{-m+h/6} 3^{-\beta_h M} N_z^{5/6}E_z^{1/6}\, .
	\end{align*}
			The inequalities $r^a s^b\le r+s\le \sqrt{2(r^2+s^2)}$ for $a+b=1$ and Young's inequality $r\sqrt{\mathsf E_z}\le \tfrac{1}{32}\mathsf E_z+8r^2$ then give
		\[
		|R_{\mathrm{flux}}|\le \tfrac{1}{32}\mathsf E_z+C\rho^{7/6} 3^{-2m+h/3} 3^{-2\beta_h M}\le \tfrac{1}{32}\mathsf E_z+C\rho^2 3^{2h-2m} 3^{-2\beta_h M}\, ,
	\]
	using $\rho^{7/6}\le \rho^2$ (since $\rho\ge 1$) and $3^{h/3}\le 3^{2h}$. For $R_w$, Cauchy--Schwarz with $\sum_{i=1}^d\|w_{i,M}-x_i\|^2_{\underline L^2}\le C3^{2M}3^{-2\beta_h M}$ from \eqref{eq.hc} gives
	\[
		|R_w|\le \rho 3^{-3m} \rho^{-1/2}3^m N_z\cdot\sqrt d 3^M3^{-\beta_h M}=\sqrt d \rho^{1/2}3^{h-m}3^{-\beta_h M} N_z\le\tfrac{1}{32}N_z^2+C\rho 3^{2h-2m}3^{-2\beta_h M}\, .
	\]
		Adding the two bounds,
		\begin{equation}\label{eq.R-bound}
			|R|\le \tfrac{1}{16}\mathsf E_z+C\rho^2 3^{2h-2m} 3^{-2\beta_h M}\, .
		\end{equation}

		\emph{Step 4.} We close the energy estimate via Lemma~\ref{lem.drift-bound}. The PDE \eqref{eq.drift-bound-pde} required by that lemma is verified via Step 1; applying it on the regime $\rho\in[1,3^{\delta m/2}]$, with $h\le\zeta_H m$ and $m\ge\X_{\beta_h}(\omega)$, gives
		\begin{equation}\label{eq.Az-bound}
			|A_z|\le \tfrac{1}{16}\mathsf E_z+C \rho^2 3^{2h-2m} 3^{-2\beta_h M}\, .
		\end{equation}
		Combining \eqref{eq.energy-id}, \eqref{eq.R-bound}, and \eqref{eq.Az-bound},
		\[
			\mathsf E_z=A_z+R\le |A_z|+|R|\le \tfrac{1}{8}\mathsf E_z+C \rho^2 3^{2h-2m} 3^{-2\beta_h M}\, ,
		\]
		whence $\mathsf E_z\le \tfrac{8}{7}C \rho^2 3^{2h-2m} 3^{-2\beta_h M}$.

		\emph{Step 5.} We conclude. The bound $\mathsf E_z\le C\rho^2 3^{2h-2m}3^{-2\beta_h M}$ gives $\|z\|_{\underline L^2(\cu_M)}\le C \rho^{1/2}3^h3^{-\beta_h M}$. Since $U_{m,i}=z_i+3^{-m}(w_{i,M}-x_i)$ and $\sum_{i=1}^d\|3^{-m}(w_{i,M}-x_i)\|^2_{\underline L^2}\le C 3^{2h}3^{-2\beta_h M}$ by \eqref{eq.hc},
		\[
			\|U_m\|_{\underline L^2(\cu_M)}\le C \rho^{1/2} 3^h 3^{-\beta_h M}\, .
		\]
		Since $M=m+h$, this gives \eqref{eq.energy-l2-main} after enlarging $C$.
	\end{proof}

The global $\underline L^2$ estimate of Proposition~\ref{pro.velocity-resolvent-L2} holds on the full cube $\cu_{m+h}$, but the passage to a pointwise bound via the interior estimate of Lemma~\ref{lem.linfty} cannot be carried out there: the constant in Lemma~\ref{lem.linfty} depends on the ellipticity ratio of the divergence-form matrix $e^{2\lambda x_1}\a$, and on $\cu_{m+h}$ the weight $e^{2\lambda x_1}$ ranges over $e^{\pm\lambda 3^{m+h}}=e^{\pm\rho 3^h}$, which grows with $\rho$ and is unbounded in the admissible regime. The next lemma resolves this by restricting to a concentric subcube $\cu_{m_1}$, smaller than $\cu_m$ by a factor $3^{-\delta_0 m}$, on which $|2\lambda x_1|\le 3$ and the tilt weight is pinned in $[e^{-3},e^{3}]$.

On $\cu_{m_1}$ the matrix $e^{2\lambda x_1}\a$ is then uniformly elliptic with constants depending only on $(d,\Lambda)$, and Lemma~\ref{lem.linfty} applies with a constant independent of $\rho$ and $h$. Passing from the global $\underline L^2(\cu_{m+h})$ estimate to a pointwise bound on the smaller cube $\cu_{m_1}$ costs the volume ratio $3^{d(m-m_1+h)/2}$. The part coming from $m-m_1$ is absorbed by the homogenization gain through the choice of $\delta$. The $h$-dependent part is the price of keeping the origin separated from the boundary, and is controlled later by choosing $h$ much smaller than $m$.
\begin{lemma}\label{lem.bounded-tilt-subcube}
	Let $\delta_0\in(0,1)$, let $m\in\N$ with $m\ge 1/(1-\delta_0)$, and let $\rho\ge 1$
	satisfy $\rho\,3^{-\delta_0 m}\le 1$. With $\lambda=\rho 3^{-m}$ and
	$m_1\coloneqq\lfloor(1-\delta_0)m\rfloor$,
	\[
		m-m_1\in[\delta_0 m,\delta_0 m+1),\qquad 1\le m_1\le m\, ,
	\]
	and on $\cu_{m_1}$,
	\[
		|2\lambda x_1|\le 3,\qquad e^{-3}\le e^{\pm 2\lambda x_1}\le e^{3}\, .
	\]
\end{lemma}

\begin{proof}
	From $m_1=\lfloor(1-\delta_0)m\rfloor\in\bigl((1-\delta_0)m-1,(1-\delta_0)m\bigr]$ we
	get $m-m_1\in[\delta_0 m,\delta_0 m+1)\subseteq[\delta_0 m-1,\delta_0 m+1]$. The
	bound $m_1\le m$ holds since $\delta_0>0$, and $m_1\ge 1$ holds since
	$(1-\delta_0)m\ge 1$ by the hypothesis $m\ge 1/(1-\delta_0)$. On $\cu_{m_1}$,
	$|x_1|\le 3^{m_1}/2$, so
	\[
		|2\lambda x_1|\le\lambda\,3^{m_1}=\rho\,3^{-(m-m_1)}\le\rho\,3^{-\delta_0 m+1}\le 3\, ,
	\]
	using $m-m_1\ge\delta_0 m-1$ and the hypothesis $\rho\,3^{-\delta_0 m}\le 1$. The
	weight bound $e^{-3}\le e^{\pm 2\lambda x_1}\le e^{3}$ follows.
\end{proof}

    Combining the previous two lemmas gives the pointwise velocity resolvent bound.

	\begin{proposition}[Localized $L^\infty$ bound for the velocity resolvent]\label{pro.velocity-resolvent-infty}
		There exists $C<\infty$ such that, for $\Q$-a.e.\ $\omega$ and every $(m,h,\rho)$ in the admissible regime~\eqref{eq.admissible}, the velocity resolvent $u_m$ satisfies
		\begin{equation}\label{eq.energy-infty-main}
			\bigl|3^{-m}u_m(0)-\ahom e_1\bigr|=|U_m(0)|\le C \rho^{1/2} 3^{(d/2+1-\beta_h)h-\delta m}\, .
		\end{equation}
	\end{proposition}

	\begin{proof}
		Throughout this proof, $m_1\coloneqq\lfloor(1-\delta)m\rfloor$, where $\delta=2\beta_h/(d+4)$ as in \eqref{eq.eta-delta-def}. We may assume $m\ge 1/(1-\delta)$. The hypothesis $\rho\le 3^{\delta m/2}$ gives $\rho 3^{-\delta m}\le 3^{-\delta m/2}\le 1$, which is the input required by Lemma~\ref{lem.bounded-tilt-subcube}. Applied with $\delta_0=\delta$, that lemma gives, for $m\ge\X_{\beta_h}(\omega)$, $\cu_{m_1}\subseteq\cu_{m+h}$ and
		\begin{equation}\label{eq.tilt-bdd}
			e^{-3}\le e^{\pm 2\lambda x_1}\le e^{3}\qquad\text{on }\cu_{m_1}\, .
		\end{equation}

		\smallskip

		\emph{Step 1.} We construct a local corrector $\psi$ on $\cu_{m_1}$. The weighted operator $-\nabla\cdot(e^{2\lambda x_1}\a\nabla \cdot)$ on $H^1_0(\cu_{m_1})$ is uniformly elliptic with weight $e^{2\lambda x_1}\in[e^{-2\lambda 3^{m_1}/2},e^{2\lambda 3^{m_1}/2}]$ on $\cu_{m_1}$. Consider the boundary value problem
		\begin{equation}\label{eq.psi-local}
			-\frac12 e^{-2\lambda x_1}\nabla\cdot(e^{2\lambda x_1}\a \nabla\psi_j)=\frac12 e^{-2\lambda x_1}\nabla\cdot(e^{2\lambda x_1}\a e_j)\quad\text{in }\cu_{m_1},\quad \psi_j=0\text{ on }\partial\cu_{m_1}\, .
		\end{equation}
		Then $\psi+x$ satisfies $-\nabla\cdot(e^{2\lambda x_1}\a\nabla(\psi+x))=0$ in $\cu_{m_1}$ and $\psi+x=x$ on $\partial\cu_{m_1}$. Applied componentwise, the maximum principle for the elliptic operator $-\nabla\cdot(e^{2\lambda x_1}\a\nabla \cdot)$ gives
		\[
			\|\psi_j+x_j\|_{L^\infty(\cu_{m_1})}\le\sup_{x\in\partial\cu_{m_1}}|x_j|\le \frac12 3^{m_1}
			\qquad\text{for every }1\le j\le d\, .
		\]
		Therefore,
		\begin{equation}\label{eq.psi-infty}
			\|\psi\|_{L^\infty(\cu_{m_1})}\le\|\psi+x\|_{L^\infty(\cu_{m_1})}+\|x\|_{L^\infty(\cu_{m_1})}\le \tfrac12\sqrt d 3^{m_1}+\tfrac12\sqrt d 3^{m_1}\le \sqrt d 3^{m_1}\, .
		\end{equation}
		\emph{Step 2.} We derive the equation satisfied by $r_m\coloneqq U_m-3^{-m}\psi$ on $\cu_{m_1}$. Combining \eqref{eq.resolvent-equation-Um} with \eqref{eq.psi-local},
		\[
		\rho 3^{-2m}r_m-\frac12 e^{-2\lambda x_1}\nabla\cdot\bigl(e^{2\lambda x_1}\a\nabla r_m\bigr)=-\rho 3^{-2m}\bigl(\ahom e_1+3^{-m}\psi\bigr)\quad\text{in }\cu_{m_1}\, .
		\]
		Multiplying both sides by $2e^{2\lambda x_1}$ produces the divergence form
		\[
		c(x)r_m-\nabla\cdot\bigl(e^{2\lambda x_1}\a\nabla r_m\bigr)=f(x)\quad\text{in }\cu_{m_1}\, ,
		\]
		where, by \eqref{eq.tilt-bdd},
		\[
		c(x)\coloneqq 2\rho 3^{-2m}e^{2\lambda x_1},\qquad \|c\|_{L^\infty(\cu_{m_1})}\le 2e^3\rho 3^{-2m}\, ,
		\]
		\[
		f(x)\coloneqq -2\rho 3^{-2m}e^{2\lambda x_1}\bigl(\ahom e_1+3^{-m}\psi\bigr),\qquad \|f\|_{L^\infty(\cu_{m_1})}\le C\rho 3^{-2m}\, .
		\]
		Since $\rho 3^{-\delta m}\le 1$ and $\delta<1$, $\|c\|_{L^\infty(\cu_{m_1})}\le 2e^3 3^{-(2-\delta)m}\le 1\le e^3\Lambda$ for $m$ large enough. The matrix $e^{2\lambda x_1}\a$ is uniformly elliptic on $\cu_{m_1}$ with constants $(e^{-3}\Lambda^{-1},e^3\Lambda)$, and Lemma~\ref{lem.linfty} applies with $e^3\Lambda$ in place of $\Lambda$.

		\emph{Step 3.} We bound the~$L^2$ norm of~$r_m$. By \eqref{eq.psi-infty}, $\|3^{-m}\psi\|_{\underline L^2(\cu_{m_1})}\le \sqrt d 3^{-(m-m_1)}\le 3\sqrt d\cdot 3^{-\delta m}$. From $\cu_{m_1}\subseteq\cu_{m+h}$, $|\cu_{m+h}|/|\cu_{m_1}|=3^{d(m-m_1+h)}$, and Proposition~\ref{pro.velocity-resolvent-L2},
		\[
		\|U_m\|_{\underline L^2(\cu_{m_1})}\le 3^{d(m-m_1+h)/2}\|U_m\|_{\underline L^2(\cu_{m+h})}\le C \rho^{1/2} 3^{d(m-m_1+h)/2+h-\beta_h(m+h)}\, .
		\]
		Using $m-m_1\le\delta m+1$ and $\delta=2\beta_h/(d+4)$,
		\[
		\tfrac{d}{2}(m-m_1)-\beta_h m\le\tfrac{d\beta_h m}{d+4}-\beta_h m=-\tfrac{4\beta_h m}{d+4}=-2\delta m\le-\delta m\, ,
		\]
		up to an additive constant absorbed into $C$. The coefficient $d/2+1-\beta_h$ is positive since $\beta_h<1$. Hence
		\[
			\|U_m\|_{\underline L^2(\cu_{m_1})}\le C \rho^{1/2} 3^{(d/2+1-\beta_h)h-\delta m}\, .
		\]
		By the triangle inequality and $3^{-m}\|\psi\|_{L^\infty(\cu_{m_1})}\le 3\sqrt d\cdot 3^{-\delta m}$,
		\begin{equation}\label{eq.rm-l2}
			\|r_m\|_{\underline L^2(\cu_{m_1})}\le\|U_m\|_{\underline L^2(\cu_{m_1})}+3\sqrt d\cdot 3^{-\delta m}\le C \rho^{1/2} 3^{(d/2+1-\beta_h)h-\delta m}\, ,
		\end{equation}
		using $\rho\ge 1$ and $h\ge 0$ to absorb the additive $3\sqrt d\cdot 3^{-\delta m}$ into the leading term.

		\emph{Step 4.} Lemma~\ref{lem.linfty} applied componentwise to $r_m$ with $M=m_1$, with the elliptic matrix $e^{2\lambda x_1}\a$ and the $c,f$ bounds of Step~2, gives
		\[
		|r_m(0)|\le C\bigl(\|r_m\|_{\underline L^2(\cu_{m_1})}+3^{2m_1}\|f\|_{L^\infty(\cu_{m_1})}\bigr)\, .
		\]
		With $3^{2m_1}\|f\|_{L^\infty}\le C\rho 3^{2(m_1-m)}\le C\rho^{1/2}3^{-\delta m}$ (using $\rho 3^{-\delta m}\le 1$ and $m-m_1\ge\delta m$), the source contribution is dominated by \eqref{eq.rm-l2}. The triangle inequality and $3^{-m}\|\psi\|_{L^\infty}\le 3\sqrt d\cdot 3^{-\delta m}$ then give \eqref{eq.energy-infty-main}.
	\end{proof}

    \subsection{Clock resolvent}\label{ss.clock-resolvent}
    
    We next estimate the clock resolvent $q_m$ defined in \eqref{eq.clock-resolvent}. It is used in Section~\ref{sec.einstein-rate} to control the clock rate $\eta(\lambda)-\E[e^{-2V(0)}]$. Here the finite-volume corrector estimates are replaced by the negative-Sobolev estimate of Lemma~\ref{lem.clock-negnorm}.
    
    By the maximum principle for the elliptic operator $\rho 3^{-2m}\,\mathrm{Id}-\mathcal L^\lambda$
    on $\cu_{m+h}$ with zero boundary data,
    \begin{equation}\label{eq.qm-Linfty-apriori}
    	\|q_m\|_{L^\infty(\cu_{m+h})}\le \|e^{-2V}-\E[e^{-2V(0)}]\|_{L^\infty(\cu_{m+h})}\le 2\Lambda\, .
    \end{equation}

		\begin{proposition}[Global $\underline L^2$ bound for the clock resolvent]\label{pro.clock-resolvent-L2}
			There exists $C<\infty$ such that, for $\Q$-a.e.\ $\omega$ and every $m,h\in\N$ and $\rho\ge 1$ with $m\ge\X_\mu(\omega)$ and $\lambda=\rho 3^{-m}$, the clock resolvent $q_m$ defined by~\eqref{eq.clock-resolvent} satisfies
			\begin{equation}\label{eq.clock-l2-main}
			\|q_m\|_{\underline L^2(\cu_{m+h})}\le C\rho^{13/12}3^{(7/12-\beta_h/2)h-(\beta_h/2)m}\, .
		\end{equation}
	\end{proposition}

	\begin{proof}
			Set $M\coloneqq m+h$. Test \eqref{eq.clock-resolvent} against $q_m\in H^1_0(\cu_M)$ and integrate by parts using $q_m=0$ on $\partial\cu_M$:
			\begin{equation}\label{eq.clock-l2-test}
			\rho 3^{-2m}\|q_m\|^2_{\underline L^2(\cu_M)}+\tfrac12\fint_{\cu_M}\nabla q_m\cdot\a\nabla q_m=\rho 3^{-2m}\fint_{\cu_M}\bigl(e^{-2V}-\E[e^{-2V(0)}]\bigr)q_m+\lambda\fint_{\cu_M}q_m (\a e_1)\cdot\nabla q_m\, .
		\end{equation}

		\emph{Step 1.}  Uniform ellipticity gives $\tfrac12\fint\nabla q_m\cdot\a\nabla q_m\ge\tfrac{1}{2\Lambda}\|\nabla q_m\|_{\underline L^2}^2$. The pointwise bound \eqref{eq.qm-Linfty-apriori} on $q_m$, together with the bounds $\|e^{-2V}-\E[e^{-2V(0)}]\|_\infty\le 2\Lambda$ and $|\a e_1|\le\Lambda$ and Cauchy--Schwarz, gives
		\[
			\Bigl|\rho 3^{-2m}\fint\bigl(e^{-2V}-\E[e^{-2V(0)}]\bigr) q_m\Bigr|\le C\rho 3^{-2m},\qquad\Bigl|\lambda\fint q_m (\a e_1)\cdot\nabla q_m\Bigr|\le C\rho 3^{-m}\|\nabla q_m\|_{\underline L^2}\, .
		\]
			Using the inequality
			\[
				C\rho 3^{-m}\|\nabla q_m\|_{\underline L^2}\le \tfrac{1}{4\Lambda}\|\nabla q_m\|_{\underline L^2}^2+C\rho^2 3^{-2m}
			\]
				and inserting the preceding bounds into \eqref{eq.clock-l2-test} yields the gradient bound in the next display. The maximum principle bound~\eqref{eq.qm-Linfty-apriori} gives the first bound in \eqref{eq.qm-grad-apriori}. Thus
		\begin{equation}\label{eq.qm-grad-apriori}
			\|q_m\|_{\underline L^2(\cu_M)}\le C,\qquad \|\nabla q_m\|_{\underline L^2(\cu_M)}\le C \rho 3^{-m}\, .
		\end{equation}

			\emph{Step 2.} We refine the energy identity. Split the drift term in \eqref{eq.clock-l2-test} using $\a=\ahom+(\a-\ahom)$. The $\ahom$-piece vanishes:
		\[
			\lambda\fint_{\cu_M}q_m (\ahom e_1)\cdot\nabla q_m=\tfrac{\lambda}{2}\fint_{\cu_M}(\ahom e_1)\cdot\nabla(q_m^2)=0\, ,
		\]
		since $q_m^2\in W^{1,1}_0(\cu_M)$ and $\ahom e_1$ is a constant vector. Combined with uniform ellipticity,
			\begin{equation}\label{eq.clock-l2-test-simplified}
			\rho 3^{-2m}\|q_m\|^2_{\underline L^2(\cu_M)}+\tfrac{1}{2\Lambda}\|\nabla q_m\|^2_{\underline L^2(\cu_M)}\le |(I)|+|(II)|\, ,
		\end{equation}
		where, using symmetry of $\a-\ahom$,
		\[
			(I)\coloneqq \rho 3^{-2m}\fint_{\cu_M}\bigl(e^{-2V}-\E[e^{-2V(0)}]\bigr)q_m,\qquad (II)\coloneqq \lambda\fint_{\cu_M}q_m ((\a-\ahom)\nabla q_m)\cdot e_1\, .
		\]

		\emph{Step 3.} We bound $\|q_m\|_{\underline H^{1/6}(\cu_M)}$. Since $q_m\in H^1_0(\cu_M)$, interpolation between $\underline L^2(\cu_M)$ and $\underline H^1_0(\cu_M)$, together with the volume-normalized scaling \eqref{eq.pullback-identities}, gives
		\[
			\|q_m\|_{\underline H^{1/6}(\cu_M)}\le C\|q_m\|_{\underline L^2(\cu_M)}^{5/6}\|\nabla q_m\|_{\underline L^2(\cu_M)}^{1/6}\, .
		\]
		Combining with \eqref{eq.qm-grad-apriori},
		\begin{equation}\label{eq.qm-Hs}
			\|q_m\|_{\underline H^{1/6}(\cu_M)}\le C\rho^{1/6}3^{-m/6}\, .
		\end{equation}

		\emph{Step 4.} We estimate the source term $(I)$. Since $m\ge \X_\mu(\omega)$ and $M=m+h$, \eqref{eq.m-mu} gives $\|e^{-2V}-\E[e^{-2V(0)}]\|_{\Hminusuls{-1/6}(\cu_M)}\le 3^{M/24}$. By the duality $\Hminusuls{-1/6}$--$\underline H^{1/6}$,
			\begin{equation}\label{eq.clock-I}
		\begin{aligned}
			|(I)|&\le \rho 3^{-2m}\|e^{-2V}-\E[e^{-2V(0)}]\|_{\Hminusuls{-1/6}(\cu_M)} \|q_m\|_{\underline H^{1/6}(\cu_M)}\\
			&\le C\rho^{7/6}3^{M/24-13m/6}=C\rho^{7/6}3^{h/24-17m/8}\, .
		\end{aligned}
		\end{equation}

			\emph{Step 5.} We estimate the drift term $(II)$. Multiplying \eqref{eq.clock-resolvent} by $-2$ rearranges to the divergence form
		\[
			-\nabla\cdot(\a\nabla q_m)=2\rho 3^{-2m}\bigl(e^{-2V}-\E[e^{-2V(0)}]-q_m\bigr)+2\lambda (\a e_1)\cdot\nabla q_m\qquad\text{in }\cu_M\, .
		\]
		Apply Lemma~\ref{lem.divinverses}(i) to the source $g\coloneqq 2\rho 3^{-2m}\bigl(e^{-2V}-\E[e^{-2V(0)}]-q_m\bigr)+2\lambda(\a e_1)\cdot\nabla q_m$ on $\cu_M$ to produce $f\in H^1(\cu_M;\R^d)$ with $\nabla\cdot f=g$ and $3^{M/6}[f]_{\underline H^{1/6}(\cu_M)}\le C 3^M\|g\|_{\underline L^2(\cu_M)}$. The source is bounded by
		\[
			\|g\|_{\underline L^2(\cu_M)}\le 2\rho 3^{-2m}(\|e^{-2V}-\E[e^{-2V(0)}]\|_\infty+\|q_m\|_\infty)+2\lambda \Lambda \|\nabla q_m\|_{\underline L^2(\cu_M)}\le C \rho^2 3^{-2m}\, ,
		\]
		using \eqref{eq.qm-grad-apriori}, $\rho\ge1$, and $\lambda=\rho 3^{-m}$. Hence
			\begin{equation}\label{eq.f-Hs-clock}
			3^{M/6}[f]_{\underline H^{1/6}(\cu_M)}\le C \rho^2 3^{M-2m}=C \rho^2 3^{h-m}\, .
		\end{equation}
			Since $M\ge \X_{\beta_h}(\omega)$ by the definition of $\X_\mu$, the estimate \eqref{eq.cg-RHS} applied to $q_m$ with the right-hand side $f$ constructed above and estimated in \eqref{eq.f-Hs-clock}, combined with \eqref{eq.qm-grad-apriori}, gives
		\[
			\|(\a-\ahom)\nabla q_m\|_{\Hminusuls{-1/6}(\cu_M)}\le C 3^{(1/6-\beta_h)M}\bigl(\rho 3^{-m}+\rho^2 3^{h-m}\bigr)\le C\rho^2 3^{(1/6-\beta_h)M+h-m}\, .
		\]
		By duality, $\|q_m e_1\|_{\underline H^{1/6}(\cu_M)}\le C\|q_m\|_{\underline H^{1/6}(\cu_M)}$, and \eqref{eq.qm-Hs},
		\[
			|(II)|\le \lambda \|(\a-\ahom)\nabla q_m\|_{\Hminusuls{-1/6}(\cu_M)} \|q_m e_1\|_{\underline H^{1/6}(\cu_M)}\le C\rho^{19/6}3^{(7/6-\beta_h)h-(2+\beta_h)m}\, .
		\]

			\emph{Step 6.} We combine. Inserting \eqref{eq.clock-I} and the bound on $|(II)|$ into \eqref{eq.clock-l2-test-simplified},
		\[
			\rho 3^{-2m}\|q_m\|^2_{\underline L^2(\cu_M)}\le C\rho^{7/6}3^{h/24-17m/8}+C\rho^{19/6}3^{(7/6-\beta_h)h-(2+\beta_h)m}\, .
		\]
			Dividing by $\rho 3^{-2m}$ and using $\rho\ge1$, we find that the exponent of $3^{-m}$ is at most $-\min\{1/8,\beta_h\}=-\beta_h$, while the $h$-coefficient is at most $7/6-\beta_h$ since $\beta_h\le1/8$; hence
		\[
			\|q_m\|^2_{\underline L^2(\cu_M)}\le C\rho^{13/6}3^{(7/6-\beta_h)h-\beta_h m}\, ,
		\]
			with $C$ deterministic. Taking square roots gives \eqref{eq.clock-l2-main}.
	\end{proof}

		\begin{proposition}[Localized $L^\infty$ bound for the clock resolvent]\label{pro.clock-resolvent-infty}
			There exists $C<\infty$ such that, for $\Q$-a.e.\ $\omega$ and every $m,h\in\N$ and $\rho\ge 1$ with $m\ge\X_\mu(\omega)$, $\rho 3^{-\delta m/2}\le 1$, and $\lambda=\rho 3^{-m}$, the clock resolvent $q_m$ defined by~\eqref{eq.clock-resolvent} satisfies
			\begin{equation}\label{eq.clock-energy-infty-main}
			|q_m(0)|\le C\rho^{13/12}3^{(d/2+7/12-\beta_h/2)h-\delta m}\, .
		\end{equation}
	\end{proposition}

	\begin{proof}
		Set $m_1\coloneqq \lfloor(1-\delta/2)m\rfloor$.  Lemma~\ref{lem.bounded-tilt-subcube} applied with $\delta_0=\delta/2$ gives, for $m\ge\X_\mu(\omega)$, $\cu_{m_1}\subseteq\cu_{m+h}$ and
		\begin{equation}\label{eq.tilt-bdd-sc}
			e^{-3}\le e^{\pm 2\lambda x_1}\le e^{3}\qquad\text{on }\cu_{m_1}\, .
		\end{equation}

		\smallskip

			\emph{Step 1.} We rewrite \eqref{eq.clock-resolvent} in divergence form on $\cu_{m_1}$. Multiplying by $2e^{2\lambda x_1}$,
		\[
		c(x)q_m-\nabla\cdot\bigl(e^{2\lambda x_1}\a\nabla q_m\bigr)=f(x)\quad\text{in }\cu_{m_1}\subseteq\cu_{m+h}\, ,
		\]
		where, by \eqref{eq.tilt-bdd-sc},
		\[
		c(x)\coloneqq 2\rho 3^{-2m}e^{2\lambda x_1},\qquad \|c\|_{L^\infty(\cu_{m_1})}\le 2e^3\rho 3^{-2m}\, ,
		\]
		\[
		f(x)\coloneqq 2\rho 3^{-2m}e^{2\lambda x_1} \bigl(e^{-2V(x)}-\E[e^{-2V(0)}]\bigr),\qquad \|f\|_{L^\infty(\cu_{m_1})}\le C\rho 3^{-2m}\, .
		\]
		Since $\rho 3^{-\delta m/2}\le 1$ and $\delta<2$, $\|c\|_{L^\infty(\cu_{m_1})}\le 2e^3 3^{-(2-\delta/2)m}\le 1\le e^3\Lambda$ for $m$ large. The matrix $e^{2\lambda x_1}\a$ is uniformly elliptic on $\cu_{m_1}$ with constants $(e^{-3}\Lambda^{-1},e^3\Lambda)$, and Lemma~\ref{lem.linfty} applies with $e^3\Lambda$ in place of $\Lambda$.

			\emph{Step 2.} We control $\|q_m\|_{\underline L^2(\cu_{m_1})}$. From $\cu_{m_1}\subseteq\cu_{m+h}$, $|\cu_{m+h}|/|\cu_{m_1}|=3^{d(m-m_1+h)}$, and Proposition~\ref{pro.clock-resolvent-L2},
		\[
		\|q_m\|_{\underline L^2(\cu_{m_1})}\le 3^{d(m-m_1+h)/2}\|q_m\|_{\underline L^2(\cu_{m+h})}\le C\rho^{13/12}3^{d(m-m_1+h)/2+(7/12-\beta_h/2)h-(\beta_h/2)m}\, .
		\]
			Using $m-m_1\le\delta m/2+1=\beta_h m/(d+4)+1$ and $\delta=2\beta_h/(d+4)$, so that $\tfrac{d\delta}{4}m-(\beta_h/2)m=-\delta m$, the additive constant $1$ contributes a multiplicative factor absorbed into $C$, giving
		\begin{equation}\label{eq.qm-l2-mone}
			\|q_m\|_{\underline L^2(\cu_{m_1})}\le C\rho^{13/12}3^{(d/2+7/12-\beta_h/2)h-\delta m}\, .
		\end{equation}

		\emph{Step 3.} We apply Lemma~\ref{lem.linfty} at the origin with $M=m_1$, with the elliptic matrix $e^{2\lambda x_1}\a$ and the $c,f$ bounds of Step~1:
		\[
		|q_m(0)|\le C\bigl(\|q_m\|_{\underline L^2(\cu_{m_1})}+3^{2m_1}\|f\|_{L^\infty(\cu_{m_1})}\bigr)\, .
		\]
			With $3^{2m_1}\|f\|_{L^\infty}\le C\rho 3^{2(m_1-m)}\le C\rho 3^{-\delta m}$ (using $m-m_1\ge\delta m/2$) and $\rho\le\rho^{13/12}$ for $\rho\ge 1$, the source contribution is dominated by \eqref{eq.qm-l2-mone}, yielding \eqref{eq.clock-energy-infty-main}.
	\end{proof}

	\subsection{Exit-time estimate}\label{ss.exit-time}

	The last input for the Feynman--Kac decomposition is an exit estimate. We first derive a maximal displacement bound for the symmetric diffusion from Aronson's heat-kernel estimate, and then extend it to the tilted diffusion via Girsanov's theorem on a time interval $[0,T_0]$ with $T_0\asymp\rho^{-1}3^{2m+h}$.

	\begin{lemma}[Displacement tail for the symmetric diffusion]\label{lem.sym-exit}
		Fix an environment $\omega$. Let $Z_t$ be a Feller diffusion with generator $\tfrac12\nabla\cdot(\a\nabla)$, where $\a$ is symmetric, uniformly elliptic, and Lipschitz with constants $(\Lambda^{-1},\Lambda)$. There exist constants $c,C>0$ depending only on $(d,\Lambda)$ such that, for every $R>0$ and $T>0$ and every $x\in\R^d$,
		\begin{equation}\label{eq.sym-max-exit}
			P_x^{0,\omega} \left(\sup_{0\le s\le T}|Z_s-x|\ge R\right)\le C\exp \left(-c\frac{R^2}{T}\right)\, .
		\end{equation}
	\end{lemma}

	\begin{proof}
	Apply the fundamental-solution estimate of \citet[Theorem~1, p.~891]{Aronson1967} to the equation $\partial_t v=\tfrac12\nabla\cdot(\a\nabla v)$. Since the coefficient matrix $\tfrac12\a$ has ellipticity constants depending only on $\Lambda$, the transition density $p^\omega(t,x,y)$ of $Z$ with respect to Lebesgue measure satisfies
		\begin{equation}\label{eq.aronson-symmetric}
			p^\omega(t,x,y)\le C t^{-d/2}\exp \left(-\frac{|x-y|^2}{C t}\right),\qquad t>0\, ,
		\end{equation}
		with $C<\infty$ depending only on $d$ and the ellipticity ratio $\Lambda$ (the factor $\tfrac12$ in the generator only changes $C$). In particular, integrating \eqref{eq.aronson-symmetric} over $\R^d\setminus B_r(x)$ gives
		\begin{equation}\label{eq.one-time-tail}
			P_x^{0,\omega}(|Z_t-x|\ge r)\le C\exp \left(-c\frac{r^2}{t}\right),\qquad r,t>0\, .
		\end{equation}

		Let $\tau_R\coloneqq \inf\{s\ge0:\ |Z_s-x|\ge R\}$. There exists $\vartheta_0=\vartheta_0(d,\Lambda)>0$ such that for $T>\vartheta_0 R^2$ the conclusion \eqref{eq.sym-max-exit} is trivial after enlarging $C$ by $\exp(c/\vartheta_0)$, so we may assume $T\le\vartheta_0 R^2$. For every $y$ with $|y-x|=R$ and every $z\in B_{R/2}(x)$, $|z-y|\ge R/2$. Applying \eqref{eq.aronson-symmetric} and integrating over $B_{R/2}(x)$, for every $0<u\le T\le\vartheta_0 R^2$,
		\[
		P_y^{0,\omega}(|Z_u-x|<R/2)\le C\left(\frac{R^2}{u}\right)^{d/2}\exp \left(-c\frac{R^2}{u}\right)\le C\exp(-c/\vartheta_0)\, ,
		\]
			using $r^{d/2}e^{-cr/2}\le C(d,c)$ for $r>0$, applied with $r=R^2/u$, in the last step. Choose $\vartheta_0$ so that $C\exp(-c/\vartheta_0)\le 1/2$. By continuity of paths and the strong Markov property at $\tau_R$,
		\[
		P_x^{0,\omega}(|Z_T-x|\ge R/2)
		\ge \frac12 P_x^{0,\omega}(\tau_R\le T)\, .
		\]
		Combining this with \eqref{eq.one-time-tail} applied with $r=R/2$ yields \eqref{eq.sym-max-exit}.
	\end{proof}

	We now transfer this displacement bound to the tilted diffusion.

	\begin{proposition}[Laplace transform of the exit time]\label{pro.exit}
		Let $\a$ be a symmetric uniformly elliptic Lipschitz coefficient field with constants $(\Lambda^{-1},\Lambda)$, and let $\tau_{m,h}$ be the exit time defined in~\eqref{eq.tau-mh}. For every $\kappa\in(0,1]$ there exist $c_\kappa,C_\kappa>0$ depending only on $(\kappa,d,\Lambda)$ such that, for every $\rho\ge 1$ and $m,h\in\N$ with $\lambda=\rho 3^{-m}$,
		\[
			E^{\lambda,\omega}_0\bigl[\exp(-\kappa\rho 3^{-2m}\tau_{m,h})\bigr]\le C_\kappa\exp(-c_\kappa 3^h)\, .
		\]
	\end{proposition}

	\begin{proof}
		Pick a small $\epsilon_0=\epsilon_0(d,\Lambda)>0$ to be fixed below, and set $T_0\coloneqq (\epsilon_0/2)3^{2m+h}/\rho$. With this choice $\kappa\rho 3^{-2m}T_0=(\kappa\epsilon_0/2)3^h$. Let $\mathsf C\coloneqq C(\R_+;\R^d)$ be the canonical path space with coordinate process $Z_s(\xi)\coloneqq\xi(s)$ and canonical filtration. Let $P^{0,\omega}_0$ denote the law on $\mathsf C$ of the symmetric reversible diffusion with generator $\tfrac12\nabla\cdot(\a\nabla)$, and $P^{\lambda,\omega}_0$ the law of the tilted diffusion with generator $\tfrac12\nabla\cdot(\a\nabla)+\lambda\a e_1\cdot\nabla$. Set
		\[
		\tau\coloneqq \inf\{t\ge0:\ Z_t\notin\cu_{m+h}\}\, ,
		\]
		a stopping time for the canonical filtration; $\tau$ has the same distribution as $\tau_{m,h}$ under $P^{\lambda,\omega}_0$. Under $P^{0,\omega}_0$, the canonical martingale $M_s\coloneqq Z_s-\tfrac12\int_0^s\nabla\cdot\a(Z_u) du$ is a continuous local martingale with quadratic variation $\langle M^i,M^j\rangle_s=\int_0^s\a_{ij}(Z_u) du$. Since exit from $\cu_{m+h}$ before time $T_0$ implies $\sup_{0\le s\le T_0}|Z_s|\ge 3^{m+h}/2$, Lemma~\ref{lem.sym-exit} applied with $R=3^{m+h}/2$ gives
		\begin{equation}\label{eq.sym-exit}
			P^{0,\omega}_0(\tau\le T_0)\le C\exp \left(-c\frac{(3^{m+h}/2)^2}{T_0}\right)\le C\exp \left(-\frac{c}{2\epsilon_0}\rho 3^h\right)\, .
		\end{equation}
		By Girsanov's theorem on canonical path space, $P^{\lambda,\omega}_0$ is absolutely continuous on $\mathcal F_{T_0}$ with respect to $P^{0,\omega}_0$, with density
		\[
		D_{T_0}=\exp \left(\lambda\int_0^{T_0}e_1\cdot dM_s-\tfrac{\lambda^2}{2}\int_0^{T_0}e_1\cdot\a(Z_s)e_1 ds\right)\, .
		\]
		Set $N_t\coloneqq \lambda\int_0^t e_1\cdot dM_s$, so that $D_{T_0}=\exp(N_{T_0}-\tfrac12\langle N\rangle_{T_0})$. Since $e_1\cdot\a e_1\le\Lambda$, the quadratic variation $\langle N\rangle_{T_0}=\lambda^2\int_0^{T_0}e_1\cdot\a(Z_s)e_1 ds$ is bounded deterministically by $\Lambda\lambda^2T_0$. Moreover,
		\[
			D_{T_0}^2=\exp(2N_{T_0}-2\langle N\rangle_{T_0})\exp(\langle N\rangle_{T_0})\, .
		\]
		The first factor is a nonnegative $P^{0,\omega}_0$-local martingale and therefore a supermartingale of expectation at most $1$. The second factor is bounded by $\exp(\Lambda\lambda^2T_0)$, so $E^{0,\omega}_0[D_{T_0}^2]\le\exp(\Lambda\lambda^2T_0)=\exp(C\epsilon_0\rho 3^h)$. By Cauchy--Schwarz and \eqref{eq.sym-exit},
		\[
		P^{\lambda,\omega}_0(\tau\le T_0)\le E^{0,\omega}_0[D_{T_0}^2]^{1/2} P^{0,\omega}_0(\tau\le T_0)^{1/2}\le C\exp \left(\tfrac{C\epsilon_0}{2}\rho 3^h-\tfrac{c}{4\epsilon_0}\rho 3^h\right)\, .
		\]
		Choose $\epsilon_0$ so small that $c/(4\epsilon_0)-C\epsilon_0/2\ge c'>0$; then $P^{\lambda,\omega}_0(\tau\le T_0)\le C\exp(-c'\rho 3^h)$. Consequently,
		\[
		E^{\lambda,\omega}_0[e^{-\kappa\rho3^{-2m}\tau_{m,h}}]\le P^{\lambda,\omega}_0(\tau\le T_0)+e^{-\kappa\rho3^{-2m} T_0}\le C_\kappa e^{-c_\kappa 3^h}\, .
		\]
	\end{proof}

	\section{Renewal estimates}\label{sec.renewal}

 This section proves a quenched law of large numbers for the time-changed diffusion $X^\lambda$ that is uniform in the tilt parameter on a random window $\lambda\in(0,\lambda_0(\omega)]$. The main result is Corollary~\ref{cor.uniform-renewal}, which controls both the displacement error $E_0^{\lambda,\omega}|X^\lambda(t)-\ell(\lambda)t|$ and the clock error $\widehat E_0^{\lambda,\omega}|A^\lambda(t)-\eta(\lambda)t|$ uniformly on that window.

The argument builds on the regeneration construction of \citet*{ballistic2003shen}, in the $\lambda$-dependent form of \citet[Section~5]{Einstein2012gantert}. For a fixed $\lambda$, the construction cuts the path of $X^\lambda$ at random times $0=\tau_0<\tau_1<\tau_2<\cdots$ into blocks whose displacement increments $\Delta X_k^\lambda$ are i.i.d.\ and whose centered versions $\Delta X_k^\lambda-\ell(\lambda)\Delta\tau_k$ have mean zero. Each block has duration of order $\lambda^{-2}$ and displacement of order $\lambda^{-1}$. A law of large numbers at fixed $\lambda$ follows from the i.i.d.\ structure; this is \citet[Proposition~5.9]{Einstein2012gantert}. The block moment bounds, the exponential tail of $\tau_1$, and the linear velocity bound $|\ell(\lambda)|\le C\lambda$ are also from \citet[Section~4--5]{Einstein2012gantert} and are recorded below in Lemmas~\ref{lem.max}--\ref{lem.lbound}.

The new ingredient is uniformity in $\lambda$. The annealed moment bound is already uniform: an interval $[0,t]$ contains about $\lambda^2 t$ blocks, each contributing a centered displacement of size $\lambda^{-1}$, so the sum fluctuates on scale $\lambda^{-1}\sqrt{\lambda^2 t}=\sqrt t$. Rosenthal's inequality and the Doob maximal inequality make this cancellation precise (Proposition~\ref{pro.moment-renewal}). The upgrade to a quenched bound uniform in $\lambda$ then proceeds by discretization. On a dyadic grid in $\lambda$, Markov's inequality turns the annealed bound into a quenched bound at each grid point that fails with rapidly summable probability; Borel--Cantelli produces a random index beyond which all grid points are good. A Girsanov second-moment comparison (Lemma~\ref{lem.girsanov-lambda}) transfers the bound from each grid point to nearby values of $\lambda$, filling the gaps.

The clock $A^\lambda$ requires an additional step. Its block increments $\Delta A_k^\lambda=\int_{\tau_k}^{\tau_{k+1}}e^{-2V(X^\lambda(s))}\,ds$ depend on the environment along the path, so the existing renewal theorems of \citet*{ballistic2003shen} and \citet[Theorem~5.6]{Einstein2012gantert}, which record only the path, do not directly apply. A renewal identity at a general regeneration time that includes the environment (Lemma~\ref{lem.conditional-renewal-tau-k}, proved in Section~\ref{sec.shen-renewal-k}) shows that the clock blocks are stationary and $2$-dependent, after which the annealed clock bound (Proposition~\ref{pro.additive-renewal}) then follows by the same Rosenthal argument.

Subsection~\ref{ss.block-structure} records the i.i.d.\ and $2$-dependent block structure from \citet*{ballistic2003shen} and \citet*[Section~5]{Einstein2012gantert}; Subsection~\ref{ss.annealed-moments} derives the annealed bounds from these inputs; and Subsection~\ref{ss.uniform-renewal} carries out the upgrade to a uniform window.

\subsection{Setup}\label{ss.renewal-setup}

The regeneration times $0=\tau_0<\tau_1<\tau_2<\cdots$ are constructed following \citet[Section~5]{Einstein2012gantert}, which adapts the method of \citet*{ballistic2003shen} to a $\lambda$-dependent spatial scale. The construction requires auxiliary randomness: we enlarge the probability space to carry the diffusion path $X^\lambda$ together with i.i.d.\ Bernoulli marks $(Y_n)_{n\ge0}$. We write $\widehat P^{\lambda,\omega}_x$, $\widehat E^{\lambda,\omega}_x$ for the quenched law and expectation on this enlarged space, and $\widehat{\mathbb P}^\lambda_x\coloneqq \int\widehat P_x^{\lambda,\omega}\,d\Q(\omega)$, $\widehat{\mathbb E}^\lambda_x$ for the annealed versions. The $X$-marginal of $\widehat P_x^{\lambda,\omega}$ equals $P_x^{\lambda,\omega}$, so for functionals of the path alone, $\widehat E_x^{\lambda,\omega}$ and $E_x^{\lambda,\omega}$ agree.

The construction depends on a deterministic integer $l=l(d,\Lambda,r_0)\in\N$ and the induced spatial scale $R(\lambda)\coloneqq l/\lambda$. We summarize the procedure of \citet[(5.8)--(5.16)]{Einstein2012gantert}. Write
\[
	M(t)\coloneqq\sup_{0\le s\le t} e_1\cdot(X^\lambda(s)-X^\lambda(0))
\]
for the running forward maximum, and let $\lceil\cdot\rceil_\lambda$ denote rounding up to the next multiple of $\lambda^{-2}$, with $\lceil\infty\rceil_\lambda=\infty$. The successive ladder times $\lceil V_k\rceil_\lambda$ are the rounded times at which $e_1\cdot(X^\lambda-X^\lambda(0))$ first exceeds $M(\lceil V_{k-1}\rceil_\lambda)+R(\lambda)$. At each ladder time, the Bernoulli variable attached to that lattice time is checked; if it equals~$1$, the candidate proceeds to the bridge step. The successful ladder time is followed by one additional time step of duration $\lambda^{-2}$, giving a candidate time $S$; during this bridge step the path is confined to a ball of radius $6R(\lambda)$ by \citet[Proposition~5.4(iii)]{Einstein2012gantert}. At $S$ the no-backtracking condition is tested: with
\[
	T_{-R(\lambda)}\coloneqq\inf\{t\ge0:e_1\cdot(X^\lambda(t)-X^\lambda(0))\le-R(\lambda)\},
	\qquad D\coloneqq\lceil T_{-R(\lambda)}\rceil_\lambda\, ,
\]
the event $\{D=\infty\}$ is that the path never drops $R(\lambda)$ below its initial $e_1$-level. If $D\circ\theta_S=\infty$, then $\tau_1\coloneqq S$; if $D\circ\theta_S<\infty$, the ladder search resumes from the backtracking time $S+D\circ\theta_S$. Subsequent regeneration times satisfy $\tau_{k+1}=\tau_k+\tau_1\circ\theta_{\tau_k}$. The integer $l$ is fixed large enough that
\begin{equation}\label{eq.regen-scale-choice}
	2l>r_0\quad\text{and}\quad\widehat{\mathbb P}_0^\lambda(D=\infty)\ge c_0>0\quad
	\text{uniformly in }\lambda\in(0,1]\, ;
\end{equation}
the first condition gives finite-range decoupling at distance $2R(\lambda)$, while the second ensures that each regeneration attempt succeeds with probability bounded below uniformly in~$\lambda$.

By the construction \citet[(5.13)--(5.16)]{Einstein2012gantert}, all $\tau_k$ lie in $\lambda^{-2}\N$, and by \citet[Lemma~5.7(ii)]{Einstein2012gantert}, $\widehat{\mathbb P}_0^\lambda(\tau_k<\infty)=1$ for every $k\ge1$. The $k$-th block is the stretch of path on $[\tau_k,\tau_{k+1}]$; we record its increments
\[
	\Delta\tau_k\coloneqq\tau_{k+1}-\tau_k,\qquad
	\Delta X_k^\lambda\coloneqq X^\lambda(\tau_{k+1})-X^\lambda(\tau_k),\qquad
	\Delta A_k^\lambda\coloneqq\int_{\tau_k}^{\tau_{k+1}}e^{-2V(X^\lambda(s))}\,ds\, ,
\]
and the truncated block
\begin{equation}\label{eq.truncated-block}
	\mathcal Z_k\coloneqq\bigl(
	(X^\lambda((\tau_k+t)\wedge(\tau_{k+1}-\lambda^{-2}))-X^\lambda(\tau_k))_{t\ge0},\
	\Delta X_k^\lambda,\ \Delta\tau_k\bigr)\, ,
\end{equation}
which discards the final segment of duration $\lambda^{-2}$ from each block.

First, $\lambda^2\Delta\tau_k$ is a positive integer for every $k\ge0$:
\begin{equation}\label{eq.lattice-tau}
	\lambda^2\Delta\tau_k\in\{1,2,\ldots\}\qquad\widehat{\mathbb P}_0^\lambda\text{-a.s.}
\end{equation}
Indeed, the regeneration times lie in $\lambda^{-2}\N$ and increase strictly.
Moreover, the paragraph following \citet[(5.16)]{Einstein2012gantert} gives
$\tau_1\ge2\lambda^{-2}$, and the recursion
$\tau_{k+1}=\tau_k+\tau_1\circ\theta_{\tau_k}$ gives
$\tau_k\ge2k\lambda^{-2}$, so the deterministic slice
$\{\tau_k=n\lambda^{-2}\}$ is empty unless $n\ge2k$. Second, within each
block, the path during the last interval of duration $\lambda^{-2}$ remains in a bounded
neighborhood of its starting point:
\begin{equation}\label{eq.final-segment-bound}
	\sup_{\tau_{k+1}-\lambda^{-2}\le s\le\tau_{k+1}}
	|X^\lambda(s)-X^\lambda(\tau_{k+1}-\lambda^{-2})|\le 11R(\lambda)=\frac{11l}{\lambda}\, ,
\end{equation}
which follows from \citet[Proposition~5.4(iii)]{Einstein2012gantert}.

\subsection{Block structure}\label{ss.block-structure}

For $k\ge1$ the displacement increments are i.i.d.\ under $\widehat{\mathbb P}_0^\lambda$ (Lemma~\ref{lem.renewal}), while the clock--duration pairs $(\Delta A_k^\lambda,\Delta\tau_k)$ are stationary and $2$-dependent (Lemma~\ref{lem.two-dependent}).

\begin{lemma}[Regeneration blocks and velocity]\label{lem.renewal}
	The blocks $\{\mathcal Z_k\}_{k\ge0}$ are independent under $\widehat{\mathbb P}_0^\lambda$, and $\{\mathcal Z_k\}_{k\ge 1}$ are i.i.d.\ with common law equal to that of $\mathcal Z_0$ conditioned on $\{D=\infty\}$. The limit $\ell(\lambda)$ in \eqref{eq.def-ell-eta} exists $\widehat{\mathbb P}_0^\lambda$-almost surely and equals
		\[
			\ell(\lambda)=\frac{\widehat{\mathbb E}_0^\lambda[\Delta X_1^\lambda]}{\widehat{\mathbb E}_0^\lambda[\Delta\tau_1]}\, .
		\]
\end{lemma}

\begin{proof}
	The a.s.\ finiteness hypothesis in \citet[Theorem~5.6]{Einstein2012gantert} is supplied by \citet[Lemma~5.7(ii)]{Einstein2012gantert}. Theorem~5.6 gives exactly the independence and common conditional law of the truncated path blocks $\mathcal Z_k$ in \eqref{eq.truncated-block}. The renewal formula for the velocity is \citet[Proposition~5.9]{Einstein2012gantert} applied to the same time-changed process.
\end{proof}

For a Borel set $S\subseteq\R^d$, write $\mathcal H_S$ for the $\sigma$-algebra generated by $(\a,e^{-2V})|_S$. When $S$ depends on a random variable $Y$, the notation $\mathcal H_{\{z:e_1\cdot z\le e_1\cdot Y-4R\}}$ denotes the $\sigma$-algebra generated by bounded variables $F(Y,\omega)$ such that $F(y,\cdot)$ is $\mathcal H_{\{z:e_1\cdot z\le e_1\cdot y-4R\}}$-measurable for every $y$.

		\begin{lemma}[Conditional renewal at regeneration times]\label{lem.conditional-renewal-tau-k}
		Let $k\ge 1$ and $R=R(\lambda)$. Define
		\[
			\mathsf L_k\coloneqq\{y\in\R^d:e_1\cdot y\le e_1\cdot X^\lambda(\tau_k)-4R\}\, ,
			\qquad
			\mathsf R_+\coloneqq\{y\in\R^d:e_1\cdot y\ge -2R\}\, .
		\]
		Let $\mathcal G_k$ denote the $\sigma$-algebra generated by:
		\begin{enumerate}
		\item[\rm(a)] the time $\tau_k$, the diffusion path $(X^\lambda(s))_{0\le s\le\tau_k-\lambda^{-2}}$, and the endpoint $X^\lambda(\tau_k)$;
		\item[\rm(b)] the auxiliary Bernoulli variables $(Y_n)_{0\le n<\lambda^2(\tau_k-\lambda^{-2})}$;
		\item[\rm(c)] the environment $(\a,e^{-2V})|_{\mathsf L_k}$.
	\end{enumerate}
	Let the shifted future data be
		\[
			\mathcal Y_k\coloneqq\bigl(
				(X^\lambda(\tau_k+t)-X^\lambda(\tau_k))_{t\ge0},\
				(Y_{\lambda^2\tau_k+n})_{n\ge 0},\
				T_{X^\lambda(\tau_k)}(\a,e^{-2V})\bigr|_{\mathsf R_+}
			\bigr)\, ,
		\]
		where $T_x(\a,e^{-2V})(z)\coloneqq(\a,e^{-2V})(x+z)$, and let $\mathcal Y_0$ be given by the same formula with $k=0$ and $\tau_0=0$. Then for every bounded measurable function $\Phi$ of such triples and every bounded $\mathcal G_k$-measurable random variable $H$,
		\begin{equation}\label{eq.shen-renewal}
			\widehat{\mathbb E}_0^\lambda[\Phi(\mathcal Y_k) H]
			=\widehat{\mathbb E}_0^\lambda[\Phi(\mathcal Y_0)\mid D=\infty]\,\widehat{\mathbb E}_0^\lambda[H]\, .
		\end{equation}
		In particular, $\mathcal Y_k$ is independent of $\mathcal G_k$ under $\widehat{\mathbb P}_0^\lambda$, and the law of $\mathcal Y_k$ equals the law of $\mathcal Y_0$ under $\widehat{\mathbb P}_0^\lambda[\,\cdot\mid D=\infty]$.
		Moreover, writing $M_i\coloneqq(\Delta A_i^\lambda,\Delta\tau_i)$ for the clock-duration block:
		\begin{enumerate}
			\item[\rm(i)] for every integer $i$ with $1\le i\le k-2$, the block $M_i$ is $\mathcal G_k$-measurable;
			\item[\rm(ii)] there is a measurable map $\Psi$, defined on the common state space of the shifted future data, such that
			\[
				(M_k,M_{k+1},\ldots)=\Psi(\mathcal Y_k)
			\]
			holds $\widehat{\mathbb P}_0^\lambda$-almost surely.
		\end{enumerate}
	\end{lemma}

	In item~(a), the path is recorded only up to $\tau_k-\lambda^{-2}$, not up to $\tau_k$; the endpoint $X^\lambda(\tau_k)$ is included separately because it determines the boundary of $\mathsf L_k$ and the spatial shift in $\mathcal Y_k$. The proof first treats functions of the shifted future data that split into a path-mark factor and an environment factor, and then extends to the bounded measurable $\Phi$ stated here by a monotone-class argument. The proof of Lemma~\ref{lem.conditional-renewal-tau-k} is given in Section~\ref{sec.shen-renewal-k}.

	The displacement increments $\Delta X_k^\lambda$ are i.i.d.\ for $k\ge 1$ by Lemma~\ref{lem.renewal}. The clock-and-duration increments $(\Delta A_k^\lambda,\Delta\tau_k)$ are not i.i.d., because the post-$\tau_k$ clock contribution depends on $e^{-2V}$ on a half-space whose $\sigma$-algebra is independent of the past only at distance $2R$.

	\begin{lemma}[Clock blocks are stationary and 2-dependent]\label{lem.two-dependent}
		Let $M_k\coloneqq(\Delta A_k^\lambda,\Delta\tau_k)$ for $k\ge 1$. The sequence $\{M_k\}_{k\ge 1}$ is strictly stationary and $2$-dependent under $\widehat{\mathbb P}_0^\lambda$: for every $r\ge1$ the law of $(M_j,\ldots,M_{j+r-1})$ is independent of $j$, and for every $I,J\subseteq\{1,2,\ldots\}$ with $\mathrm{dist}(I,J)\ge 3$,
		\begin{equation}\label{eq.two-dep}
			\sigma(M_i:i\in I)\quad\text{and}\quad\sigma(M_j:j\in J)\quad\text{are independent}\, .
		\end{equation}
	\end{lemma}

	\begin{proof}
		Throughout, $\mathcal G_k$, $\mathcal Y_k$, and the measurable map $\Psi$ are as in Lemma~\ref{lem.conditional-renewal-tau-k}.

		\smallskip

		\emph{Step 1.} We prove independence of past and future at separation $3$. Fix $j\ge3$. By Lemma~\ref{lem.conditional-renewal-tau-k}(i), applied with $k=j-1$, every block $M_i$ with $1\le i\le j-3$ is $\mathcal G_{j-1}$-measurable. By Lemma~\ref{lem.conditional-renewal-tau-k}(ii), applied with $k=j-1$,
		\[
			(M_{j-1},M_j,M_{j+1},\ldots)=\Psi(\mathcal Y_{j-1})
		\]
		almost surely. Therefore $(M_j,M_{j+1},\ldots)$ is, up to a null set, a measurable function of $\mathcal Y_{j-1}$.

		\smallskip

		For bounded measurable $F$ and $G$, define $\Phi_G$ on the state space of the shifted future data by applying $G$ to the sequence obtained from $\Psi$ after deleting its first coordinate. Then
		\[
			G(M_j,M_{j+1},\ldots)=\Phi_G(\mathcal Y_{j-1})
		\]
		almost surely, and $F(M_1,\ldots,M_{j-3})$ is $\mathcal G_{j-1}$-measurable. Applying \eqref{eq.shen-renewal} at $k=j-1$ with $H=F(M_1,\ldots,M_{j-3})$ and $\Phi=\Phi_G$, and then applying the same identity with $H\equiv1$, gives
		\begin{equation}\label{eq.tail-indep}
			\widehat{\mathbb E}_0^\lambda[F(M_1,\ldots,M_{j-3}) G(M_j,M_{j+1},\ldots)]
			=\widehat{\mathbb E}_0^\lambda[F(M_1,\ldots,M_{j-3})] \widehat{\mathbb E}_0^\lambda[G(M_j,M_{j+1},\ldots)]\, ,
		\end{equation}
		so $\sigma(M_i:i\le j-3)\perp\sigma(M_j,M_{j+1},\dots)$ for every $j\ge 3$.

		\smallskip

		\emph{Step 2.} We deduce \eqref{eq.two-dep} from \eqref{eq.tail-indep}, first for finite sets. Let $I,J\subseteq\{1,2,\ldots\}$ be finite and disjoint with $\mathrm{dist}(I,J)\ge 3$. Partition $I\cup J$ into maximal runs $C_1,\ldots,C_q$ of indices in $I\cup J$ where consecutive indices differ by at most $2$. By definition, two consecutive runs are separated by a gap of at least $3$.

		No run contains both an $I$-index and a $J$-index. Indeed, otherwise traversing the run from first index to last would produce a transition from an $I$-index $i_l$ to a $J$-index $i_{l+1}$, or conversely, with $|i_{l+1}-i_l|\le 2$, contradicting $\mathrm{dist}(I,J)\ge 3$. Hence each run is contained in either $I$ or $J$.

		Iterating \eqref{eq.tail-indep} at the boundaries between runs shows that the run $\sigma$-fields $\sigma(M_i:i\in C_r)$ are mutually independent. Therefore $\sigma(M_i:i\in I)$ and $\sigma(M_j:j\in J)$ are products of disjoint subcollections of these mutually independent run $\sigma$-fields, and hence are independent.

			For general separated sets, let $\mathcal C_I$ be the cylinder $\pi$-system generated by events $\{(\Delta A_i^\lambda,\Delta\tau_i)\in B_i:i\in I_0\}$ with finite $I_0\subseteq I$, and define $\mathcal C_J$ similarly. The finite-set result gives independence on $\mathcal C_I\times\mathcal C_J$, and Dynkin's $\pi$-$\lambda$ theorem extends this to $\sigma(M_i:i\in I)$ and $\sigma(M_j:j\in J)$.

		\smallskip

		\emph{Step 3.} We prove stationarity. Fix $j\ge 1$ and $r\ge 1$. Let $\Psi_r$ be the first $r$ coordinates of the sequence-valued map $\Psi$. By Lemma~\ref{lem.conditional-renewal-tau-k}(ii),
		\[
			\Psi_r(\mathcal Y_j)=(M_j,\ldots,M_{j+r-1})
		\]
		almost surely. For every bounded measurable $g$ on $(\R\times[\lambda^{-2},\infty))^r$, Lemma~\ref{lem.conditional-renewal-tau-k} at $k=j$ with $H\equiv 1$ and $\Phi=g\circ\Psi_r$ gives
		\[
			\widehat{\mathbb E}_0^\lambda\bigl[g(M_j,\ldots,M_{j+r-1})\bigr]=\widehat{\mathbb E}_0^\lambda\bigl[g(\Psi_r(\mathcal Y_0))\mid D=\infty\bigr]\, .
		\]
		The right-hand side does not depend on $j$, so $\{M_k\}_{k\ge 1}$ is stationary.
	\end{proof}

	\subsection{Annealed moment bounds}\label{ss.annealed-moments}

	\begin{lemma}[Maximal moment bound, {\protect\citep[Lemma~4.5]{Einstein2012gantert}}]\label{lem.max}
		For every $p\ge 1$, there exists a finite constant $C=C(p,d,\Lambda)>0$ such that for every environment $\omega$ and every $\lambda\in(0,1]$,
		\begin{equation}
			E^{\lambda,\omega}_0\big[\max_{0\le s\le t}|X^\lambda(s)|^p\big]\le C(\lambda t)^p\qquad\text{for every }t\ge\lambda^{-2}\, .
		\end{equation}
	\end{lemma}

	The first regeneration time has quenched exponential tails on the diffusive timescale $\lambda^{-2}$. After annealing, the uniform lower bound on $\widehat{\mathbb P}_0^\lambda(D=\infty)$ implies the same tail for every later block.
	
	\begin{lemma}[Regeneration-time tails]\label{lem.tau-tails}
		There exist constants $C=C(d,\Lambda,r_0)>0$ and $c=c(d,\Lambda,r_0)>0$ such that for every $\lambda\in(0,1]$ and $t>0$,
		\begin{equation}\label{eq.tail-tau1}
			\sup_\omega \widehat P^{\lambda,\omega}_0\bigl(\tau_1\ge\lambda^{-2}t\bigr)\le Ce^{-ct}\, ,
		\end{equation}
		and for every $k\ge 1$,
		\begin{equation}\label{eq.tail-increment}
			\widehat{\mathbb P}^\lambda_0\bigl(\Delta\tau_k\ge\lambda^{-2}t\bigr)\le Ce^{-ct}\, .
		\end{equation}
	\end{lemma}

	\begin{proof}
		The first estimate is the quenched bound proved in \citet[Lemma~5.8, display~(5.24)]{Einstein2012gantert}. For $k\ge 1$, Lemma~\ref{lem.renewal} gives the law of $\Delta\tau_k$ as the law of $\tau_1$ conditioned on $\{D=\infty\}$. Together with \eqref{eq.regen-scale-choice},
		\[
			\widehat{\mathbb P}^\lambda_0\bigl(\Delta\tau_k\ge\lambda^{-2}t\bigr)
			\le c_0^{-1}\widehat{\mathbb P}^\lambda_0\bigl(\tau_1\ge\lambda^{-2}t\bigr)
			\le C e^{-ct}\, ,
		\]
		after increasing $C$.
	\end{proof}

	\begin{lemma}[Linear bound on the regeneration velocity, {\protect\citep[(4.25)]{Einstein2012gantert}}]\label{lem.lbound}
		There exists $C=C(d,\Lambda,r_0)<\infty$ such that for every $\lambda\in(0,1]$,
		\begin{equation}\label{eq.lbound}
			|\ell(\lambda)|\le C\lambda\, .
		\end{equation}
	\end{lemma}

	\begin{lemma}[Regeneration block moment bounds]\label{lem.block-moment-bounds}
		Let
		\[
		R_k\coloneqq\max_{0\le s\le \Delta\tau_k}|X^\lambda(\tau_k+s)-X^\lambda(\tau_k)|\, .
		\]
		For every integer $r\ge1$, there exists $C=C(r,d,\Lambda,r_0)<\infty$ such that, for every $\lambda\in(0,1]$ and every $k\ge0$,
		\begin{equation}\label{eq.RD-r}
			\widehat{\mathbb E}_0^\lambda[(\Delta\tau_k)^r]\le C \lambda^{-2r},\qquad \widehat{\mathbb E}_0^\lambda[R_k^r]\le C \lambda^{-r}\, .
		\end{equation}
		Consequently,
		\begin{equation}\label{eq.blockX-r}
			\widehat{\mathbb E}_0^\lambda|\Delta X_k^\lambda|^r\le C \lambda^{-r}\qquad\text{for }k\ge0\, ,
		\end{equation}
		and, for
		\[
		W_k=\Delta X_k^\lambda-\ell(\lambda)\Delta\tau_k\, ,
		\]
		\begin{equation}\label{eq.W-r}
			\widehat{\mathbb E}_0^\lambda|W_k|^r\le C \lambda^{-r}\qquad\text{for }k\ge0\, .
		\end{equation}
	\end{lemma}

	\begin{proof}
		\emph{Step 1.} We first treat $k=0$. Integration by parts and the quenched tail \eqref{eq.tail-tau1} give
		\[
			\sup_\omega\widehat E_0^{\lambda,\omega}[\tau_1^r]
			=\sup_\omega r\int_0^\infty u^{r-1}\widehat P_0^{\lambda,\omega}(\tau_1\ge u) du
			\le r\int_0^\infty u^{r-1}\sup_\omega\widehat P_0^{\lambda,\omega}(\tau_1\ge u) du\le C\lambda^{-2r}\, .
		\]
		For $R_0$, decompose on the events $\{\tau_1\in[\lambda^{-2}n,\lambda^{-2}(n+1))\}$ for $n\ge 0$. Cauchy--Schwarz, Lemma~\ref{lem.max} with $p=2r$ at $t=\lambda^{-2}(n+1)\ge\lambda^{-2}$, and \eqref{eq.tail-tau1} give
			\[
				\widehat E_0^{\lambda,\omega}[R_0^r]
				\le\sum_{n\ge0}\bigl(\widehat E_0^{\lambda,\omega}[\max_{s\le\lambda^{-2}(n+1)}|X^\lambda(s)|^{2r}]\bigr)^{1/2}
				\widehat P_0^{\lambda,\omega}(\tau_1\ge\lambda^{-2}n)^{1/2}\, .
			\]
			The right-hand side is at most
			\[
				C\sum_{n\ge0}\lambda^{-r}(n+1)^r e^{-cn/2}\le C\lambda^{-r}
			\]
			uniformly in $\omega$. Integrating over $\Q$ proves \eqref{eq.RD-r} for $k=0$.

		\smallskip

		\emph{Step 2.} We next treat $k\ge 1$. By Lemma~\ref{lem.renewal}, the blocks $\{\mathcal Z_k\}_{k\ge 1}$ are i.i.d.\ with the law of $\mathcal Z_0$ conditioned on $\{D=\infty\}$. Since $\widehat{\mathbb P}_0^\lambda(D=\infty)\ge c_0>0$, every nonnegative functional $F$ of $\mathcal Z_0$ satisfies
		\[
			\widehat{\mathbb E}_0^\lambda[F\mid D=\infty]\le c_0^{-1}\widehat{\mathbb E}_0^\lambda[F]\, .
		\]
		Applying this to $F=(\Delta\tau_0)^r$ gives the first bound in \eqref{eq.RD-r} for every $k\ge 1$.

		\smallskip

		\emph{Step 3.} We control the range over a block. Decompose $R_k=R_k^{\rm trunc}\vee R_k^{\rm tail}$, where $R_k^{\rm trunc}$ runs over $s\le\Delta\tau_k-\lambda^{-2}$ and $R_k^{\rm tail}$ over $s\in[\Delta\tau_k-\lambda^{-2},\Delta\tau_k]$. The truncated max is a functional of $\mathcal Z_k$ alone, so Step~2 gives $\widehat{\mathbb E}_0^\lambda[(R_k^{\rm trunc})^r]\le C\lambda^{-r}$ for $k\ge 1$. On the final segment, \eqref{eq.final-segment-bound} and the triangle inequality give
		\[
			R_k^{\rm tail}\le R_k^{\rm trunc}+\frac{11l}{\lambda}.
		\]
		Hence $R_k\le R_k^{\rm trunc}+11l/\lambda$, and the second bound in \eqref{eq.RD-r} follows.

		\smallskip

		\emph{Step 4.} It remains to bound $\Delta X_k^\lambda$ and $W_k$. The bound \eqref{eq.blockX-r} follows from $|\Delta X_k^\lambda|\le R_k$, and \eqref{eq.W-r} follows from \eqref{eq.lbound}, \eqref{eq.RD-r}, and \eqref{eq.blockX-r}.
	\end{proof}

	\begin{proposition}[Annealed polynomial moment bound for the displacement]\label{pro.moment-renewal}
		For every integer $M\ge2$, there exists $C=C(M,d,\Lambda,r_0)<\infty$ such that for every $\lambda\in(0,1]$ and every $t\ge\lambda^{-2}$,
		\begin{equation}\label{eq.moment-renewal-Mth}
			\E_\Q \widehat E^{\lambda,\omega}_0\big|X^\lambda(t)-\ell(\lambda)t\big|^M
			\le Ct^{M/2}\, .
		\end{equation}
	\end{proposition}
	
	\begin{proof}
		Let
		\[
		\mathbf n(t)\coloneqq \max\{n\ge0:\tau_n\le t\}\, .
		\]
		By the triangle inequality applied pointwise,
		\[
			|X^\lambda(t)-\ell(\lambda)t|^M\le 3^{M-1}\bigl(|X^\lambda(t)-X^\lambda(\tau_{\mathbf n(t)})|^M+|X^\lambda(\tau_{\mathbf n(t)})-\ell(\lambda)\tau_{\mathbf n(t)}|^M+|\ell(\lambda)|^M|t-\tau_{\mathbf n(t)}|^M\bigr)\, ,
		\]
		so taking $\widehat E^{\lambda,\omega}_0$ and then $\E_\Q$,
		\begin{equation}\label{eq.split-AM}
			\E_\Q \widehat E^{\lambda,\omega}_0\big|X^\lambda(t)-\ell(\lambda)t\big|^M\le 3^{M-1}\bigl(J_1(t)+J_2(t)+J_3(t)\bigr)\, ,
		\end{equation}
		with
		\begin{align*}
			J_1(t)&\coloneqq \widehat{\mathbb E}^\lambda_0\big|X^\lambda(t)-X^\lambda(\tau_{\mathbf n(t)})\big|^M\, ,\\
			J_2(t)&\coloneqq \widehat{\mathbb E}^\lambda_0\big|X^\lambda(\tau_{\mathbf n(t)})-\ell(\lambda)\tau_{\mathbf n(t)}\big|^M\, ,\\
			J_3(t)&\coloneqq |\ell(\lambda)|^M \widehat{\mathbb E}^\lambda_0\big|t-\tau_{\mathbf n(t)}\big|^M\, .
		\end{align*}

		\emph{Step 1.} We estimate $J_1$ and $J_3$ by replacing the random index $\mathbf n(t)$ with the deterministic bound $N_t+1$. Let $R_k$ and $W_k$ be as in Lemma~\ref{lem.block-moment-bounds}.

		The block moment bounds $\widehat{\mathbb E}_0^\lambda[R_k^M]\le C\lambda^{-M}$ and $\widehat{\mathbb E}_0^\lambda[(\Delta\tau_k)^M]\le C\lambda^{-2M}$ from \eqref{eq.RD-r} are used below uniformly in $k\ge 0$. Let $N_t\coloneqq\lfloor\lambda^2 t\rfloor$. The initial block is left out of this counting argument and will be included in the maxima below. The integer-valuedness of $\lambda^2\Delta\tau_k$ in \eqref{eq.lattice-tau} gives $\Delta\tau_k\ge\lambda^{-2}$ for $k\ge 1$, so
		\begin{equation}\label{eq.n-bound}
			(\mathbf n(t)-1)_+\lambda^{-2}
			\le\sum_{k=1}^{(\mathbf n(t)-1)_+}\Delta\tau_k
			\le\tau_{\mathbf n(t)}\le t,\qquad\text{hence}\qquad\mathbf n(t)\le N_t+1\, .
		\end{equation}
		Since $\mathbf n(t)\le N_t+1$ by \eqref{eq.n-bound}, while $t-\tau_{\mathbf n(t)}\le\Delta\tau_{\mathbf n(t)}$ and $|X^\lambda(t)-X^\lambda(\tau_{\mathbf n(t)})|\le R_{\mathbf n(t)}$,
		\[
		|X^\lambda(t)-X^\lambda(\tau_{\mathbf n(t)})|\le \max_{0\le k\le N_t+1}R_k,\qquad
		|t-\tau_{\mathbf n(t)}|\le \max_{0\le k\le N_t+1}\Delta\tau_k\, .
		\]
		The bound $\max_{0\le k\le N_t+1} a_k^M\le\sum_{k=0}^{N_t+1} a_k^M$ for nonnegative $a_k$ and \eqref{eq.RD-r} with $r=M$ give
		\[
		J_1(t)=\widehat{\mathbb E}^\lambda_0\big|X^\lambda(t)-X^\lambda(\tau_{\mathbf n(t)})\big|^M\le \widehat{\mathbb E}^\lambda_0\Big[\max_{0\le k\le N_t+1}R_k^M\Big]\le\sum_{k=0}^{N_t+1}\widehat{\mathbb E}^\lambda_0[R_k^M]\le C(N_t+2)\lambda^{-M}\, .
		\]
		Similarly, using $|\ell(\lambda)|\le C\lambda$ (Lemma~\ref{lem.lbound}) and \eqref{eq.RD-r} with $r=M$,
		\[
		J_3(t)\le C\lambda^M\sum_{k=0}^{N_t+1}\widehat{\mathbb E}^\lambda_0[(\Delta\tau_k)^M]\le C(N_t+2)\lambda^M\cdot\lambda^{-2M}=C(N_t+2)\lambda^{-M}\, .
		\]
		Since $t\ge\lambda^{-2}$, we have $(N_t+2)\le C\lambda^2 t$, and therefore
		\begin{equation}\label{eq.A1M}
			J_1(t)\le C\lambda^{2-M}t\le Ct^{M/2}\, ,
		\end{equation}
		\begin{equation}\label{eq.A3M}
			J_3(t)\le C\lambda^{2-M}t\le Ct^{M/2}\, ,
		\end{equation}
		using $t\ge\lambda^{-2}$, which implies $\lambda^{-1}\le\sqrt t$ and hence $\lambda^{2-M}t=\lambda^{-(M-2)}t\le t^{M/2}$.

		\smallskip

		\emph{Step 2.} We estimate $J_2$. We have
		\[
		X^\lambda(\tau_n)-\ell(\lambda)\tau_n=\sum_{k=0}^{n-1}W_k,\qquad J_2(t)=\widehat{\mathbb E}^\lambda_0\Big|\sum_{k=0}^{\mathbf n(t)-1}W_k\Big|^M\, .
		\]

		By Lemma~\ref{lem.renewal}, $\{W_k\}_{k\ge1}$ are i.i.d.\ and centered under $\widehat{\mathbb P}_0^\lambda$, while $W_0$ is independent of $\{W_k\}_{k\ge1}$. The deterministic bound \eqref{eq.n-bound} of Step~1 gives $\mathbf n(t)-1\le N_t=\lfloor\lambda^2 t\rfloor$. Set
		\[
		S_m\coloneqq \sum_{k=1}^m W_k,\qquad m\ge1,\qquad S_0\coloneqq 0\, .
		\]
		Then
		\[
		\Big|\sum_{k=0}^{\mathbf n(t)-1}W_k\Big|^M\le 2^{M-1}\Big(|W_0|^M+\max_{1\le m\le N_t}|S_m|^M\Big)\, ,
		\]
		and hence, by \eqref{eq.W-r} with $r=M$,
		\begin{equation}\label{eq.A2-reduce}
			J_2(t)\le C\lambda^{-M}+C \widehat{\mathbb E}_0^\lambda\Big[\max_{1\le m\le N_t}|S_m|^M\Big]\, .
		\end{equation}

		Now the sequence $\{S_m\}_{m\ge1}$ is an $\mathbb{R}^d$-valued martingale under $\widehat{\mathbb P}_0^\lambda$, so $\{|S_m|\}_{m\ge0}$ is a nonnegative submartingale. By Doob's $L^M$ maximal inequality,
		\begin{equation}\label{eq.Doob}
			\widehat{\mathbb E}_0^\lambda\Big[\max_{0\le m\le N_t}|S_m|^M\Big]\le \Big(\frac{M}{M-1}\Big)^M \widehat{\mathbb E}_0^\lambda|S_{N_t}|^M\, .
		\end{equation}

		Write $S_{N_t}=(S_{N_t}^{(1)},\dots,S_{N_t}^{(d)})$.
		Since
		\[
		|x|^M\le d^{M/2-1}\sum_{i=1}^d |x_i|^M,\qquad x\in\mathbb{R}^d\, ,
		\]
		it is enough to estimate each coordinate. For each $i$, the random variables $\{W_k^{(i)}\}_{k\ge 1}$ are i.i.d.\ mean-zero real-valued random variables under $\widehat{\mathbb P}_0^\lambda$. By Rosenthal's inequality and \eqref{eq.W-r} with $r=2$ and $r=M$,
		\[
		\begin{aligned}
			\widehat{\mathbb E}_0^\lambda|S_{N_t}^{(i)}|^M
			&\le C\left(N_t  \widehat{\mathbb E}_0^\lambda|W_1^{(i)}|^M+\big(N_t  \widehat{\mathbb E}_0^\lambda|W_1^{(i)}|^2\big)^{M/2}\right)\\
			&\le C\left(N_t\lambda^{-M}+N_t^{M/2}\lambda^{-M}\right)\le CN_t^{M/2}\lambda^{-M}\, ,
		\end{aligned}
		\]
		where in the last step we used $M\ge2$. Summing over $1\le i\le d$ yields
		\[
		\widehat{\mathbb E}_0^\lambda|S_{N_t}|^M\le C N_t^{M/2}\lambda^{-M}\, .
		\]
		Together with \eqref{eq.Doob},
		\[
		\widehat{\mathbb E}_0^\lambda\Big[\max_{0\le m\le N_t}|S_m|^M\Big]\le C N_t^{M/2}\lambda^{-M}\, .
		\]
		Since $N_t\le \lambda^2 t$,
		\[
		N_t^{M/2}\lambda^{-M}\le t^{M/2}\, .
		\]
		Substituting this into \eqref{eq.A2-reduce}, we obtain
		\begin{equation}\label{eq.A2M}
			J_2(t)\le C\big(t^{M/2}+\lambda^{-M}\big)\, .
		\end{equation}

		\medskip
		Combining \eqref{eq.split-AM}, \eqref{eq.A1M}, \eqref{eq.A3M}, and \eqref{eq.A2M},
		\[
			\E_\Q \widehat E^{\lambda,\omega}_0\big|X^\lambda(t)-\ell(\lambda)t\big|^M
			\le C\big(t^{M/2}+\lambda^{-M}\big)\le Ct^{M/2}\, ,
		\]
		where the last inequality uses $t\ge\lambda^{-2}$.
	\end{proof}

	Proposition~\ref{pro.additive-renewal} is the clock analogue of Proposition~\ref{pro.moment-renewal}. It differs in two ways: the bound carries an extra factor $\lambda^{-M}$, and, because the clock blocks are only $2$-dependent, the partial sums are handled by splitting into three i.i.d.\ residue classes rather than by a single application of Rosenthal's inequality.

\begin{proposition}[Annealed renewal for the additive clock]\label{pro.additive-renewal}
	For every $\lambda\in(0,1]$, the limit $\lim_{t\to\infty}t^{-1}A^\lambda(t)$ exists $\widehat{\mathbb P}_0^\lambda$-almost surely and in $L^1(\widehat{\mathbb P}_0^\lambda)$ and is the deterministic number
	\begin{equation}\label{eq.eta-def-renewal}
	\eta(\lambda)=\frac{\widehat{\mathbb E}_0^\lambda\left[\int_{\tau_1}^{\tau_2}e^{-2V(X^\lambda(s))} ds\right]}{\widehat{\mathbb E}_0^\lambda[\Delta\tau_1]}\, .
	\end{equation}
	Moreover,
	\begin{equation}\label{eq.eta-bound-renewal}
		\Lambda^{-1}\le\eta(\lambda)\le\Lambda\, .
	\end{equation}
		For every integer $M\ge2$, there exists a finite constant $C=C(M,d,\Lambda,r_0)>0$ such that for all $\lambda\in(0,1]$ and all $t\ge\lambda^{-2}$,
		\begin{equation}\label{eq.additive-renewal-Mth}
			\E_\Q \widehat E^{\lambda,\omega}_0\big|A^\lambda(t)-\eta(\lambda)t\big|^M
			\le C\lambda^{-M}t^{M/2}\, .
		\end{equation}
	\end{proposition}

	\begin{proof}
		Let
		\[
			\bar\eta(\lambda)\coloneqq\frac{\widehat{\mathbb E}_0^\lambda[\Delta A_1^\lambda]}{\widehat{\mathbb E}_0^\lambda[\Delta\tau_1]}\, .
		\]
		Since $\Lambda^{-1}\Delta\tau_1\le \Delta A_1^\lambda\le \Lambda\Delta\tau_1$, we have $\Lambda^{-1}\le\bar\eta(\lambda)\le\Lambda$. Define
		\[
		Z_k\coloneqq \Delta A_k^\lambda-\bar\eta(\lambda)\Delta\tau_k\, .
		\]
	By Lemma~\ref{lem.two-dependent}, the pairs $(\Delta A_k^\lambda,\Delta\tau_k)$ for $k\ge 1$ are stationary and $2$-dependent under $\widehat{\mathbb P}_0^\lambda$. By stationarity and the definition of $\bar\eta(\lambda)$,
		\[
		\widehat{\mathbb E}_0^\lambda Z_k
		=
		\widehat{\mathbb E}_0^\lambda[\Delta A_1^\lambda]
		-\bar\eta(\lambda)\widehat{\mathbb E}_0^\lambda[\Delta\tau_1]
		=0\qquad\text{for every }k\ge 1\, ,
		\]
		so $\{Z_k\}_{k\ge 1}$ is stationary, $2$-dependent, and centered.
		Since $e^{-2V}\le \Lambda$ and $\bar\eta(\lambda)\le \Lambda$,
		\[
		|\Delta A_k^\lambda|+|Z_k|\le C\Delta\tau_k\, .
		\]
		Using Lemma~\ref{lem.block-moment-bounds}, for every integer $r\ge1$,
		\begin{equation}\label{eq.clock-block-moments}
			\widehat{\mathbb E}_0^\lambda|\Delta A_k^\lambda|^r+\widehat{\mathbb E}_0^\lambda|Z_k|^r\le C\lambda^{-2r}\qquad\text{for }k\ge0\, .
		\end{equation}

		Let $\mathbf n(t)\coloneqq \max\{n\ge0:\tau_n\le t\}$ and $N_t\coloneqq \lfloor\lambda^2t\rfloor$. As in \eqref{eq.n-bound}, $\mathbf n(t)-1\le N_t$. Decompose pointwise
		\[
		A^\lambda(t)-\bar\eta(\lambda)t
		=
		B_t+\mathrm{Rem}_t,\qquad B_t\coloneqq \sum_{k=0}^{\mathbf n(t)-1}Z_k,\quad
		\mathrm{Rem}_t\coloneqq A^\lambda(t)-A^\lambda(\tau_{\mathbf n(t)})-\bar\eta(\lambda)(t-\tau_{\mathbf n(t)})\, ,
		\]
		so $|A^\lambda(t)-\bar\eta(\lambda)t|^M\le 2^{M-1}(|B_t|^M+|\mathrm{Rem}_t|^M)$ pointwise. The remainder satisfies $|\mathrm{Rem}_t|\le C\Delta\tau_{\mathbf n(t)}\le C\max_{0\le k\le N_t+1}\Delta\tau_k$, so
		\[
			\widehat{\mathbb E}_0^\lambda|\mathrm{Rem}_t|^M\le C\widehat{\mathbb E}_0^\lambda\max_{0\le k\le N_t+1}(\Delta\tau_k)^M\le C(N_t+2)\lambda^{-2M}\le C\lambda^{-M}t^{M/2}
		\]
		using $t\ge\lambda^{-2}$.

		For the block sum, set $S_n\coloneqq \sum_{k=1}^n Z_k$ for $n\ge 1$ and $S_0\coloneqq 0$, so $B_t=Z_0\cdot\mathbf 1_{\{\mathbf n(t)\ge 1\}}+S_{(\mathbf n(t)-1)\vee 0}$ and $|B_t|^M\le 2^{M-1}(|Z_0|^M+\max_{0\le n\le N_t}|S_n|^M)$ pointwise. Hence
		\[
			\widehat{\mathbb E}_0^\lambda|B_t|^M\le C\widehat{\mathbb E}_0^\lambda|Z_0|^M+C\widehat{\mathbb E}_0^\lambda\max_{0\le n\le N_t}|S_n|^M\, .
		\]
		The initial block contributes
		\[
		\widehat{\mathbb E}_0^\lambda |Z_0|^M\le C\lambda^{-2M}\le C\lambda^{-M}t^{M/2}\, ,
		\]
			because $t\ge\lambda^{-2}$. For the centered $2$-dependent block sum, split into three residue classes modulo~$3$:
			\[
				S_n=\sum_{r\in\{0,1,2\}}S_n^{(r)},\qquad
				S_n^{(r)}\coloneqq\sum_{\substack{1\le k\le n\\ k\equiv r\pmod 3}}Z_k\, .
			\]
				Within each residue class, indices are at distance at least $3$. Lemma~\ref{lem.two-dependent} and stationarity therefore make each residue-class subsequence i.i.d.\ and centered. Pointwise,
			\[
				\max_{1\le n\le N_t}|S_n|^M
				\le C\sum_{r\in\{0,1,2\}}\max_{0\le m\le N_t+1}\left|\sum_{\substack{1\le k\le m\\ k\equiv r\pmod 3}}Z_k\right|^M\, .
			\]
			Doob's $L^M$ maximal inequality and Rosenthal's inequality applied to each reindexed i.i.d.\ centered subsequence, using \eqref{eq.clock-block-moments}, give
			\[
				\widehat{\mathbb E}_0^\lambda\max_{0\le m\le N_t+1}\left|\sum_{\substack{1\le k\le m\\ k\equiv r\pmod 3}}Z_k\right|^M
				\le C\left(N_t\lambda^{-2M}+N_t^{M/2}\lambda^{-2M}\right)
				\le C N_t^{M/2}\lambda^{-2M}\, .
			\]
		Therefore
		\[
		\widehat{\mathbb E}_0^\lambda\max_{1\le n\le N_t}|S_n|^M\le C N_t^{M/2}\lambda^{-2M}\, .
		\]
		Since $N_t\le\lambda^2t$, $\widehat{\mathbb E}_0^\lambda|B_t|^M\le C\lambda^{-M}t^{M/2}$. Combining with the remainder bound gives
		\begin{equation}\label{eq.additive-renewal-Mth-bareta}
			\E_\Q \widehat E^{\lambda,\omega}_0\big|A^\lambda(t)-\bar\eta(\lambda)t\big|^M
			\le C\lambda^{-M}t^{M/2}\, .
		\end{equation}
		With $M=2$, this gives
		\[
			\widehat{\mathbb E}_0^\lambda|t^{-1}A^\lambda(t)-\bar\eta(\lambda)|\le C\lambda^{-1}t^{-1/2}\to 0
			\qquad\text{as }t\to\infty\, .
		\]
		Thus the $L^1$ limit defining $\eta(\lambda)$ exists and equals
		$\bar\eta(\lambda)$, which proves \eqref{eq.eta-def-renewal} and
		\eqref{eq.eta-bound-renewal}. Replacing $\bar\eta(\lambda)$ by
		$\eta(\lambda)$ in \eqref{eq.additive-renewal-Mth-bareta} yields
		\eqref{eq.additive-renewal-Mth}. It remains to prove almost-sure
		convergence.

		For almost-sure convergence, again split $S_n=\sum_{k=1}^n Z_k$ by
		residues modulo~$3$. On each residue class, Lemma~\ref{lem.two-dependent}
		and stationarity give an i.i.d.\ centered sequence with finite variance by
		\eqref{eq.clock-block-moments}. Applying the strong law to each
		residue class gives $n^{-1}S_n\to 0$
		$\widehat{\mathbb P}_0^\lambda$-almost surely.

		For the remainder, fix an integer $m\ge 1$. Applying
		\eqref{eq.tail-increment} with $u=\lambda^2 k/m$ gives
			\[
				\sum_{k\ge 1}\widehat{\mathbb P}_0^\lambda(\Delta\tau_k\ge k/m)
				\le \sum_{k\ge 1}Ce^{-c\lambda^2 k/m}<\infty\, .
			\]
				Borel--Cantelli gives $\Delta\tau_k<k/m$ eventually for each fixed $m$. Intersecting over integers $m\ge 1$ yields $\Delta\tau_k/k\to 0$ almost surely. Since $\mathbf n(t)\to\infty$ and $\mathbf n(t)\le\lambda^2 t+2$ by \eqref{eq.n-bound}, for all sufficiently large $t$ this gives
			\[
				\frac{|\mathrm{Rem}_t|}{t}
				\le C\frac{\Delta\tau_{\mathbf n(t)}}{t}
				= C\frac{\Delta\tau_{\mathbf n(t)}}{\mathbf n(t)}\frac{\mathbf n(t)}{t}\to 0
			\]
			almost surely. The strong law also gives
			\[
				\frac{S_{(\mathbf n(t)-1)\vee0}}{t}
				=
				\frac{S_{(\mathbf n(t)-1)\vee0}}{(\mathbf n(t)-1)\vee1}
				\frac{(\mathbf n(t)-1)\vee1}{t}
				\longrightarrow0,
			\]
			using $\mathbf n(t)\le\lambda^2t+2$. Together with $Z_0/t\to 0$, the decomposition $A^\lambda(t)-\bar\eta(\lambda)t=Z_0\mathbf 1_{\{\mathbf n(t)\ge 1\}}+S_{(\mathbf n(t)-1)\vee 0}+\mathrm{Rem}_t$ gives $t^{-1}A^\lambda(t)\to\bar\eta(\lambda)$ $\widehat{\mathbb P}_0^\lambda$-almost surely.
	\end{proof}

	\subsection{Uniform quenched renewal control}\label{ss.uniform-renewal}

	The annealed bounds of Subsection~\ref{ss.annealed-moments} hold with constants uniform in $\lambda$. We upgrade them to a quenched bound uniform on a random window $\lambda\in(0,\lambda_0(\omega)]$, by a Borel--Cantelli argument over a dyadic grid with Girsanov transfer between nearby tilts.

	\begin{lemma}[Girsanov second-moment bound]\label{lem.girsanov-lambda}
		There exists $C=C(d,\Lambda)<\infty$ such that, for every environment
		$\omega$, every $\lambda,\lambda'\in(0,1]$, every $t>0$, and every
		nonnegative, possibly $\omega$-dependent, path functional $F$ measurable
		with respect to $\sigma(X(s):0\le s\le t)$,
		\begin{equation}\label{eq.girsanov-cs}
			E_0^{\lambda',\omega}[F]\le \exp\bigl(C(\lambda'-\lambda)^2 t/2\bigr) \bigl(E_0^{\lambda,\omega}[F^2]\bigr)^{1/2}\, .
		\end{equation}
		If in addition $|\lambda-\lambda'|^2 t\le 1$, then for every
		square-integrable, possibly $\omega$-dependent, path functional $G$
		measurable with respect to $\sigma(X(s):0\le s\le t)$,
		\begin{equation}\label{eq.girsanov-diff}
			\bigl|E_0^{\lambda',\omega}[G]-E_0^{\lambda,\omega}[G]\bigr|\le C|\lambda-\lambda'|\sqrt t \bigl(E_0^{\lambda,\omega}[G^2]\bigr)^{1/2}\, .
		\end{equation}
		The same estimates hold with the expectations $\widehat E_0^{\lambda,\omega}$ whenever $F$ and $G$ depend only on the diffusion path. For every integer $q\ge1$, the second estimate also holds for $\R^q$-valued $G$, with a constant depending additionally on $q$ and with $|G|^2$ in place of $G^2$.
	\end{lemma}

	\begin{proof}
		Write $\b_\mu(x)\coloneqq\tfrac12\nabla\cdot\a(x)+\mu\a(x)e_1$. Under $P_0^{\lambda,\omega}$, the SDE \eqref{eq.SDE-degenerate} has drift $\b_\lambda$ and diffusion coefficient $\a^{1/2}$. Define
		\[
			M_t\coloneqq (\lambda'-\lambda)\int_0^t\a^{1/2}(X^\lambda(s))e_1\cdot dW_s,\qquad R_t^{\lambda,\lambda'}\coloneqq \exp\bigl(M_t-\tfrac12\langle M\rangle_t\bigr)\, .
		\]
		The quadratic variation $\langle M\rangle_t=(\lambda'-\lambda)^2\int_0^t e_1\cdot\a(X^\lambda(s))e_1 ds\le\Lambda(\lambda-\lambda')^2 t$ is bounded deterministically. This gives the exponential integrability needed for the exponential martingale, so $R_t^{\lambda,\lambda'}$ is a true $P_0^{\lambda,\omega}$-martingale with $E_0^{\lambda,\omega}[R_t^{\lambda,\lambda'}]=1$. Girsanov's theorem changes the drift from $\b_\lambda$ to $\b_{\lambda'}$, and uniqueness in law identifies $R_t^{\lambda,\lambda'}$ as the Radon--Nikodym derivative of $P_0^{\lambda',\omega}|_{\mathcal F_t}$ with respect to $P_0^{\lambda,\omega}|_{\mathcal F_t}$. Moreover,
		\[
			(R^{\lambda,\lambda'}_t)^2=\exp(2M_t-2\langle M\rangle_t)\exp(\langle M\rangle_t)\, .
		\]
		The first factor is a nonnegative $P_0^{\lambda,\omega}$-local martingale and therefore a supermartingale of expectation at most $1$. The second factor is bounded by $\exp(\Lambda(\lambda-\lambda')^2 t)$, so
		\[
			E_0^{\lambda,\omega}\bigl[(R_t^{\lambda,\lambda'})^2\bigr]\le\exp(C(\lambda-\lambda')^2 t)\, .
		\]
		Cauchy--Schwarz on $E_0^{\lambda',\omega}[F]=E_0^{\lambda,\omega}[R_t^{\lambda,\lambda'}F]$ gives \eqref{eq.girsanov-cs}. If $|\lambda-\lambda'|^2 t\le 1$, then $E_0^{\lambda,\omega}[(R_t^{\lambda,\lambda'}-1)^2]=E_0^{\lambda,\omega}[(R_t^{\lambda,\lambda'})^2]-1\le e^{C|\lambda-\lambda'|^2 t}-1\le C|\lambda-\lambda'|^2 t$, and Cauchy--Schwarz on $|E_0^{\lambda',\omega}[G]-E_0^{\lambda,\omega}[G]|=|E_0^{\lambda,\omega}[(R_t^{\lambda,\lambda'}-1)G]|$ gives \eqref{eq.girsanov-diff}.
	\end{proof}

	The annealed renewal estimates now give moment bounds for the quenched $L^2$ errors.

	\begin{lemma}[Quenched second-moment renewal bound]\label{lem.second-moment-renewal}
		Fix an integer $M\ge 2$. There is a constant $C=C(M,d,\Lambda,r_0)$ such that, for every $\lambda\in(0,1]$ and every $t\ge\lambda^{-2}$,
		\[
			\E_\Q\Bigl[\bigl(E_0^{\lambda,\omega}|X^\lambda(t)-\ell(\lambda)t|^2\bigr)^{M/2}\Bigr]\le C t^{M/2}\, ,
		\]
		and
		\[
			\E_\Q\Bigl[\bigl(\widehat E_0^{\lambda,\omega}|A^\lambda(t)-\eta(\lambda)t|^2\bigr)^{M/2}\Bigr]\le C \lambda^{-M} t^{M/2}\, .
		\]
	\end{lemma}

	\begin{proof}
		Jensen's inequality gives, for every environment,
		\[
			\bigl(E_0^{\lambda,\omega}|Z|^2\bigr)^{M/2}\le E_0^{\lambda,\omega}|Z|^M
		\]
		for any path functional $Z$. Apply this with $Z=X^\lambda(t)-\ell(\lambda)t$, average over $\Q$, and use Proposition~\ref{pro.moment-renewal}. The clock estimate is identical, with $\widehat E_0^{\lambda,\omega}$ and Proposition~\ref{pro.additive-renewal}.
	\end{proof}

	To pass from dyadic grids of tilts to all tilts, we need a deterministic modulus of continuity for the limiting velocity and clock rate.

	\begin{lemma}[H\"older continuity of $\ell$ and $\eta$ in $\lambda$]\label{lem.speed-modulus}
		Let $I_k=(2^{-k-1},2^{-k}]$ and $L_k=2^{-k}$. There exists $C=C(d,\Lambda,r_0)<\infty$ such that for every $k\ge 1$ and every $\lambda,\lambda'\in I_k$,
		\begin{equation}\label{eq.ell-modulus}
			|\ell(\lambda)-\ell(\lambda')|\le C\bigl(L_k|\lambda-\lambda'|\bigr)^{1/2}\, ,
		\end{equation}
		and
		\begin{equation}\label{eq.eta-modulus}
			|\eta(\lambda)-\eta(\lambda')|\le C\bigl(|\lambda-\lambda'|/L_k\bigr)^{1/2}\, .
		\end{equation}
	\end{lemma}

	\begin{proof}
		Set $\Delta\coloneqq |\lambda-\lambda'|$. If $\Delta=0$, then both left sides are zero; assume $\Delta>0$. If $\Delta\ge L_k/4$, then by Lemma~\ref{lem.lbound} $|\ell(\lambda)-\ell(\lambda')|\le C(\lambda+\lambda')\le CL_k\le C(L_k\Delta)^{1/2}$, and $|\eta(\lambda)-\eta(\lambda')|\le 2\Lambda\le C(\Delta/L_k)^{1/2}$. Assume $\Delta<L_k/4$ and set $T\coloneqq (L_k\Delta)^{-1}$. Since $\lambda,\lambda'\ge L_k/2$, $T\ge 4L_k^{-2}\ge\lambda^{-2}\vee(\lambda')^{-2}$ and $\Delta^2 T=\Delta/L_k<1/4$.

		\smallskip

		\emph{Step 1.} We prove \eqref{eq.ell-modulus}. Jensen's inequality applied to $X^\lambda(T)-\ell(\lambda)T$, combined with the case $M=2$ of Lemma~\ref{lem.second-moment-renewal}, gives
		\[
			\E_\Q\bigl[(E_0^{\lambda,\omega}|X^\lambda(T)-\ell(\lambda)T|)^2\bigr]\le CT\, .
		\]
		The $L^2$-to-$L^1$ bound then gives
			\[
			\begin{aligned}
				\Bigl|T^{-1}\E_0^\lambda[X^\lambda(T)]-\ell(\lambda)\Bigr|
				&\le T^{-1}\E_0^\lambda|X^\lambda(T)-\ell(\lambda)T|\\
				&\le T^{-1}\bigl(\E_\Q\bigl[(E_0^{\lambda,\omega}|X^\lambda(T)-\ell(\lambda)T|)^2\bigr]\bigr)^{1/2}\\
				&\le CT^{-1/2}\, ,
			\end{aligned}
			\]
		and the same with $\lambda'$. Hence
		\begin{equation}\label{eq.ell-mod-decomp}
			|\ell(\lambda)-\ell(\lambda')|\le 2CT^{-1/2}+T^{-1}\bigl|\E_0^\lambda[X^\lambda(T)]-\E_0^{\lambda'}[X^{\lambda'}(T)]\bigr|\, .
		\end{equation}
			Applying \eqref{eq.girsanov-diff} with the canonical functional $G(\gamma)=\gamma(T)$ gives, for every $\omega$,
			\[
				|E_0^{\lambda',\omega}[X^{\lambda'}(T)]-E_0^{\lambda,\omega}[X^\lambda(T)]|
				\le C\Delta\sqrt T(E_0^{\lambda,\omega}|X^\lambda(T)|^2)^{1/2}\, .
			\]
			Lemma~\ref{lem.max} with $p=2$ gives $E_0^{\lambda,\omega}|X^\lambda(T)|^2\le C(\lambda T)^2\le C(L_kT)^2$. Averaging over $\Q$ and inserting into \eqref{eq.ell-mod-decomp}, we obtain
		\[
			|\ell(\lambda)-\ell(\lambda')|\le 2CT^{-1/2}+CL_k\Delta T^{1/2}=C(L_k\Delta)^{1/2}\, ,
		\]
		using $T^{-1/2}=(L_k\Delta)^{1/2}=L_k\Delta T^{1/2}$.

		\smallskip

		\emph{Step 2.} We prove \eqref{eq.eta-modulus}. Jensen's inequality applied to $A^\lambda(T)-\eta(\lambda)T$, combined with the case $M=2$ of Lemma~\ref{lem.second-moment-renewal}, gives
		\[
			\E_\Q\bigl[(\widehat E_0^{\lambda,\omega}|A^\lambda(T)-\eta(\lambda)T|)^2\bigr]\le C\lambda^{-2}T\, .
		\]
		Hence by Cauchy--Schwarz,
			\[
			\begin{aligned}
				\Bigl|T^{-1}\widehat{\mathbb E}_0^\lambda[A^\lambda(T)]-\eta(\lambda)\Bigr|
				&\le T^{-1}\widehat{\mathbb E}_0^\lambda|A^\lambda(T)-\eta(\lambda)T|\\
				&\le T^{-1}\bigl(\E_\Q\bigl[(\widehat E_0^{\lambda,\omega}|A^\lambda(T)-\eta(\lambda)T|)^2\bigr]\bigr)^{1/2}\\
				&\le C\lambda^{-1}T^{-1/2}\le CL_k^{-1}T^{-1/2}
				=C(\Delta/L_k)^{1/2}\, ,
			\end{aligned}
			\]
		and the same with $\lambda'$. Applying \eqref{eq.girsanov-diff} with the canonical functional $G(\gamma)=\int_0^Te^{-2V(\gamma_s)} ds$ and the deterministic bound $0\le G\le \Lambda T$ gives
		\[
			\bigl|\widehat E_0^{\lambda',\omega}[A^{\lambda'}(T)]-\widehat E_0^{\lambda,\omega}[A^\lambda(T)]\bigr|\le C\Delta\sqrt T\cdot T= C\Delta T^{3/2}\, .
		\]
		Dividing by $T$ gives $C\Delta T^{1/2}=C(\Delta/L_k)^{1/2}$. Combining yields \eqref{eq.eta-modulus}.
	\end{proof}

	The annealed bounds and the H\"older modulus combine via Markov's inequality and a Borel--Cantelli argument over a dyadic grid in $\lambda$ to produce a single quenched bound, uniform in $\lambda$ on a random window.

	\begin{corollary}[Quenched renewal estimate uniform in $\lambda$]\label{cor.uniform-renewal}
		Let $\theta\in(0,1/4)$ and $\varepsilon>0$. There exist a $\Q$-full-probability event $\Omega_{\theta,\varepsilon}$ and a deterministic constant $C=C(\theta,\varepsilon,d,\Lambda,r_0)<\infty$ such that, for every $\omega\in\Omega_{\theta,\varepsilon}$, there exists a finite random scale $\lambda_0(\omega)>0$ with the property that for every $0<\lambda\le\lambda_0(\omega)$ and every $t\ge\lambda^{-2-\varepsilon}$,
		\begin{equation}\label{eq.uniform-X}
			E_0^{\lambda,\omega}|X^\lambda(t)-\ell(\lambda)t|\le C\sqrt t (\lambda^2 t)^\theta\, ,
		\end{equation}
		and
		\begin{equation}\label{eq.uniform-A}
			\widehat E_0^{\lambda,\omega}|A^\lambda(t)-\eta(\lambda)t|\le C\lambda^{-1}\sqrt t (\lambda^2 t)^\theta\, .
		\end{equation}
		Moreover, $\lambda_0$ has polynomial tails of every order: for each $p>0$ there is $C_p=C_p(p,\theta,\varepsilon,d,\Lambda,r_0)<\infty$ such that
		\begin{equation}\label{eq.lambda0-tail}
			\Q\bigl(\lambda_0\le 2^{-k}\bigr)\le C_p 2^{-pk}\qquad\text{for every integer }k\ge 0\, .
		\end{equation}
	\end{corollary}

	\begin{proof}
		Define the quenched $L^2$ functionals
		\[
			\mathcal W_n^X(\lambda,\omega)\coloneqq \bigl(E_0^{\lambda,\omega}|X^\lambda(n)-\ell(\lambda)n|^2\bigr)^{1/2},\qquad
			\mathcal W_n^A(\lambda,\omega)\coloneqq \bigl(\widehat E_0^{\lambda,\omega}|A^\lambda(n)-\eta(\lambda)n|^2\bigr)^{1/2}\, .
		\]
		For $k\ge 1$, set $L_k\coloneqq 2^{-k}$, $I_k\coloneqq (2^{-k-1},2^{-k}]$, $N_k\coloneqq \lceil L_k^{-2-\varepsilon}\rceil$, and choose a $r_{k,n}$-net $\mathcal N_{k,n}\subseteq I_k$ with mesh
		\begin{equation}\label{eq.net-mesh}
			r_{k,n}\coloneqq c_{\rm net} L_k^{4\theta-1}n^{2\theta-1}
		\end{equation}
		for $n\ge N_k$, where $c_{\rm net}\in(0,1]$ is deterministic and depends on $(\theta,\varepsilon,d,\Lambda,r_0)$. Since $2\theta-1<0$ and $n\ge N_k$, choosing $c_{\rm net}$ small enough gives
		\[
			r_{k,n}\le Cc_{\rm net}L_k^{1+\varepsilon(1-2\theta)}\le L_k\, .
		\]
		Thus the net may be chosen so that $|\mathcal N_{k,n}|\le CL_k/r_{k,n}\le CL_k^{2-4\theta}n^{1-2\theta}$.

		\smallskip

		\emph{Step 1.} Choose an integer
		$M_0\ge 2$ with $M_0\theta>2/\varepsilon+2$.

		Lemma~\ref{lem.second-moment-renewal} requires the renewal time
		$n\ge(\lambda')^{-2}$. Since $\lambda'\ge L_k/2$ on $I_k$, we have
		$(\lambda')^{-2}\le 4L_k^{-2}\le N_k$ as soon as
		$L_k^{-\varepsilon}\ge 4$, i.e.,\ $k\ge k_*(\varepsilon)\coloneqq
		\lceil 2/\varepsilon\rceil$. For $k\ge k_*$, every $n\ge N_k$ satisfies
		$n\ge(\lambda')^{-2}$ for every $\lambda'\in I_k$.

		By Markov's inequality and Lemma~\ref{lem.second-moment-renewal}, for
		every $k\ge k_*$, every $\lambda'\in I_k$, and every $n\ge N_k$,
		\begin{align*}
			\Q\bigl(\mathcal W_n^X(\lambda',\cdot)>\tfrac12\sqrt n((\lambda')^2 n)^\theta\bigr)&\le C((\lambda')^2n)^{-M_0\theta}\, ,\\
			\Q\bigl(\mathcal W_n^A(\lambda',\cdot)>\tfrac12(\lambda')^{-1}\sqrt n((\lambda')^2 n)^\theta\bigr)&\le C((\lambda')^2n)^{-M_0\theta}\, .
		\end{align*}
		Let $B_{k,n}$ be the event that some $\xi\in\mathcal N_{k,n}$ violates
		either bound. Since $L_k/2<\xi\le L_k$ for $\xi\in I_k$, the union bound gives, for
		$k\ge k_*$,
		\[
			\Q(B_{k,n})\le C L_k^{2-4\theta-2M_0\theta}n^{1-2\theta-M_0\theta}\, .
		\]
		The choice $M_0\theta>2/\varepsilon+2$ makes the exponent of $n$
		strictly less than $-1-2/\varepsilon$, hence summable. Summing in
		$n\ge N_k$ and using $N_k=\lceil L_k^{-2-\varepsilon}\rceil$, the integral
		test gives
		$\sum_{n\ge N_k}\Q(B_{k,n})\le
		CL_k^{-2+\varepsilon(M_0\theta+2\theta-2)}$. The exponent of $L_k$ is
		positive, so with $E_k\coloneqq \bigcup_{n\ge N_k}B_{k,n}$,
		$\sum_{k\ge k_*}\Q(E_k)<\infty$. Borel--Cantelli yields a
		$\Q$-full-probability event $\Omega_{\theta,\varepsilon}$ on which only
		finitely many $E_k$ occur.

		\smallskip

		\emph{Step 2.} Fix $\omega\in\Omega_{\theta,\varepsilon}$. For the
		prefactor $\exp(C\Delta^2n/2)$ in \eqref{eq.girsanov-cs} to remain
		bounded, the mesh must satisfy $\Delta^2 n\le 1$. With $\Delta\le r_{k,n}$ from
		\eqref{eq.net-mesh} and $n\ge N_k=\lceil L_k^{-2-\varepsilon}\rceil$, and using
		$4\theta-1<0$ together with $n\ge N_k$,
		\[
			\Delta^2 n\le r_{k,n}^2 n=c_{\rm net}^2 L_k^{8\theta-2}n^{4\theta-1}\le c_{\rm net}^2 L_k^{8\theta-2}N_k^{4\theta-1}\le c_{\rm net}^2 L_k^{\varepsilon(1-4\theta)}\to 0
		\]
		as $k\to\infty$, since $\theta<1/4$. Decrease $c_{\rm net}$, if
		necessary, so that the constants multiplying $c_{\rm net}^{1/2}$ in the
		displacement and clock estimates are at most one, and then choose an
		integer $k_1=k_1(\theta,\varepsilon,d,\Lambda,r_0)<\infty$ with
		$c_{\rm net}^2 L_k^{\varepsilon(1-4\theta)}\le 1$ for every $k\ge k_1$.

		Define
		\begin{equation}\label{eq.k0-canonical}
			k_0(\omega)\coloneqq \max\Bigl\{k_1, 1+\sup\{j\ge k_*:E_j(\omega)\text{ occurs}\}\Bigr\}\, ,
		\end{equation}
		with the convention $\sup\emptyset\coloneqq k_*-1$, so that
		$k_0(\omega)<\infty$ on $\Omega_{\theta,\varepsilon}$. Then for every
		$k\ge k_0(\omega)$, $E_k(\omega)$ does not occur, hence $B_{k,n}^c$ for
		every $n\ge N_k$. For every $\lambda\in I_k$ with $k\ge k_0(\omega)$ and
		every integer $n\ge N_k$, choose $\xi\in\mathcal N_{k,n}$ with
		$\Delta\coloneqq |\lambda-\xi|\le r_{k,n}$, so that $\Delta^2 n\le 1$.
		Lemma~\ref{lem.girsanov-lambda} \eqref{eq.girsanov-cs} then gives, for
		every $\omega$,
		\[
			E_0^{\lambda,\omega}|X^\lambda(n)-\ell(\xi)n|\le e^{C\Delta^2 n/2}\bigl(E_0^{\xi,\omega}|X^\xi(n)-\ell(\xi)n|^2\bigr)^{1/2}\le C\mathcal W_n^X(\xi,\omega)\, .
		\]
		The triangle inequality and Lemma~\ref{lem.speed-modulus} \eqref{eq.ell-modulus} yield
		\[
			E_0^{\lambda,\omega}|X^\lambda(n)-\ell(\lambda)n|\le C\mathcal W_n^X(\xi,\omega)+n|\ell(\lambda)-\ell(\xi)|\le C\sqrt n(L_k^2 n)^\theta+Cn(L_k\Delta)^{1/2}\, .
		\]
		The second term satisfies $Cn(L_k\Delta)^{1/2}\le Cn(L_kr_{k,n})^{1/2}=Cc_{\rm net}^{1/2}\sqrt n(L_k^2 n)^\theta$ (using $\Delta\le r_{k,n}$), which is absorbed into the first term after fixing $c_{\rm net}$ small. Since $L_k/2<\lambda\le L_k$, $(L_k^2n)^\theta\le C(\lambda^2n)^\theta$, giving
		\begin{equation}\label{eq.uniform-int-X}
			E_0^{\lambda,\omega}|X^\lambda(n)-\ell(\lambda)n|\le C\sqrt n(\lambda^2 n)^\theta\qquad\text{for every integer }n\ge N_k\, .
		\end{equation}
		For the clock, Lemma~\ref{lem.girsanov-lambda} \eqref{eq.girsanov-cs} applied to the path functional $\gamma\mapsto \left|\int_0^n e^{-2V(\gamma_s)} ds-\eta(\xi)n\right|$ gives
		\[
			\widehat E_0^{\lambda,\omega}|A^\lambda(n)-\eta(\xi)n|
			\le e^{C\Delta^2 n/2}\bigl(\widehat E_0^{\xi,\omega}|A^\xi(n)-\eta(\xi)n|^2\bigr)^{1/2}
			\le C\mathcal W_n^A(\xi,\omega)\, .
		\]
		From \eqref{eq.eta-modulus} and $\Delta\le r_{k,n}$,
		\[
			n|\eta(\lambda)-\eta(\xi)|\le Cn(\Delta/L_k)^{1/2}\le Cc_{\rm net}^{1/2}L_k^{-1}\sqrt n(L_k^2n)^\theta\, ,
		\]
		again absorbed by choosing $c_{\rm net}$ small. Combining with $\mathcal W_n^A(\xi,\omega)\le C L_k^{-1}\sqrt n(L_k^2n)^\theta$,
		\begin{equation}\label{eq.uniform-int-A}
			\widehat E_0^{\lambda,\omega}|A^\lambda(n)-\eta(\lambda)n|\le C\lambda^{-1}\sqrt n(\lambda^2 n)^\theta\qquad\text{for every integer }n\ge N_k\, .
		\end{equation}

		\smallskip

		\emph{Step 3.} For arbitrary $t\ge\lambda^{-2-\varepsilon}$, set $n\coloneqq \lceil t\rceil$. Then $n-t\le 1$ and $n\ge N_k$. By boundedness of the drift and the Burkholder--Davis--Gundy inequality applied on $[t,n]$, $E_0^{\lambda,\omega}|X^\lambda(t)-X^\lambda(n)|\le C$. Together with $|\ell(\lambda)|\le C\lambda\le C$ from Lemma~\ref{lem.lbound} and \eqref{eq.uniform-int-X}, this gives $E_0^{\lambda,\omega}|X^\lambda(t)-\ell(\lambda)t|\le C\sqrt n(\lambda^2 n)^\theta+C$. Since $\sqrt t(\lambda^2t)^\theta\ge 1$ for $t\ge\lambda^{-2}$, the right-hand side is at most $C\sqrt t(\lambda^2 t)^\theta$, proving \eqref{eq.uniform-X}.

		For the clock, $0\le e^{-2V}\le\Lambda$ and \eqref{eq.eta-bound-renewal} give $|A^\lambda(t)-A^\lambda(n)|+|\eta(\lambda)| |t-n|\le C$; combining this with \eqref{eq.uniform-int-A} proves \eqref{eq.uniform-A}.

		Setting $\lambda_0(\omega)\coloneqq 2^{-k_0(\omega)}$ on $\Omega_{\theta,\varepsilon}$, and $\lambda_0(\omega)\coloneqq 1$ on $\Omega\setminus\Omega_{\theta,\varepsilon}$, defines $\lambda_0$ as a random variable on all of $\Omega$.

		\smallskip

		\emph{Step 4.} We prove the tail bound for $\lambda_0$. The thresholds in $B_{k,n}$ and $E_k$ do not depend on the moment exponent, so we may repeat Step~1 with a larger exponent. Fix $p>0$ and choose an integer $M_p\ge 2$ with $-2+\varepsilon(M_p\theta+2\theta-2)\ge p$. Using $M_p$ in Markov's inequality and summing over $n\ge N_k$ gives, for every $k\ge k_*$,
		\[
			\Q(E_k)\le \sum_{n\ge N_k}\Q(B_{k,n})\le C_p L_k^{-2+\varepsilon(M_p\theta+2\theta-2)}\le C_p 2^{-pk}\, ,
		\]
		with $C_p$ depending on $(p,\theta,\varepsilon,d,\Lambda,r_0)$. From the canonical definition \eqref{eq.k0-canonical}, $\{k_0(\omega)\ge k\}\subseteq \{k\le k_1\vee k_*\}\cup\bigcup_{j\ge \max\{k-1,k_*\}}E_j$. For $k\le k_1\vee k_*$ the inclusion is trivially absorbed by enlarging $C_p$; for $k> k_1\vee k_*$ the first term is empty, so
		\begin{equation*}
			\Q\bigl(\lambda_0\le 2^{-k}\bigr)=\Q\bigl(k_0(\omega)\ge k\bigr)\le \sum_{j\ge k-1}\Q(E_j)\le C_p 2^{-pk}\, ,
		\end{equation*}
		which is \eqref{eq.lambda0-tail} and completes the proof.
	\end{proof}

	\section{Proof of Theorem~\ref{t.main}}\label{sec.einstein-rate}

This section proves Theorem~\ref{t.main}. By the time change of
Subsection~\ref{ss.reduction}, the mobility error
$\ell_X(\lambda)/\lambda-\Sigma_X e_1$ is controlled, through the identities
\eqref{eq.time-change-identities}, by two errors of the time-changed
diffusion: the velocity error $\ell(\lambda)/\lambda-\ahom e_1$ and the clock
error $\eta(\lambda)-\E[e^{-2V(0)}]$. We prove algebraic rates in
$\lambda$ for these two quantities.

The Feynman--Kac formula connects the resolvent solution $u_m(0)$ from
Section~\ref{sec.tilted-resolvent} to the path of $X^\lambda$: it represents
$3^{-m}u_m(0)$ as an exponentially weighted time-average of the rescaled
displacement of the particle, plus an exponentially damped boundary term. On the diffusive time scale
$t\sim\lambda^{-2}$, the law of large numbers thus gives
$3^{-m}u_m(0)\approx\ell(\lambda)/\lambda$. The velocity error
$\ell(\lambda)/\lambda-\ahom e_1$ therefore splits into the difference
$\ell(\lambda)/\lambda-3^{-m}u_m(0)$, which measures the law-of-large-numbers
error inside the Feynman--Kac average, and the deviation
$3^{-m}u_m(0)-\ahom e_1$, which is the homogenization error estimated in
Section~\ref{sec.tilted-resolvent}.

The first difference decomposes further into three terms: a short-time term
where the law-of-large-numbers approximation has not yet taken effect; a
long-time term controlled by the uniform quenched law of large numbers of
Section~\ref{sec.renewal}; and a boundary term controlled by the exit estimate
of Proposition~\ref{pro.exit}. Together with the homogenization term, this gives
four contributions to the velocity error. The clock error is treated by the
same decomposition applied to the clock resolvent $q_m$ from
\eqref{eq.clock-resolvent}. Lemmas~\ref{lem.velocity-decomp}
and~\ref{lem.clock-decomp} make the four-term bounds explicit, and
Theorem~\ref{t.main} follows by optimizing the choices of parameters in each of these estimates. 

		We first record the scaled Feynman--Kac identity used below. Define
		\[
			v_\lambda(r)\coloneqq \frac{\lambda}{r}X^\lambda \Bigl(\frac{r}{\lambda^2}\Bigr),\qquad r>0\, .
		\]
		Set $v_\lambda(0)\coloneqq0$.
	The Feynman--Kac formula~\eqref{eq.FK-resolvent} gives, for $m,h\in\N$ with $\lambda=\rho3^{-m}$,
	\begin{equation}\label{eq.FK-scaled-final}
		3^{-m}u_m(0)=E^{\lambda,\omega}_0\int_0^{\rho3^{-2m}\tau_{m,h}}\xi e^{-\xi}v_\lambda(\rho \xi) d\xi+E^{\lambda,\omega}_0\left[e^{-\rho3^{-2m}\tau_{m,h}}\left(3^{-m}X^\lambda(\tau_{m,h})+\ahom e_1\right)\right]\, .
	\end{equation}

	For an exponent $\alpha\in(0,1)$, a parameter $\zeta>0$, and each $\lambda\in(0,1]$, write
	\[
		m(\lambda)\coloneqq \Bigl\lceil\frac{-\log_3\lambda}{1-\alpha}\Bigr\rceil,\qquad \rho(\lambda)\coloneqq \lambda 3^{m(\lambda)},\qquad h(\lambda)\coloneqq \lfloor\zeta m(\lambda)\rfloor\, ,
	\]
	so that
	\begin{align}\label{eq.rho-bounds}
		3^{\alpha m(\lambda)}&\le\rho(\lambda)\le 3^{1-\alpha} 3^{\alpha m(\lambda)}\, ,\notag\\
		\rho(\lambda)&\asymp\lambda^{-\alpha/(1-\alpha)}\, ,\\
		3^{-m(\lambda)}&\asymp\lambda^{1/(1-\alpha)}\, .\notag
	\end{align}

	Let $\beta_h$ and $\delta=2\beta_h/(d+4)$ be as in Section~\ref{sec.coarse-graining-inputs}, and let $\zeta_H$ be the constant fixed in Section~\ref{sec.tilted-resolvent}.

		The next lemma decomposes the time-changed velocity error into the four terms described above.

			\begin{lemma}[Velocity decomposition]\label{lem.velocity-decomp}
				Let $\alpha\in(0,\delta/2)$, let $\zeta\in(0,\zeta_H]$, let $\theta\in(0,1/4)$, and let $\varepsilon>0$ satisfy $\varepsilon(1-\alpha)<\alpha$.
				There exist deterministic constants $C,C'<\infty$ and $c>0$, a deterministic exponent $\sigma\in(0,1)$, and a random scale $\lambda_*^H\colon\Omega\to(0,1]$ such that
				\begin{equation}\label{eq.lambda-star-tail}
						\Q\bigl(\lambda_*^H\le\lambda\bigr)\le C'\exp\bigl(-c (\log(1/\lambda))^\sigma\bigr)\qquad\text{for every }\lambda\in(0,1]\, .
					\end{equation}
					The constants $C,C',c,\sigma$ depend only on the displayed parameters and on $(d,\Lambda,r_0,\beta_h)$. For $\Q$-a.e.\ $\omega$ and every $0<\lambda\le\lambda_*^H(\omega)$,
			\begin{equation}\label{eq.velocity-decomp-bound}
			\begin{aligned}
				\Bigl|\frac{\ell(\lambda)}{\lambda}-\ahom e_1\Bigr|
				&\le C\rho(\lambda)^{-2}\lambda^{-2\varepsilon}
				+C\rho(\lambda)^{-1/2+\theta}
				+Ce^{-c3^{h(\lambda)}}\\
				&\qquad
				+C \rho(\lambda)^{1/2}3^{(d/2+1-\beta_h)h(\lambda)-\delta m(\lambda)}\, .
			\end{aligned}
			\end{equation}
	\end{lemma}

		\begin{proof}
			We work on the intersection of the $\Q$-full-probability events of Proposition~\ref{pro.velocity-resolvent-infty} and of Corollary~\ref{cor.uniform-renewal} applied with parameters $(\theta,\varepsilon)$. Throughout, write $m=m(\lambda)$, $\rho=\rho(\lambda)$, $h=h(\lambda)$.
			By \eqref{eq.rho-bounds} and $\alpha<\delta/2$, the deterministic scale conditions required by Proposition~\ref{pro.velocity-resolvent-infty}, including $\rho\le 3^{\delta m/2}$, hold for all sufficiently large $m$.

		\smallskip

		\emph{Step 1.} The boundary data $u_m=3^m\ahom e_1$ on $\partial\cu_{m+h}$ in \eqref{eq.resolvent-equation} accounts for the term $\ahom e_1$ in the boundary contribution of \eqref{eq.FK-scaled-final}. From \eqref{eq.FK-scaled-final}, $\int_0^\infty\xi e^{-\xi} d\xi=1$, and the triangle inequality,
		\begin{multline}\label{eq.main-split}
		\Bigl|\frac{\ell(\lambda)}{\lambda}-3^{-m}u_m(0)\Bigr|\le\underbrace{E^{\lambda,\omega}_0\int_0^{\rho3^{-2m}\tau_{m,h}}\xi e^{-\xi}\Bigl|v_\lambda(\rho\xi)-\frac{\ell(\lambda)}{\lambda}\Bigr|d\xi}_{(a)}\\
		+\underbrace{E^{\lambda,\omega}_0\int_{\rho3^{-2m}\tau_{m,h}}^\infty \xi e^{-\xi}\Bigl|\frac{\ell(\lambda)}{\lambda}\Bigr|d\xi}_{(b)}+
		\underbrace{E^{\lambda,\omega}_0\Bigl[e^{-\rho3^{-2m}\tau_{m,h}}\Bigl|3^{-m}X^\lambda(\tau_{m,h})+\ahom e_1\Bigr|\Bigr]}_{(c)}\, .
		\end{multline}

		\smallskip

		\emph{Step 2.} By \eqref{eq.rho-bounds}, $\lambda^{-\varepsilon}/\rho\le 3^{1-\alpha}\lambda^{\alpha/(1-\alpha)-\varepsilon}\to 0$ since $\varepsilon(1-\alpha)<\alpha$. Split the integral of~$(a)$ into two terms~$(a)=(a.1)+(a.2)$ at $(\lambda^{-\varepsilon}/\rho)\wedge \rho3^{-2m}\tau_{m,h}$. For $r\ge 1$, Lemma~\ref{lem.max} at time $r/\lambda^2$ gives $E_0^{\lambda,\omega}|v_\lambda(r)|\le C$. For $0<r<1$, use $|X^\lambda(r/\lambda^2)|\le\max_{s\le\lambda^{-2}}|X^\lambda(s)|$ and apply the same lemma at time $\lambda^{-2}$ to get $E_0^{\lambda,\omega}|v_\lambda(r)|\le Cr^{-1}$. Combined with $|\ell(\lambda)|\le C\lambda$,
		\[
			E_0^{\lambda,\omega}\Bigl|v_\lambda(r)-\frac{\ell(\lambda)}{\lambda}\Bigr|\le C(1+r^{-1}),\qquad r>0\, ,
		\]
		so
		\[
			(a.1)\le C\int_0^{\lambda^{-\varepsilon}/\rho}\xi e^{-\xi}\bigl(1+(\rho\xi)^{-1}\bigr)d\xi\le C\rho^{-2}\lambda^{-2\varepsilon}\, .
		\]

		\smallskip

		\emph{Step 3.} For $\xi\ge\lambda^{-\varepsilon}/\rho$, $t\coloneqq \rho\xi/\lambda^2\ge\lambda^{-2-\varepsilon}$. Corollary~\ref{cor.uniform-renewal} \eqref{eq.uniform-X} gives
		\[
			E_0^{\lambda,\omega}\Bigl|v_\lambda(\rho\xi)-\frac{\ell(\lambda)}{\lambda}\Bigr|=\frac{\lambda}{\rho\xi}E_0^{\lambda,\omega}|X^\lambda(t)-\ell(\lambda)t|\le\frac{\lambda}{\rho\xi}\cdot C\sqrt t(\lambda^2t)^\theta=C\rho^{-1/2+\theta}\xi^{-1/2+\theta}\, ,
		\]
		so
		\[
			(a.2)\le C\rho^{-1/2+\theta}\int_0^\infty\xi^{1/2+\theta}e^{-\xi} d\xi\le C\rho^{-1/2+\theta}\, .
		\]

		\smallskip

			\emph{Step 4.} The identity $\int_r^\infty \xi e^{-\xi} d\xi=(r+1)e^{-r}\le Ce^{-r/2}$ controls the exponential tail in $(b)$. Combining this with $|\ell(\lambda)|\le C\lambda$ and Proposition~\ref{pro.exit} ($\kappa=1/2$) gives $(b)\le Ce^{-c3^h}$. Since $X^\lambda(\tau_{m,h})\in\partial\cu_{m+h}$, the triangle inequality gives $|3^{-m}X^\lambda(\tau_{m,h})+\ahom e_1|\le C3^h$. Proposition~\ref{pro.exit} ($\kappa=1$) then gives $(c)\le C 3^h e^{-c3^h}\le Ce^{-c'3^h}$ after decreasing $c'$ and increasing $C$. Hence
		\begin{equation}\label{eq.geom-bc}
			(b)+(c)\le Ce^{-c3^h}\, .
		\end{equation}

		\smallskip

		\emph{Step 5.} Combining \eqref{eq.main-split}--\eqref{eq.geom-bc} with Proposition~\ref{pro.velocity-resolvent-infty} yields \eqref{eq.velocity-decomp-bound} with deterministic $C$ and the random threshold
		\[
			\lambda_*^H(\omega)\coloneqq 3^{-(1-\alpha)\X_{\beta_h}(\omega)}\wedge\lambda_0(\omega,\theta,\varepsilon)\wedge 1\, .
		\]
			The condition $\lambda\le\lambda_*^H(\omega)$ together with $m(\lambda)\ge -\log_3\lambda/(1-\alpha)$ guarantees the hypotheses of Proposition~\ref{pro.velocity-resolvent-infty}. The tail \eqref{eq.lambda-star-tail} follows from \eqref{eq.m-beta0-tail}, \eqref{eq.lambda0-tail}, and Lemma~\ref{lem.tail-conversion}.
	\end{proof}

	The clock resolvent $q_m$ from \eqref{eq.clock-resolvent} admits the same Feynman--Kac representation, but with the centered integrand $e^{-2V}-\E[e^{-2V(0)}]$ in place of the velocity. After integration by parts, this representation decomposes $q_m(0)$ into an exponentially weighted average of the time-averaged clock error and an exit term. Unlike the velocity resolvent estimate, Proposition~\ref{pro.clock-resolvent-infty} imposes no upper bound of the form $h\le\zeta_Hm$, so the next lemma allows arbitrary $\zeta>0$.

		\begin{lemma}[Clock decomposition]\label{lem.clock-decomp}
			Let $\alpha\in(0,\delta/2)$, let $\zeta>0$, let $\theta\in(0,1/4)$, and let $\varepsilon>0$ satisfy $\varepsilon(1-\alpha)<\alpha$.
			There exist deterministic constants $C,C'<\infty$ and $c>0$, a deterministic exponent $\sigma\in(0,1)$, and a random scale $\lambda_*^{\rm sc}\colon\Omega\to(0,1]$ such that
			\begin{equation}\label{eq.clock-lambda-star-tail}
					\Q\bigl(\lambda_*^{\rm sc}\le\lambda\bigr)\le C'\exp\bigl(-c (\log(1/\lambda))^\sigma\bigr)\qquad\text{for every }\lambda\in(0,1]\, .
				\end{equation}
				The constants $C,C',c,\sigma$ depend only on the displayed parameters and on $(d,\Lambda,r_0,\beta_h)$. For $\Q$-a.e.\ $\omega$ and every $0<\lambda\le\lambda_*^{\rm sc}(\omega)$,
			\begin{equation}\label{eq.clock-decomp-bound}
			\begin{aligned}
				|\eta(\lambda)-\E[e^{-2V(0)}]|
				&\le C\rho(\lambda)^{-2}\lambda^{-2\varepsilon}
				+C\rho(\lambda)^{-1/2+\theta}
				+Ce^{-c3^{h(\lambda)}}\\
				&\qquad
				+C\rho(\lambda)^{13/12}3^{(d/2+7/12-\beta_h/2)h(\lambda)-\delta m(\lambda)}\, .
			\end{aligned}
			\end{equation}
	\end{lemma}

	\begin{proof}
		We work on the intersection of the $\Q$-full-probability events of Proposition~\ref{pro.clock-resolvent-infty} and of Corollary~\ref{cor.uniform-renewal} applied with parameters $(\theta,\varepsilon)$. Write $m=m(\lambda),\rho=\rho(\lambda),h=h(\lambda)$.
		By \eqref{eq.rho-bounds} and $\alpha<\delta/2$, $\rho 3^{-\delta m/2}\le 1$ for all sufficiently large $m$; absorbing this deterministic threshold into $\X_\mu$ preserves the tail \eqref{eq.m-mu-tail} after increasing the prefactor.

		Throughout the proof we use $\widehat E_0^{\lambda,\omega}$ (the expectation on the enlarged probability space carrying the Bernoulli marks, defined in Subsection~\ref{ss.renewal-setup}) in place of $E_0^{\lambda,\omega}$; the two agree because the integrands depend only on the diffusion-path marginal. The Feynman--Kac formula for \eqref{eq.clock-resolvent} gives
		\[
			q_m(0)=\rho 3^{-2m}\widehat E_0^{\lambda,\omega}\biggl[\int_0^{\tau_{m,h}}e^{-\rho 3^{-2m}t}\bigl(e^{-2V(X^\lambda(t))}-\E[e^{-2V(0)}]\bigr) dt\biggr]\, .
		\]
		Set
		\[
			I(t)\coloneqq\int_0^t\bigl(e^{-2V(X^\lambda(s))}-\E[e^{-2V(0)}]\bigr) ds\, .
		\]
		Integration by parts in $t$ rewrites the inner integral as
		\[
			e^{-\rho 3^{-2m}\tau_{m,h}}I(\tau_{m,h})+\rho 3^{-2m}\int_0^{\tau_{m,h}}e^{-\rho 3^{-2m}t}I(t) dt\, .
		\]
		Substituting $\xi=\rho 3^{-2m}t$ in the second term gives
		\[
		\begin{aligned}
			q_m(0)
			&=\widehat E_0^{\lambda,\omega}\Bigl[\int_0^{\rho3^{-2m}\tau_{m,h}}\xi e^{-\xi}
			\frac{\lambda^2}{\rho\xi}\int_0^{\rho\xi/\lambda^2}\bigl(e^{-2V(X^\lambda(s))}-\E[e^{-2V(0)}]\bigr)ds d\xi\Bigr]\\
			&\quad+\rho3^{-2m}\widehat E_0^{\lambda,\omega}\Bigl[e^{-\rho3^{-2m}\tau_{m,h}}\int_0^{\tau_{m,h}}\bigl(e^{-2V(X^\lambda(s))}-\E[e^{-2V(0)}]\bigr)ds\Bigr]\, .
		\end{aligned}
		\]
		Since $\int_0^\infty\xi e^{-\xi} d\xi=1$, the triangle inequality gives $|\eta(\lambda)-\E[e^{-2V(0)}]-q_m(0)|\le A_1+A_2+A_3$ with
		\begin{align*}
			A_1&\coloneqq \widehat E_0^{\lambda,\omega}\int_0^{\rho3^{-2m}\tau_{m,h}}\xi e^{-\xi}\left|\frac{\lambda^2}{\rho\xi}\int_0^{\rho\xi/\lambda^2}\bigl(e^{-2V(X^\lambda(s))}-\E[e^{-2V(0)}]\bigr)ds-(\eta(\lambda)-\E[e^{-2V(0)}])\right| d\xi\, ,\\
			A_2&\coloneqq \widehat E_0^{\lambda,\omega}\int_{\rho3^{-2m}\tau_{m,h}}^\infty\xi e^{-\xi}|\eta(\lambda)-\E[e^{-2V(0)}]| d\xi\, ,\\
			A_3&\coloneqq \rho3^{-2m}\widehat E_0^{\lambda,\omega}\Bigl[e^{-\rho3^{-2m}\tau_{m,h}}\Bigl|\int_0^{\tau_{m,h}}\bigl(e^{-2V(X^\lambda(s))}-\E[e^{-2V(0)}]\bigr)ds\Bigr|\Bigr]\, .
		\end{align*}
		By \eqref{eq.rho-bounds} and $\varepsilon(1-\alpha)<\alpha$, $\lambda^{-\varepsilon}/\rho\le 3^{1-\alpha}\lambda^{\alpha/(1-\alpha)-\varepsilon}\to 0$. Split $A_1=A_{1,1}+A_{1,2}$ at $(\lambda^{-\varepsilon}/\rho)\wedge\rho3^{-2m}\tau_{m,h}$. Since the absolute value in $A_1$ is at most $C$,
		\[
			A_{1,1}\le C\int_0^{\lambda^{-\varepsilon}/\rho}\xi e^{-\xi} d\xi\le C\rho^{-2}\lambda^{-2\varepsilon}\, .
		\]
		For $\xi\ge\lambda^{-\varepsilon}/\rho$, $t\coloneqq \rho\xi/\lambda^2\ge\lambda^{-2-\varepsilon}$, and Corollary~\ref{cor.uniform-renewal} \eqref{eq.uniform-A} gives
		\[
		\begin{aligned}
			&\widehat E_0^{\lambda,\omega}\left|\frac{\lambda^2}{\rho\xi}\int_0^{\rho\xi/\lambda^2}\bigl(e^{-2V(X^\lambda(s))}-\E[e^{-2V(0)}]\bigr)ds-(\eta(\lambda)-\E[e^{-2V(0)}])\right|\\
			&=\frac{\lambda^2}{\rho\xi}\widehat E_0^{\lambda,\omega}|A^\lambda(t)-\eta(\lambda)t|\\
			&\le\frac{\lambda^2}{\rho\xi}\cdot C\lambda^{-1}\sqrt t(\lambda^2 t)^\theta
			=C\rho^{-1/2+\theta}\xi^{-1/2+\theta}\, ,
		\end{aligned}
		\]
		so $A_{1,2}\le C\rho^{-1/2+\theta}$. The bounds $A_2\le Ce^{-c3^h}$ and $A_3\le Ce^{-c3^h}$ follow from the same exit-time calculation as $(b)$ and $(c)$ in Step~4 of the proof of Lemma~\ref{lem.velocity-decomp}. For $A_2$, use $|\eta(\lambda)-\E[e^{-2V(0)}]|\le2\Lambda$; for $A_3$, use $|e^{-2V}-\E[e^{-2V(0)}]|\le 2\Lambda$. The remaining inputs are the identity $\int_r^\infty\xi e^{-\xi}\,d\xi=(r+1)e^{-r}$, the bound $\sup_{u\ge0}ue^{-u/2}<\infty$, and Proposition~\ref{pro.exit} with $\kappa=1/2$.
		Combining with Proposition~\ref{pro.clock-resolvent-infty} yields \eqref{eq.clock-decomp-bound} with deterministic $C$. The threshold
		\[
			\lambda_*^{\rm sc}(\omega)\coloneqq 3^{-(1-\alpha)\X_\mu(\omega)}\wedge\lambda_0(\omega,\theta,\varepsilon)\wedge 1
		\]
		guarantees the hypotheses of Proposition~\ref{pro.clock-resolvent-infty}. The tail \eqref{eq.clock-lambda-star-tail} follows from \eqref{eq.m-mu-tail}, \eqref{eq.lambda0-tail}, and Lemma~\ref{lem.tail-conversion}.
	\end{proof}

	For the final transfer through the time change, we record that \eqref{eq.eta-bound-renewal} and the bounds on $e^{-2V}$ in Subsection~\ref{ss.standing} give $\eta(\lambda),\E[e^{-2V(0)}]\in[\Lambda^{-1},\Lambda]$; also $|\ahom|_{\rm op}\le\Lambda$.

	\begin{proof}[Proof of Theorem~\ref{t.main}]
		Fix $\beta_h\coloneqq 1/8$, so that $\delta=2\beta_h/(d+4)=1/(4(d+4))$, and choose any $\beta^*\in(0,\delta/(2(2-\delta)))$. Fix also $\beta'>0$ such that $\beta^*+\beta'<\delta/(2(2-\delta))$.

		\emph{Step 1.} Velocity rate. Apply Lemma~\ref{lem.velocity-decomp} with parameters $\alpha\in(0,\delta/2)$, $\zeta\in(0,\zeta_H]$, $\theta\in(0,1/4)$, and $\varepsilon\in(0,\alpha/(1-\alpha))$. The first, second, and fourth terms of \eqref{eq.velocity-decomp-bound} translate to powers of $\lambda$ with exponents
		\[
		\begin{aligned}
			\alpha_1&\coloneqq\frac{2\alpha}{1-\alpha}-2\varepsilon\, ,\\
			\alpha_2&\coloneqq\frac{\alpha(1/2-\theta)}{1-\alpha}\, ,\\
			\alpha_3&\coloneqq\frac{\delta-\alpha/2-(d/2+1-\beta_h)\zeta}{1-\alpha}\, .
		\end{aligned}
		\]
		As $\alpha\uparrow\delta/2$ and $(\theta,\varepsilon,\zeta)\downarrow 0$, the limits of $\alpha_1,\alpha_2,\alpha_3$ are $2\delta/(2-\delta)$, $\delta/(2(2-\delta))$, and $3\delta/(2(2-\delta))$, with minimum $\delta/(2(2-\delta))>\beta^*+\beta'$ by the choice of $\beta'$.
		Choose $(\alpha,\zeta,\theta,\varepsilon)$ close enough to this limiting regime that all three exponents exceed $\beta^*+\beta'$. The same parameters also make the clock homogenization exponent in Step~2 exceed $\beta^*+\beta'$, since its limit $11\delta/(12(2-\delta))$ is strictly larger. The third term $e^{-c3^{h(\lambda)}}$ in \eqref{eq.velocity-decomp-bound} is super-polynomial in $\lambda$ since $h(\lambda)\ge\zeta m(\lambda)-1$ and $3^{-m(\lambda)}\asymp\lambda^{1/(1-\alpha)}$ by \eqref{eq.rho-bounds}. Hence Lemma~\ref{lem.velocity-decomp} gives, for $\Q$-a.e.\ $\omega$ and every $0<\lambda\le\lambda_*^H(\omega)$,
		\begin{equation}\label{eq.intermediate-velocity}
			\Bigl|\frac{\ell(\lambda)}{\lambda}-\ahom e_1\Bigr|\le C\lambda^{\beta^*+\beta'}\, ,
		\end{equation}
		with $\lambda_*^H$ a random scale satisfying \eqref{eq.lambda-star-tail}.

		\emph{Step 2.} Clock rate. For the clock bound, the first two algebraic exponents are unchanged. The homogenization exponent is
		\[
			\frac{\delta-\frac{13}{12}\alpha-(d/2+\frac{7}{12}-\frac{\beta_h}{2})\zeta}{1-\alpha}
			\longrightarrow
			\frac{\delta-\frac{13}{24}\delta}{1-\frac12\delta}
			=\frac{11\delta}{12(2-\delta)}
		\]
		as $\alpha\uparrow\delta/2$ and $\zeta\downarrow0$. This limit is strictly larger than $\beta^*+\beta'$, so the minimum is the same as in Step~1.
		Thus
		\begin{equation}\label{eq.intermediate-clock}
			|\eta(\lambda)-\E[e^{-2V(0)}]|\le C\lambda^{\beta^*+\beta'}
		\end{equation}
		for every $0<\lambda\le\lambda_*^{\rm sc}(\omega)$, with $\lambda_*^{\rm sc}$ satisfying \eqref{eq.clock-lambda-star-tail}.

		\emph{Step 3.} Transfer through the time change. The time-change identity
		\[
			\frac{\ell_X(\lambda)}{\lambda}-\Sigma_X e_1=\frac{1}{\eta(\lambda)}\Bigl(\frac{\ell(\lambda)}{\lambda}-\ahom e_1\Bigr)+\bigl(\eta(\lambda)^{-1}-\E[e^{-2V(0)}]^{-1}\bigr)\ahom e_1\, ,
		\]
		together with $\eta(\lambda),\E[e^{-2V(0)}]\in[\Lambda^{-1},\Lambda]$, $|\ahom e_1|\le\Lambda$, and the bounds \eqref{eq.intermediate-velocity}--\eqref{eq.intermediate-clock}, yields
		\[
			\Bigl|\frac{\ell_X(\lambda)}{\lambda}-\Sigma_X e_1\Bigr|\le C_0\lambda^{\beta^*+\beta'}
		\]
		for $\lambda\le\lambda_*^H(\omega)\wedge\lambda_*^{\rm sc}(\omega)$, with deterministic $C_0<\infty$. Setting
		\[
			\lambda_0(\omega)\coloneqq (C_0^{-1/\beta'}\wedge1)\bigl(\lambda_*^H(\omega)\wedge\lambda_*^{\rm sc}(\omega)\bigr)
		\]
		replaces the prefactor $C_0\lambda^{\beta^*+\beta'}$ by $\lambda^{\beta^*}$ on $\lambda\le\lambda_0(\omega)$, yielding~\eqref{eq.intro-rate}. The tail~\eqref{eq.intro-rate-tail} follows from a union bound over the two stretched-exponential tails of $\lambda_*^H$ and $\lambda_*^{\rm sc}$, after adjusting constants to account for the deterministic factor $C_0^{-1/\beta'}\wedge1$.
	\end{proof}

		\section{Renewal identity at regeneration times}\label{sec.shen-renewal-k}

			The goal of this section is to prove Lemma~\ref{lem.conditional-renewal-tau-k}. The clock estimates in Section~\ref{sec.renewal} require a renewal identity at a general regeneration time that includes the environment. The available results do not directly give this: \citet*{ballistic2003shen,ballistic2004shen} prove the first-regeneration decomposition and the half-space measurability that underlie the environment factorization, but only at $\tau_1$; \citet[Section~5]{Einstein2012gantert} adapts the construction to a $\lambda$-dependent scale and establishes the path renewal structure for all $\tau_k$, but does not track the environment beyond the first regeneration time. The clock increment
			\[
				\Delta A_k^\lambda=\int_{\tau_k}^{\tau_{k+1}}e^{-2V(X^\lambda(s))}\,ds
			\]
			depends on the environment after $\tau_k$, so the shifted future data in Lemma~\ref{lem.conditional-renewal-tau-k} must include both the path and the translated environment. The key point is that the regeneration construction forces the path to advance by at least $2R(\lambda)$, where $R(\lambda)$ is the regeneration scale fixed below, in the $e_1$-direction at each regeneration time. This creates a spatial gap between the past and future half-spaces. Since $2R(\lambda)>r_0$, the finite range of dependence makes the environment in these two half-spaces independent under $\Q$, which is the factorization used in the proof.

		\subsection{Setup}\label{ss.shen-renewal-setup}

		Fix $\lambda\in(0,1]$ throughout this section and set
		\[
			\Delta\coloneqq\lambda^{-2},\qquad R\coloneqq R(\lambda)=l/\lambda\, .
		\]
		For $x\in\R^d$, let
		\[
			B^x\coloneqq B_R(x+9Re_1)
		\]
		be the bridge endpoint ball. For $t\in\R$, write
		\[
			\mathsf L(t)\coloneqq\{z\in\R^d:e_1\cdot z\le t\},\qquad
			\mathsf R(t)\coloneqq\{z\in\R^d:e_1\cdot z\ge t\},
		\]
		and set $\mathsf R_+\coloneqq\mathsf R(-2R)$. For a Borel set $S\subseteq\R^d$, $\mathcal H_S$ denotes the $\sigma$-algebra generated by $(\a,e^{-2V})|_S$.

		The canonical time-shift on path-mark space is
		\[
			\theta_t(X^\lambda_\cdot,(Y_n)_{n\ge0})=(X^\lambda_{t+\cdot},(Y_{\lambda^2 t+n})_{n\ge0})\qquad\text{for }t\in\Delta\N\, .
		\]
		Let $p_*\in(0,1)$ be the Bernoulli success probability supplied by \citet[display~(5.3) and Proposition~5.4(i)]{Einstein2012gantert}. The notation $T_y(\a,e^{-2V})$ denotes the deterministic coefficient field $z\mapsto(\a,e^{-2V})(y+z)$. Stationarity is used only after averaging this field over~$\Q$.

		The proof has four steps. Lemma~\ref{lem.regeneration-toolkit} records the regeneration inputs from \citet*{ballistic2003shen,ballistic2004shen} and \citet*[Section~5]{Einstein2012gantert}, together with the scaled consequences used below. Lemma~\ref{lem.tau-k-decomp} extends the deterministic-time decomposition from $\tau_1$ to a general $\tau_k$ by induction. Lemma~\ref{lem.one-slice-renewal} proves the renewal identity on each deterministic time level, and Lemma~\ref{lem.block-compatibility} verifies the past and future measurability assertions for the clock-duration blocks.

			\begin{lemma}\label{lem.regeneration-toolkit}
			The coupled regeneration construction of \citet*{ballistic2003shen,ballistic2004shen} and \citet*[Section~5]{Einstein2012gantert}, as recalled in Subsection~\ref{ss.renewal-setup}, has the following properties.
		\begin{enumerate}
			\item[\rm(a)] The regeneration times depend only on the path and Bernoulli marks. Under $\widehat{\mathbb P}_0^\lambda$ they are almost surely finite, take values in $\Delta\N$, satisfy $\tau_1\ge2\Delta$, and obey
			\[
				\tau_{k+1}=\tau_k+\tau_1\circ\theta_{\tau_k}\, .
			\]
			\item[\rm(b)] For every $n\ge2$, the first-regeneration slice has an event
			\[
				\Gamma_{1,n}\in\sigma(X^\lambda_s:0\le s\le n\Delta-\Delta)\vee\sigma(Y_0,\ldots,Y_{n-2})
			\]
			such that
			\[
				\{\tau_1=n\Delta\}
				=\Gamma_{1,n}\cap\{Y_{n-1}=1\}\cap\{D\circ\theta_{n\Delta}=\infty\}
			\]
				almost surely, and
				\[
					\Gamma_{1,n}\subseteq
					\left\{\sup_{0\le s\le n\Delta-\Delta}e_1\cdot X^\lambda(s)
					\le e_1\cdot X^\lambda(n\Delta-\Delta)+R\right\}.
				\]
					\item[\rm(c)] At the bridge step ending at $n\Delta$, given the path up to $n\Delta-\Delta$ and the marks through $Y_{n-2}$, the mark $\{Y_{n-1}=1\}$ contributes $p_*$, and the endpoint is averaged uniformly over $B^{X^\lambda(n\Delta-\Delta)}$. Given endpoint $y$, the shifted future path and marks have the coupled law $\widehat P_y^{\lambda,\omega}$.
				\item[\rm(d)] Before the terminal bridge the candidate has advanced at least $2R$ in the $e_1$ direction from the previous regeneration time. The successful bridge satisfies
				\begin{equation}\label{eq.final-segment-ball}
					X^\lambda(s)\in B_{6R}\bigl(X^\lambda(n\Delta-\Delta)+5Re_1\bigr),\qquad n\Delta-\Delta\le s\le n\Delta .
				\end{equation}
					\item[\rm(e)] The half-space measurability bounds are as follows. For the past factor, if $H_{\rm pre}$ is bounded and measurable in the path on $[0,n\Delta-\Delta]$ and the marks $Y_0,\ldots,Y_{n-2}$, then
					\[
						\omega\mapsto\widehat E_0^{\lambda,\omega}\!\left[
							H_{\rm pre}\,
						\mathbf 1_{\{\sup_{0\le s\le n\Delta-\Delta}e_1\cdot X^\lambda(s)\le y\cdot e_1-7R\}}
					\right]
				\]
				is $\mathcal H_{\mathsf L(y\cdot e_1-4R)}$-measurable. For the future factor, if $\Phi$ is a product of a bounded functional of the shifted future path and shifted marks with a bounded functional of the shifted environment restricted to $\mathsf R_+$, then
				\[
					\omega\mapsto \widehat E_y^{\lambda,\omega}[\Phi(\mathcal Y_0);D=\infty]
					\]
					is $\mathcal H_{\mathsf R(y\cdot e_1-2R)}$-measurable.
				\item[\rm(f)] The coupled kernel is spatially covariant:
				\[
					\widehat P^{\lambda,(\a,e^{-2V})}_y
					=\bigl((X^\lambda,Y)\mapsto(y+X^\lambda,Y)\bigr)_{\#}
				\widehat P^{\lambda,T_y(\a,e^{-2V})}_0\, .
			\]
		\end{enumerate}
		\end{lemma}

		\begin{proof}
		These properties are consequences of the regeneration construction of \citet*{ballistic2003shen,ballistic2004shen}, in the $\lambda$-dependent form of \citet*[Section~5]{Einstein2012gantert}.

			Item {\rm(a)} follows from the construction \citet*[(5.13)--(5.16)]{Einstein2012gantert}, the recursive definition of the later regeneration times, and \citet*[Lemma~5.7(ii)]{Einstein2012gantert}. Item {\rm(c)} follows from \citet*[Proposition~5.4(i)--(iii)]{Einstein2012gantert}. For item {\rm(d)}, the $2R$ advance follows from the ladder-time definitions \citet*[(5.9)--(5.16)]{Einstein2012gantert}: the candidate bridge starts at least $R$ beyond the previous running maximum, and the successful bridge endpoint lies in $B_R(x+9Re_1)$. The bridge confinement is \citet*[Proposition~5.4(iii)]{Einstein2012gantert}. Item {\rm(b)} is the scaled version of the deterministic-time decomposition in the proof of \citet*[Theorem~2.4]{ballistic2003shen}; in the notation of \citet*{ballistic2003shen}, this is the decomposition preceding display~{\rm(2.20)}, with the scaled definitions \citet*[(5.13)--(5.16)]{Einstein2012gantert}. Item {\rm(e)} is the scaled version of \citet[items (1)--(3)]{ballistic2004shen}, with Shen's unit time step replaced by $\Delta=\lambda^{-2}$ and $R$ by $R(\lambda)$. Item {\rm(f)} follows from translation covariance of the diffusion and bridge kernels in the one-step coupling construction of \citet[Theorem~2.1]{ballistic2003shen} and \citet[Proposition~5.4]{Einstein2012gantert}.
		\end{proof}

		By Lemma~\ref{lem.regeneration-toolkit}(a), $\{\tau_k=n\Delta\}=\emptyset$ for $n\le2k-1$.

		\begin{lemma}[Deterministic-time regeneration decomposition]\label{lem.tau-k-decomp}
		For every $k\ge 1$ and $n\ge 2k$, there exists an event
		\[
			\Gamma_{k,n}\in\sigma\bigl(X^\lambda_s:0\le s\le n\Delta-\Delta\bigr)\vee\sigma\bigl(Y_0,\ldots,Y_{n-2}\bigr)
		\]
		such that the following statements hold $\widehat{\mathbb P}_0^\lambda$-a.s.:
		\begin{equation}\label{eq.tau-k-decomposition}
			\{\tau_k=n\Delta\}=\Gamma_{k,n}\cap\{Y_{n-1}=1\}\cap\{D\circ\theta_{n\Delta}=\infty\}\, ,
		\end{equation}
		\begin{equation}\label{eq.tau-k-pre-segment}
			\Gamma_{k,n}\subseteq
			\left\{\sup_{0\le s\le n\Delta-\Delta} e_1\cdot X^\lambda(s)
			\le e_1\cdot X^\lambda(n\Delta-\Delta)+R\right\}\, .
		\end{equation}
		Consequently, also $\widehat{\mathbb P}_0^\lambda$-a.s., for every $y\in\R^d$,
		\begin{equation}\label{eq.tau-k-candidate-confinement}
			\Gamma_{k,n}\cap\{y\in B^{X^\lambda(n\Delta-\Delta)}\}\subseteq\left\{\sup_{0\le s\le n\Delta-\Delta}e_1\cdot X^\lambda(s)\le e_1\cdot y-7R\right\}\, ,
		\end{equation}
		and
		\begin{equation}\label{eq.tau-k-confinement}
			\Gamma_{k,n}\cap\{Y_{n-1}=1\}\subseteq\Bigl\{\sup_{0\le s\le n\Delta-\Delta}e_1\cdot X^\lambda(s)\le e_1\cdot X^\lambda(n\Delta)-7R\Bigr\}\, .
		\end{equation}
	\end{lemma}

			\begin{proof}
				\emph{Step 1.} The case $k=1$ is Lemma~\ref{lem.regeneration-toolkit}(b).

			\emph{Step 2.} The induction step. Assume $k\ge 2$ and $n\ge 2k$. The regeneration recursion gives
		\begin{equation}\label{eq.tau-k-slice-decomp}
			\{\tau_k=n\Delta\}=\bigsqcup_{2(k-1)\le j\le n-2}\{\tau_{k-1}=j\Delta\}\cap\{\tau_1\circ\theta_{j\Delta}=(n-j)\Delta\}\, .
		\end{equation}
			Applying the induction hypothesis at $k-1$ to the first factor and the shifted $k=1$ case to the second, the natural factorization contains the intermediate no-backtracking event $\{D\circ\theta_{j\Delta}=\infty\}$. This event depends on the path after $n\Delta-\Delta$, so we replace it temporarily by its finite-time consequence
			\[
				D^{j,n}\coloneqq\Bigl\{e_1\cdot X^\lambda(s)>e_1\cdot X^\lambda(j\Delta)-R\;\text{for all }s\in[j\Delta,n\Delta-\Delta]\Bigr\}\, .
			\]
			Define
			\[
				\Gamma_{k,n}^{(j)}\coloneqq\Gamma_{k-1,j}\cap\{Y_{j-1}=1\}\cap(\Gamma_{1,n-j}\circ\theta_{j\Delta})\cap D^{j,n},\qquad \Gamma_{k,n}\coloneqq\bigcup_j\Gamma_{k,n}^{(j)}\, .
			\]
			Each factor in $\Gamma_{k,n}^{(j)}$ depends only on the path on $[0,n\Delta-\Delta]$ and the marks $Y_0,\ldots,Y_{n-2}$, so $\Gamma_{k,n}$ has the required measurability.

			\emph{Step 3.} Recovery of $\{D\circ\theta_{j\Delta}=\infty\}$. A successful bridge ending at $m\Delta$ satisfies
			\[
				e_1\cdot X^\lambda(m\Delta-\Delta)-R
			\le e_1\cdot X^\lambda(s)
			\le e_1\cdot X^\lambda(m\Delta-\Delta)+11R,\qquad m\Delta-\Delta\le s\le m\Delta,
		\]
		by \eqref{eq.final-segment-ball}. It remains to show that on the slice
		\[
			\Gamma_{k,n}^{(j)}\cap\{Y_{n-1}=1\}\cap\{D\circ\theta_{n\Delta}=\infty\}\, ,
		\]
		the full no-backtracking event from $j\Delta$ follows by controlling three intervals. First, $D^{j,n}$ gives
		\[
			e_1\cdot X^\lambda(s)>e_1\cdot X^\lambda(j\Delta)-R,\qquad j\Delta\le s\le n\Delta-\Delta.
		\]
		The shifted first-regeneration construction also gives the ladder advance
		\[
			e_1\cdot X^\lambda(n\Delta-\Delta)\ge e_1\cdot X^\lambda(j\Delta)+2R\, .
		\]
		The bridge bound with $m=n$ then gives, for $s\in[n\Delta-\Delta,n\Delta]$,
		\[
			e_1\cdot X^\lambda(s)\ge e_1\cdot X^\lambda(n\Delta-\Delta)-R\ge e_1\cdot X^\lambda(j\Delta)+R>e_1\cdot X^\lambda(j\Delta)-R\, .
		\]
		Finally, on $[n\Delta,\infty)$, the endpoint condition $X^\lambda(n\Delta)\in B^{X^\lambda(n\Delta-\Delta)}$ together with $\{D\circ\theta_{n\Delta}=\infty\}$ gives
		\[
			e_1\cdot X^\lambda(s)\ge e_1\cdot X^\lambda(n\Delta)-R\ge e_1\cdot X^\lambda(n\Delta-\Delta)+7R\ge e_1\cdot X^\lambda(j\Delta)+9R>e_1\cdot X^\lambda(j\Delta)-R\, .
		\]
			Hence $\{D\circ\theta_{j\Delta}=\infty\}$ holds on the slice $\Gamma_{k,n}^{(j)}\cap\{Y_{n-1}=1\}\cap\{D\circ\theta_{n\Delta}=\infty\}$. Together with the implication from $\{D\circ\theta_{j\Delta}=\infty\}$ to $D^{j,n}$ used in Step~2, this gives the decomposition \eqref{eq.tau-k-decomposition}.

				\emph{Step 4.} Bounds \eqref{eq.tau-k-pre-segment}--\eqref{eq.tau-k-confinement}. We first propagate \eqref{eq.tau-k-pre-segment} through the completed blocks in a fixed sub-slice of the induction. Write $r\Delta$ for a completed regeneration time and $q\Delta-\Delta$ for the pre-bridge time of the next successful candidate. The endpoint condition and the next ladder advance give
			\[
				e_1\cdot X^\lambda(r\Delta)\ge e_1\cdot X^\lambda(r\Delta-\Delta)+8R,
			\]
			\[
				e_1\cdot X^\lambda(q\Delta-\Delta)\ge e_1\cdot X^\lambda(r\Delta)+2R.
			\]
			Consequently,
			\[
				e_1\cdot X^\lambda(q\Delta-\Delta)\ge e_1\cdot X^\lambda(r\Delta-\Delta)+10R.
			\]
			The terminal bridge ending at $r\Delta$ is controlled by
			\[
				\sup_{r\Delta-\Delta\le s\le r\Delta} e_1\cdot X^\lambda(s)
				\le e_1\cdot X^\lambda(r\Delta-\Delta)+11R
				\le e_1\cdot X^\lambda(q\Delta-\Delta)+R.
			\]
			The shifted $k=1$ pre-segment bound inside the next block gives
			\[
				\sup_{r\Delta\le s\le q\Delta-\Delta}e_1\cdot X^\lambda(s)
				\le e_1\cdot X^\lambda(q\Delta-\Delta)+R.
			\]
			If the bound holds up to $r\Delta-\Delta$, the inequality above and the previous bridge estimate make it hold up to $q\Delta-\Delta$. Iterating over the finitely many completed blocks in the sub-slice proves \eqref{eq.tau-k-pre-segment}.

			For \eqref{eq.tau-k-candidate-confinement}, every $y\in B^{X^\lambda(n\Delta-\Delta)}=B_R(X^\lambda(n\Delta-\Delta)+9Re_1)$ satisfies
			\[
				e_1\cdot y\ge e_1\cdot X^\lambda(n\Delta-\Delta)+8R.
			\]
			Combining this with \eqref{eq.tau-k-pre-segment} yields \eqref{eq.tau-k-candidate-confinement}.

			On $\{Y_{n-1}=1\}$, the endpoint condition $X^\lambda(n\Delta)\in B^{X^\lambda(n\Delta-\Delta)}$ holds, so \eqref{eq.tau-k-confinement} follows by taking $y=X^\lambda(n\Delta)$ in \eqref{eq.tau-k-candidate-confinement}.
	\end{proof}

		\begin{lemma}[Renewal identity on a deterministic slice]\label{lem.one-slice-renewal}
		Fix $k\ge1$ and $n\ge2k$. Let $H_{\rm past}$ be a bounded random variable measurable with respect to
		\[
			\sigma(X^\lambda_s:0\le s\le n\Delta-\Delta)\vee\sigma(Y_0,\ldots,Y_{n-2}),
		\]
			let $G$ be a Borel subset of $\R^d$, and let $F\colon\R^d\times\Omega\to\R$ be bounded measurable with $F(y,\cdot)$ $\mathcal H_{\mathsf L(y\cdot e_1-4R)}$-measurable for each $y$. Define
				\begin{equation}\label{eq.product-test}
					H=\mathbf 1_{\{\tau_k=n\Delta\}} H_{\rm past} \mathbf 1_{\{X^\lambda(n\Delta)\in G\}} F(X^\lambda(n\Delta),\omega),\qquad n\ge 2k\, ,
				\end{equation}
			and let $\Phi$ be a bounded function of the shifted future data that has the following form: for bounded measurable $f$ of the shifted path and marks and bounded measurable $g$ of the shifted environment restricted to $\mathsf R_+$,
			\[
				\Phi(\gamma,\beta,\omega_+)=f(\gamma,\beta)g(\omega_+)\, .
			\]
			Then
			\[
				\widehat{\mathbb E}_0^\lambda[\Phi(\mathcal Y_k)H]
				=\widehat{\mathbb E}_0^\lambda[\Phi(\mathcal Y_0)\mid D=\infty]\,
			\widehat{\mathbb E}_0^\lambda[H]\, .
		\]
		\end{lemma}

			\begin{proof}
				Since $\Gamma_{k,n}$ contains only finite-time past data, the terminal mark, bridge endpoint, and no-backtracking event can be averaged as in the proof of \citet*[Theorem~2.4]{ballistic2003shen}, around the endpoint integration leading to display~(2.21).

						On the slice $\{\tau_k=n\Delta\}$, the trace of $\mathcal G_k$ is generated by the path on $[0,n\Delta-\Delta]$, the marks $Y_0,\ldots,Y_{n-2}$, the endpoint $X^\lambda(n\Delta)$, and the environment in $\mathsf L(e_1\cdot X^\lambda(n\Delta)-4R)$. As $H_{\rm past}$, $G$, and $F$ range over their generating classes, the variables $H$ in \eqref{eq.product-test} generate this trace, with $F(y,\cdot)$ supplying the environment generators.

				\emph{Step 1.} Slice formula. With this form of $\Phi$,
				\[
					\Phi(\mathcal Y_k)=f\bigl(X^\lambda(\tau_k+\cdot)-X^\lambda(\tau_k),(Y_{\lambda^2\tau_k+j})_{j\ge0}\bigr) g\bigl(T_{X^\lambda(\tau_k)}(\a,e^{-2V})|_{\mathsf R_+}\bigr)\, ,
				\]
			and $H$ as in \eqref{eq.product-test}. The decomposition \eqref{eq.tau-k-decomposition} from Lemma~\ref{lem.tau-k-decomp} restricts the integration to the slice
		\[
			\{\tau_k=n\Delta\}=\Gamma_{k,n}\cap\{Y_{n-1}=1\}\cap\{D\circ\theta_{n\Delta}=\infty\}\, .
		\]
					By Lemma~\ref{lem.regeneration-toolkit}(c), the mark $\{Y_{n-1}=1\}$ contributes $p_*$, the endpoint is averaged uniformly over $B^{X^\lambda(n\Delta-\Delta)}$, and the shifted future path and marks restart from $y$. Thus
	\begin{equation}\label{eq.slice-factorization}
		\widehat{\mathbb E}_0^\lambda\bigl[\Phi(\mathcal Y_k) H\bigr]
		= p_* \widehat{\mathbb E}_0^\lambda \left[\mathbf{1}_{\Gamma_{k,n}} H_{\rm past}  \fint_{B^{X^\lambda(n\Delta-\Delta)}}  \mathbf{1}_{\{y\in G\}} F(y,\omega) \widehat E_y^{\lambda,\omega}\bigl[\Phi(\mathcal Y_0); D=\infty\bigr] dy\right]\, .
	\end{equation}
				To factor the expectation using finite-range independence, first take the quenched expectation over the path and marks before $n\Delta$. For fixed $y\in\R^d$ define
		\[
				A^{\rm pre}_{k,n}(y,\omega)
			\coloneqq
			\widehat E_0^{\lambda,\omega}\!\left[
				\mathbf 1_{\Gamma_{k,n}}H_{\rm past}
				\mathbf 1_{\{y\in B^{X^\lambda(n\Delta-\Delta)}\}}
		\right],
	\]
		and set
		\[
				A^{\rm past}_{k,n}(y,\omega)
				\coloneqq \mathbf 1_{\{y\in G\}}F(y,\omega)A^{\rm pre}_{k,n}(y,\omega),
				\qquad
				A^{\rm future}_{\Phi}(y,\omega)
				\coloneqq \widehat E_y^{\lambda,\omega}\bigl[\Phi(\mathcal Y_0);D=\infty\bigr] .
		\]
							The factors $A^{\rm past}_{k,n}$ and $A^{\rm future}_{\Phi}$ are jointly Borel measurable in $(y,\omega)$. When $H_{\rm past}$ is an indicator function, measurability of the past factor follows from the finite-time construction of $\Gamma_{k,n}$ and the measurable coupled kernels; bounded $H_{\rm past}$ then follow by the monotone-class theorem. The future factor is measurable because $\{D=\infty\}$ is a monotone limit of finite-time events.
		By Fubini's theorem, \eqref{eq.slice-factorization} is equivalently
		\begin{equation}\label{eq.slice-A1-A2}
			\widehat{\mathbb E}_0^\lambda\bigl[\Phi(\mathcal Y_k) H\bigr]
			=\frac{p_*}{|B_R|}\int_{\R^d}
				\E_\Q\!\left[A^{\rm past}_{k,n}(y,\cdot)A^{\rm future}_{\Phi}(y,\cdot)\right]\,dy .
		\end{equation}

					\emph{Step 2.} Half-space factorization. The event $\Gamma_{k,n}\cap\{y\in B^{X^\lambda(n\Delta-\Delta)}\}$ appearing in $A^{\rm pre}_{k,n}(y,\omega)$ is supported, by \eqref{eq.tau-k-candidate-confinement}, on
		\[
			\left\{\sup_{0\le s\le n\Delta-\Delta}e_1\cdot X^\lambda(s)\le y\cdot e_1-7R\right\}.
		\]
						Lemma~\ref{lem.regeneration-toolkit}(e) therefore makes $A^{\rm pre}_{k,n}(y,\cdot)$ $\mathcal H_{\mathsf L(y\cdot e_1-4R)}$-measurable, first for indicator $H_{\rm past}$ and then, by monotone class, for bounded $H_{\rm past}$. Multiplying by $F(y,\cdot)$ preserves this measurability, so $A^{\rm past}_{k,n}(y,\cdot)$ is also $\mathcal H_{\mathsf L(y\cdot e_1-4R)}$-measurable. The future factor $A^{\rm future}_{\Phi}(y,\cdot)$ is $\mathcal H_{\mathsf R(y\cdot e_1-2R)}$-measurable by the future half-space statement in the same lemma.

		Since the half-spaces $\mathsf L(y\cdot e_1-4R)$ and $\mathsf R(y\cdot e_1-2R)$ are separated by distance $2R\ge 2l>r_0$, Assumption~\ref{a.environment} gives, for every fixed $y$,
	\[
				\E_\Q\!\left[A^{\rm past}_{k,n}(y,\cdot)A^{\rm future}_{\Phi}(y,\cdot)\right]
				=\E_\Q\!\left[A^{\rm past}_{k,n}(y,\cdot)\right]\,
				\E_\Q\!\left[A^{\rm future}_{\Phi}(y,\cdot)\right].
	\]

				\emph{Step 3.} Conclusion on the slice. Lemma~\ref{lem.regeneration-toolkit}(f) gives
	\[
		\widehat E_y^{\lambda,\omega}\bigl[\Phi(\mathcal Y_0); D=\infty\bigr]=\widehat E_0^{\lambda,T_y(\a,e^{-2V})}\bigl[\Phi(\mathcal Y_0); D=\infty\bigr]
	\]
	for every $y\in\R^d$. Combined with the spatial stationarity of the law of $(\a,e^{-2V})$ under $\Q$, this makes the $\E_\Q$-mean of $A^{\rm future}_{\Phi}(y,\cdot)$ independent of $y$ and equal to $\widehat{\mathbb E}_0^\lambda[\Phi(\mathcal Y_0);D=\infty]$. Applying \eqref{eq.slice-A1-A2} with $\Phi\equiv1$ gives
	\[
		\widehat{\mathbb E}_0^\lambda[H]
		=\frac{p_*}{|B_R|}\widehat{\mathbb P}_0^\lambda(D=\infty)
			\int_{\R^d}\E_\Q\!\left[A^{\rm past}_{k,n}(y,\cdot)\right]\,dy .
	\]
	Since $\widehat{\mathbb P}_0^\lambda(D=\infty)>0$ by \eqref{eq.regen-scale-choice}, the formula for $\widehat{\mathbb E}_0^\lambda[H]$ above and \eqref{eq.slice-A1-A2} give
		\[
		\widehat{\mathbb E}_0^\lambda\bigl[\Phi(\mathcal Y_k) H\bigr]=\widehat{\mathbb E}_0^\lambda[\Phi(\mathcal Y_0)\mid D=\infty]\cdot\widehat{\mathbb E}_0^\lambda[H]\, .
	\]
			\end{proof}

		To extend the identity proved on each slice $\{\tau_k=n\Delta\}$ to the clock-duration blocks $M_i$, past blocks must be $\mathcal G_k$-measurable and future blocks must be determined by $\mathcal Y_k$.

		\begin{lemma}[Past and future measurability of clock blocks]\label{lem.block-compatibility}
		For every $k\ge1$, the following hold.
		\begin{enumerate}
			\item[\rm(i)] If $1\le i\le k-2$, then $M_i$ is $\mathcal G_k$-measurable.
			\item[\rm(ii)] There is a measurable map $\Psi$ on the common state space of the shifted future data such that
			\[
				(M_k,M_{k+1},\ldots)=\Psi(\mathcal Y_k)
			\]
			$\widehat{\mathbb P}_0^\lambda$-almost surely.
		\end{enumerate}
		\end{lemma}

		\begin{proof}
					\emph{Step 1.} We prove {\rm(i)}. Fix an integer $i$ with $1\le i\le k-2$ and work on the slice $\{\tau_k=n\Delta\}$ with $n\ge2k$.

			Unfold the recursive construction of $\Gamma_{k,n}$ in Lemma~\ref{lem.tau-k-decomp}, starting from the terminal slice $\tau_k=n\Delta$ and continuing down to $\tau_1$. This gives a finite disjoint family of finite-time cells $E_{\mathbf n}$ indexed by tuples
			\[
					\mathbf n=(n_0,\ldots,n_k),\qquad n_0=0,\quad n_k=n,\quad n_j-n_{j-1}\ge2\quad(1\le j\le k)\, .
				\]
				After intersecting with $\{Y_{n-1}=1\}\cap\{D\circ\theta_{n\Delta}=\infty\}$, these cells partition $\{\tau_k=n\Delta\}$ up to null sets, and on each cell $\tau_j=n_j\Delta$ for every $0\le j\le k$. Every intermediate no-backtracking condition has already been replaced by a finite-time condition before $n\Delta-\Delta$, so each $E_{\mathbf n}$ is measurable with respect to
				\[
					\sigma(X^\lambda_s:0\le s\le n\Delta-\Delta)\vee\sigma(Y_0,\ldots,Y_{n-2}).
				\]

			By \eqref{eq.lattice-tau}, $\tau_{k-1}\le\tau_k-\Delta=n\Delta-\Delta$ on the slice. Since $i+1\le k-1$, also $\tau_{i+1}\le n\Delta-\Delta$, so $\Delta\tau_i$ is determined on each cell $E_{\mathbf n}$. For the clock increment, \eqref{eq.tau-k-confinement} places the path on $[\tau_i,\tau_{i+1}]$ inside $\mathsf L_k$. Hence
		\[
			\Delta A_i^\lambda=\int_{\tau_i}^{\tau_{i+1}}e^{-2V(X^\lambda(s))}\,ds
		\]
	is a measurable function of the path on $[\tau_i,\tau_{i+1}]$ and the restricted field $e^{-2V}|_{\mathsf L_k}$, that is, of the data in items {\rm(a)} and {\rm(c)} of Lemma~\ref{lem.conditional-renewal-tau-k}. Since $\tau_k$ is recorded in $\mathcal G_k$, the slices are $\mathcal G_k$-measurable; assembling these slicewise representatives over the countably many slices gives a $\mathcal G_k$-measurable representative of $M_i$. This proves~{\rm(i)}.

			\emph{Step 2.} We prove {\rm(ii)}. We work on the $\widehat{\mathbb P}_0^\lambda$-full event $\{\tau_j<\infty\text{ for every }j\ge1\}$. On the slice $\{\tau_k=n\Delta\}$, the factor $\{D\circ\theta_{n\Delta}=\infty\}$ in \eqref{eq.tau-k-decomposition} gives
	\[
		e_1\cdot(X^\lambda(\tau_k+s)-X^\lambda(\tau_k))\ge -R\qquad\text{for every }s\ge0\, .
	\]
			Thus the shifted post-$\tau_k$ path lies in $\mathsf R_+$. Every future clock integral uses only the restricted environment in $\mathsf R_+$, and the future bridge balls also stay in $\mathsf R_+$: their starts have shifted $e_1$-coordinate at least $-R$, while Lemma~\ref{lem.regeneration-toolkit}(d) centers each radius-$6R$ bridge ball another $5R$ forward.

					Let $\mathfrak Y$ be the common state space of triples
		\[
			(\gamma,(\beta_n)_{n\ge0},\omega_+)
			\in C(\R_+;\R^d)\times\{0,1\}^{\N}\times \Omega_{\mathsf R_+},
		\]
		with its product Borel $\sigma$-field, where $\Omega_{\mathsf R_+}$ is the restriction space for the field on $\mathsf R_+$. On a triple in $\mathfrak Y$, run the deterministic regeneration recursion using $\gamma$ and $\beta$. The hitting times, rounded hitting times, candidate times, and Bernoulli tests are Borel functionals of $(\gamma,\beta)$.

			Let $s_0=0<s_1<s_2<\cdots$ be the resulting regeneration times when they are all finite. Since the fields are uniformly Lipschitz, the evaluation map $(\omega_+,z)\mapsto e^{-2V_+}(z)$ is Borel. On the Borel event that all $s_j$ are finite and $\gamma([s_j,s_{j+1}])\subseteq\mathsf R_+$ for every $j\ge0$, set
		\[
			\Psi_j(\gamma,\beta,\omega_+)
			\coloneqq\left(\int_{s_j}^{s_{j+1}}e^{-2V_+(\gamma(u))}\,du,\ s_{j+1}-s_j\right),
			\qquad j\ge0,
		\]
		where $e^{-2V_+}$ denotes the clock integrand in $\omega_+$. Off this event, assign a fixed deterministic default sequence. The map $(\gamma,\omega_+,u)\mapsto e^{-2V_+(\gamma(u))}$ is jointly Borel, so the time integrals are measurable by Fubini's theorem applied to the jointly Borel integrand, after truncating on $\{s_{j+1}\le m\}$ and letting $m\uparrow\infty$. This gives a measurable map
		\[
			\Psi\colon\mathfrak Y\to([0,\infty)\times[\Delta,\infty))^{\N}.
		\]
				On the actual shifted future data $\mathcal Y_k$, no-backtracking keeps the shifted path in $\mathsf R_+$ for all future times. The path--mark recursion read from $\mathcal Y_k$ is, by translation covariance, the regeneration recursion for the blocks after $\tau_k$. Each future clock integral uses only the displayed environment restriction, so $(M_k,M_{k+1},\ldots)=\Psi(\mathcal Y_k)$ almost surely. This proves~{\rm(ii)}.
		\end{proof}

		Combining these lemmas via a monotone-class argument gives Lemma~\ref{lem.conditional-renewal-tau-k}.

		\begin{proof}[Proof of Lemma~\ref{lem.conditional-renewal-tau-k}]
			Lemma~\ref{lem.one-slice-renewal} proves \eqref{eq.shen-renewal} on a fixed deterministic slice for functions $\Phi$ that split into a path-mark factor and an environment factor. The slices $\{\tau_k=n\Delta\}$ are countable, disjoint, and exhaustive up to a null set by Lemma~\ref{lem.regeneration-toolkit}(a). On each slice, the variables $H$ in \eqref{eq.product-test}, with $H_{\rm past}$, $G$, and $F$ ranging over their generating classes, generate the trace of $\mathcal G_k$.

		A functional monotone-class argument therefore extends \eqref{eq.shen-renewal} first to all bounded $\mathcal G_k$-measurable $H$ supported on one slice, then by summing over slices to all bounded $\mathcal G_k$-measurable $H$. The same argument extends $\Phi$ from functions of this separated form to all bounded measurable functionals of $\mathcal Y_k$. The measurability assertions are exactly Lemma~\ref{lem.block-compatibility}.
		\end{proof}

		\bibliographystyle{plainnat}
	\bibliography{Ref}

\end{document}